\documentclass[reqno]{amsart}
\usepackage{amsmath}
\usepackage{amsxtra}
\usepackage{amscd}
\usepackage{amsthm}
\usepackage{amsfonts}
\usepackage{amssymb}
\usepackage{enumitem}
\usepackage{microtype}
\usepackage{graphicx}
\usepackage{epstopdf}
\usepackage{mathrsfs}
\usepackage{eucal} 
\usepackage{comment}
\usepackage[hidelinks, pagebackref, bookmarks=false]{hyperref}
\hypersetup{
    colorlinks,
    citecolor=black,
    filecolor=black,
    linkcolor=blue,
    urlcolor=black,
}
\usepackage{cite}
\usepackage{microtype}
\usepackage{url}

\usepackage{thmtools} 
\usepackage{thm-restate}

\input xy
\xyoption{all}

\numberwithin{equation}{subsection} 
\numberwithin{figure}{subsection} 

\makeatletter
\let\c@equation\c@figure
\makeatother

\newtheorem{theorem}[equation]{Theorem}
\newtheorem{corollary}[equation]{Corollary}
\newtheorem{lemma}[equation]{Lemma}
\newtheorem{prop}[equation]{Proposition}
\newtheorem{propconstr}[equation]{Proposition-Construction}

\theoremstyle{remark}
\newtheorem{remark}[equation]{Remark}
\newtheorem{example}[equation]{Example}
\newtheorem{question}[equation]{Question}

\theoremstyle{definition}
\newtheorem{defn}[equation]{Definition}

\newcommand\nc{\newcommand}
\nc\on{\operatorname}

\newcommand\fp{{\mathfrak p}}
\newcommand\fq{{\mathfrak q}}
\newcommand\fm{{\mathfrak m}}

\nc\Hom{\on{Hom}}
\nc\Sections{\on{Sections}}
\nc\Sym{\on{Sym}}
\nc\Spec{\on{Spec}}
\nc\Specm{\on{Specm}}
\nc\ul{\underline}
\nc{\dfp}{\overset{\cdot}{\fp}}
\nc{\dfq}{\overset{\cdot}{\fq}}
\nc{\dfm}{\overset{\cdot}{\fm}}

\begin{document}

\title{Arithmetic restrictions on geometric monodromy}
\author{Daniel Litt}

\begin{abstract}
Let $X$ be a normal complex algebraic variety, and $p$ a prime.  We show that there exists an integer $N=N(X, p)$ such that: any non-trivial, irreducible representation of the fundamental group of $X$, which arises from geometry, must be non-trivial mod $p^N$.  

The proof involves an analysis of the action of the Galois group of a finitely generated field on the \'etale fundamental group of $X$.  We also prove many arithmetic statements about fundamental groups which are of independent interest, and give several applications.
\end{abstract}
\maketitle

\setcounter{tocdepth}{1}
\tableofcontents

\section{Introduction}
Let $X$ be an algebraic variety over $\mathbb{C}$, and let $\Lambda$ be a commutative ring.  The goal of this paper is to study the possible representations $$\pi_1(X)\to GL_n(\Lambda)$$ which arise from geometry, i.e.~those that arise as monodromy representations associated to locally constant subquotients of $$R^i\pi^{\text{an}}_*\underline{\Lambda},$$ where $\pi: Y\to X$ is a morphism of algebraic varieties and $\underline{\Lambda}$ is the constant sheaf on $Y^{\text{an}}$.  For example, if $\pi: Y\to X$ is an Abelian scheme and $\Lambda=\mathbb{Z}_\ell$, we are asking which monodromy representations can arise from $\ell$-adic Tate modules of Abelian schemes over $X$.

\subsection{Statement of Main Results: Applications}
Our main result in this direction is: 
\begin{theorem}\label{main-theorem-intro}
Let $X$ be a connected normal variety over $\mathbb{C}$, and $\ell$ a prime.  Then there exists an integer $$N=N(X,\ell)$$ such that any representation $$\pi_1(X)\to GL_n(\mathbb{Z})$$ which arises from geometry, and is trivial modulo $\ell^N$, is unipotent.
\end{theorem}
See Corollary \ref{geometric-monodromy-corollary} and the surrounding remarks for more details.

As a corollary, we obtain the statement in the abstract, as an irreducible unipotent representation must be trivial.

Though this theorem is about complex algebraic varieties, the proof is fundamentally \emph{arithmetic} --- in fact, Theorem \ref{main-theorem-intro} is a consequence of a statement about the \'etale fundamental groups of varieties over finitely generated fields.

Recall that if $$\rho: \pi_1^{\text{\'et}}(X)\to GL_n(\mathbb{Z}_\ell)$$ is an $\ell$-adic representation of the \'etale fundamental group of $X$ which arises from the cohomology of a family of varieties over $X$, then there exists a finitely-generated subfield $k\subset \mathbb{C}$ and a variety $\mathcal{X}/k$ such that
\begin{enumerate}
\item $\mathcal{X}_{\mathbb{C}}\simeq X$, and
\item There exists a representation $$\rho': \pi_1^{\text{\'et}}(\mathcal{X})\to GL_n(\mathbb{Z}_\ell)$$ such that $\rho$ is naturally identified with the restriction of $\rho'$ to the geometric fundamental group of $\mathcal{X}$.
\end{enumerate}
It turns out that this \emph{arithmetic} nature of geometric monodromy representations places serious restrictions on which representations are possible.  An  antecedent to this method of analyzing monodromy representations is Grothendieck's proof of the quasi-unipotent monodromy theorem \cite[Appendix]{serre-tate}. 

Let $k$ be a finitely generated field and $X$ a variety over $k$.  We say that a continuous $\ell$-adic representation $\rho$ of $\pi_1^{\text{\'et}}(X_{\bar k})$ is \emph{arithmetic} if there exists a finite extension $k'$ of $k$ and a representation $\tilde\rho$ of $\pi_1^{\text{\'et}}(\mathcal{X}_{k'})$ such that $\rho$ arises as a subquotient of $\tilde\rho|_{\pi_1^{\text{\'et}}(\mathcal{X}_{\bar k})}$.  Theorem \ref{main-theorem-intro} follows by a standard specialization argument from the main number-theoretic result of this paper:
\begin{theorem}\label{main-theorem-intro-2}
Let $k$ be a finitely generated field of characteristic zero and $X/k$ a geometrically connected, normal variety.  Let $\ell$ be a prime.  Then there exists an integer $$N=N(X,\ell)$$ such that any continuous arithmetic representation $$\pi_1^{\text{\'et}}(X_{\bar k})\to GL_n(\mathbb{Z}_\ell)$$ which is trivial mod $\ell^N$ is unipotent.
\end{theorem}
See Theorem \ref{main-arithmetic-result} and the surrounding remarks for more details.

Note that in both Theorems \ref{main-theorem-intro} and \ref{main-theorem-intro-2}, $N$ is independent of $n$, the dimension of the representation in question.  In many cases, (for example, if $X$ is a curve), $N$ depends in an explicit way on the image of the representation $$\on{Gal}(\bar k/k)\to GL(H^1(X_{\bar k}, \mathbb{Z}_\ell)).$$   For example, if $X=\mathbb{P}^1\setminus\{x_1, \cdots, x_n\}$, we have
\begin{theorem}\label{main-theorem-P1}
Let $k$ be a field with prime subfield $k_0$, and let $$X=\mathbb{P}^1_{\bar k}\setminus\{x_1, \cdots, x_n\}.$$ Let $K$ be the field generated by cross-ratios of the $x_i$, that is, $$K=k_0\left(\frac{x_a-x_b}{x_c-x_d}\right)_{1\leq a<b<c<d\leq n}.$$   Let $\ell$ be a prime different from the characteristic of $k$ and $q\in \mathbb{Z}_\ell^\times$ be any element of the image of the cyclotomic character $$\chi: \on{Gal}(\overline{K}/K)\to \mathbb{Z}_\ell^\times;$$ let $s$ be the order of $q$ in $\mathbb{F}_\ell^\times$ if $\ell\not=2$ and in $(\mathbb{Z}/4\mathbb{Z})^\times$ if $\ell=2$.  Let $\epsilon=1$ if $\ell=2$ and $0$ otherwise.  For any basepoint $x$ of $X$, let $$\rho: \pi_1^{\text{\'et}}(X, x)\to GL_m(\mathbb{Z}_\ell)$$ be a continuous representation which is trivial mod $\ell^k$, with $$k>\frac{1}{s}\left(v_\ell(q^s-1)+\frac{1}{\ell-1}+\epsilon\right).$$  Then if $\rho$ arises from geometry, it is unipotent.  
\end{theorem}
See Theorem \ref{p1-main-theorem} and the surrounding remarks for this and related results.

Let us unpack this a bit.  First of all, observe that if the cyclotomic character $$\chi: \on{Gal}(\overline{K}/K)\to \mathbb{Z}_\ell^\times$$ is surjective and $\ell>2$, then the bound in question becomes $$k>\frac{1}{\ell-1}+\frac{1}{(\ell-1)^2}.$$  That is, we may take $k=1$.  In other words, we have the following corollary.
\begin{corollary}\label{p1-cor-intro}
As before, let $$X=\mathbb{P}^1_{\mathbb{C}}\setminus\{x_1, \cdots, x_n\}$$ and let $K$ be the field generated by cross-ratios of the $x_i$, that is, $$K=\mathbb{Q}\left(\frac{x_a-x_b}{x_c-x_d}\right)_{1\leq a<b<c<d\leq n}.$$   Let $\ell$ be an odd prime and suppose the cyclotomic character $$\chi: \on{Gal}(\bar K/K)\to \mathbb{Z}_\ell^\times$$ is surjective.  For any basepoint $x$ of $X$, let $$\rho: \pi_1^{\text{\'et}}(X, x)\to GL_m(\mathbb{Z}_\ell)$$ be a representation which is trivial mod $\ell$.  Then if $\rho$ arises from geometry, it is unipotent. 
\end{corollary}
In particular, the hypothesis is satisfied for all odd $\ell$ if $\frac{x_a-x_b}{x_c-x_d}\in \mathbb{Q}$ for all $a<b<c<d$.  Thus we have the following corollary:
\begin{corollary}\label{p1-cor-intro-2}
Let $$X=\mathbb{P}^1_{\mathbb{C}}\setminus\{x_1, \cdots, x_n\}$$ and suppose that $$\left(\frac{x_a-x_b}{x_c-x_d}\right)\in \mathbb{Q}$$ for all $a,b,c,d$.   Then any representation $$\rho: \pi_1(X(\mathbb{C})^{\text{an}})\to GL_m(\mathbb{Z})$$ which arises from geometry and is trivial modulo $N$, for some integer $N\not\in \{1,2,4\}$, is unipotent. 
\end{corollary}
The hypothesis on the cyclotomic character in Corollaries \ref{p1-cor-intro}, \ref{p1-cor-intro-2} may seem strange, but they are in fact necessary.
\begin{example}
Let $\ell=3$ or $5$, and consider the (connected) modular curve $Y(\ell)$, parametrizing elliptic curves with full level $\ell$ structure.  $Y(\ell)$ has genus zero.  Let $E\to Y(\ell)$ be the universal family, and $\bar x$ any geometric point of $Y(\ell)$.  Then the tautological representation $$\pi_1^{\text{\'et}}(Y(\ell), \bar x)\to GL(T_\ell(E_{\bar x}))$$ is trivial mod $\ell$.  This does not contradict Corollary \ref{p1-cor-intro} because the field generated by the cross-ratios of the cusps of $Y(\ell)$ is $\mathbb{Q}(\zeta_\ell)$, whose cyclotomic character is not surjective at $\ell$.
\end{example}
We may use these results to construct example of representations of fundamental groups which do not come from geometry --- indeed, the following is an example of a representation which does not come from geometry, and which we do not know how to rule out by other means:
\begin{example}\label{ridic-example}
As before, let $Y(3)$ be the modular curve parametrizing elliptic curves with full level three structure.  Then $$Y(3)_{\mathbb{C}}\simeq \mathbb{P}^1\setminus\{0,1,\infty, \lambda\}$$ where $\lambda\in \mathbb{Q}(\zeta_3)$; let $x\in Y(3)(\mathbb{C})$ be a point.  Let $\rho$ be the tautological representation $$\rho: \pi_1(Y(3)(\mathbb{C})^{\text{an}}, x)\to GL(H^1(E_x(\mathbb{C})^{\text{an}}, \mathbb{Z})).$$  Let $$X=\mathbb{P}^1\setminus \{0,1,\infty, \beta\}$$ where $\beta\in \mathbb{Q}\setminus\{0,1\}$.  Then $X(\mathbb{C})^{\text{an}}$ is homeomorphic to $Y(3)(\mathbb{C})^{\text{an}}$ (indeed, both are homeomorphic to a four-times-punctured sphere); let $$j: X(\mathbb{C})^{\text{an}}\to Y(3)(\mathbb{C})^{\text{an}}$$ be such a homeomorphism.  Then the representation $$\tilde{\rho}: \pi_1(X(\mathbb{C})^{\text{an}}, j^{-1}(x))\overset{j_*}{\longrightarrow} \pi_1(Y(3)(\mathbb{C})^{\text{an}}, x)\overset{\rho}{\longrightarrow}  GL(H^1(E_x(\mathbb{C})^{\text{an}}, \mathbb{Z}))$$ is trivial mod $3$ and thus cannot arise from geometry, by Corollary \ref{p1-cor-intro-2}.  

We do not know a way to see this using pre-existing methods, since any criterion ruling out this representation would have to detect the difference between $X$ and $Y(3)$; for example, the quasi-unipotent local monodromy theorem does not rule out $\tilde \rho$.  This example was suggested to the author by George Boxer.
\end{example}
\begin{example}
Let $X/\mathbb{C}$ be a proper genus two curve; recall that $$\pi_1(X^{\text{an}})=<a_1, b_1, a_2, b_2>/([a_1,b_1][a_2,b_2]=1).$$  Let $p$ be a prime and $A, B$ non-unipotent $n\times n$ integer matrices which are equal to the identity mod $p^N$.  Then for $N\gg 0$, Theorem \ref{main-theorem-intro} implies that the representation $$a_1\mapsto A, b_1\mapsto B, a_2\mapsto B, b_2\mapsto A$$ does not come from geometry.  Again, we do not know how to see this using previously known results.
\end{example}

Part of the  motivation for this paper is the \emph{uniform boundedness conjecture}, also called the \emph{torsion conjecture}. Namely, there exists a function  $N=N(g, d)$ such that if $K$ is a number field and $A$ is a $g$-dimensional Abelian variety over $K$, then  $$\#|A(K)_{\text{tors}}|<N(g, [K:\mathbb{Q}]).$$
This conjecture is known for elliptic curves (the case $g=1$) by work of Merel \cite{merel} (building on work of Mazur \cite{mazur}, Kamienny \cite{kamienny}, and others).

The geometric analogue (the \emph{geometric torsion conjecture}) replaces number fields with function fields.  Namely, let $k$.  Then the conjecture is that there exists $N=N(g,d)$ such that if $A/k(X)$ is an Abelian variety of dimension $g$, then $$\#|A(k(X))_{\text{tors}}|< N(g, \text{gonality}(X))$$ for any curve $X$, unless $A$ has nontrivial $k(X)/k$-trace (see e.g.~\cite[Definition 6.1]{lang-neron} for a definition of the trace of an Abelian variety).  Here $\text{gonality}(X)$ is the minimal degree of a non-constant map  $X\to \mathbb{P}^1_k$.  This conjecture is also known in the case $g=1$; it is simply the statement that the gonality of modular curves tends to infinity with the level (see e.g.~\cite{poonen}).

Thus far the methods used to study the geometric torsion conjecture have been largely complex-analytic or algebro-geometric.  One of the goals of this paper was to use essentially \emph{anabelian} methods, which could conceivably be used to study the analogous questions for number fields.
\subsection{Statement of Main Results: Theoretical Aspects and Proofs}
We now sketch the idea of the proof of Theorem \ref{main-theorem-intro-2}.  The proof requires several technical results on arithmetic fundamental groups which are of independent interest.  For simplicity we assume $X$ is a smooth, geometrically connected curve over a finitely generated field $k$; indeed, one can immediately reduce to this case using an appropriate Lefschetz theorem.

\emph{Step 1  (Section \ref{Ql-galois-section}).}  Let $x\in X(k)$ be a rational point, and choose an embedding $k\hookrightarrow \bar k$; let $\bar x$ be the associated geometric point of $x$.  Let $\pi_1^{\text{\'et}, \ell}(X_{\bar k}, \bar x)$ be the pro-$\ell$ completion of the geometric \'etale fundamental group of $X$.  Then we construct certain special elements $\sigma\in G_k$, called quasi-scalars, whose action on the pro-unipotent completion of $\pi_1^{\text{\'et}, \ell}(X_{\bar k}, \bar x)$ is well-adapted to the weight filtration.
\begin{theorem}\label{pi1-quasi-scalar-intro}
For all $\alpha\in \mathbb{Z}_\ell^\times$ sufficiently close to $1$, there exists $\sigma\in G_k$ such that $\sigma$ acts on $$\on{gr}^{-i}_W \mathbb{Q}_\ell[[\pi_1^{\text{\'et}, \ell}(X_{\bar k}, \bar x)]]$$ by $\alpha^i$.  Here $W$ denotes the weight filtration on the $\mathbb{Q}_\ell$-Mal'cev Hopf algebra $\mathbb{Q}_\ell[[\pi_1^{\text{\'et}, \ell}(X_{\bar k}, \bar x)]]$ associated to $\pi_1^{\text{\'et}, \ell}(X_{\bar k}, \bar x)$.
\end{theorem}
This is the first of our main results on the structure of the Galois action on $\pi_1^{\text{\'et}, \ell}(X_{\bar k}, \bar x)$.  A key step in the construction of such a $\sigma$ (Theorem \ref{H1-quasi-scalar}) was suggested to the author by Will Sawin \cite{sawin} after reading an earlier version of this paper, and is related to ideas of Bogomolov \cite{bogomolov}.  For a more general statement, see Corollary \ref{pi1-quasi-scalar}.  Another key step in the proof of Theorem \ref{pi1-quasi-scalar-intro} is of additional independent interest:
\begin{theorem}
Let $k$ be a finite field and $X/k$ a smooth variety with a simple normal crossings compactification; let $x\in X(k)$ be a rational point.  Then Frobenius acts semi-simply on $$\mathbb{Q}_\ell[[\pi_1^{\text{\'et}, \ell}(X_{\bar k}, \bar x)]],$$ i.e.~its eigenvectors have dense span.
\end{theorem}
See Theorem \ref{semisimple-pi1} for a more general statement.

\emph{Step 2 (Sections \ref{integral-l-adic-periods-section} and \ref{convergent-section}).}  For each real number $r>0$, we construct certain Galois-stable subalgebras $$\mathbb{Q}_\ell[[\pi_1^{\text{\'et}, \ell}(X_{\bar k}, \bar x)]]^{\ell^{-r}}\subset \mathbb{Q}_\ell[[\pi_1^{\text{\'et}, \ell}(X_{\bar k}, \bar x)]]$$ called ``convergent group rings," which have the following property:  If $$\rho:\pi_1^{\text{\'et}, \ell}(X_{\bar k}, \bar x)\to GL_n(\mathbb{Z}_\ell)$$ is trivial mod $\ell^n$ with $r<n$, then there is a natural commutative diagram of continuous ring maps
$$\xymatrix{
\mathbb{Z}_\ell[[\pi_1^{\text{\'et}, \ell}(X_{\bar k}, \bar x)]]\ar[r]^-{\rho} \ar[d] & \mathfrak{gl}_n(\mathbb{Z}_\ell)\ar[d] \\
\mathbb{Q}_\ell[[\pi_1^{\text{\'et}, \ell}(X_{\bar k}, \bar x)]]^{\ell^{-r}} \ar[r] & \mathfrak{gl}_n(\mathbb{Q}_\ell).
}$$
where $\mathbb{Z}_\ell[[\pi_1^{\text{\'et}, \ell}(X_{\bar k}, \bar x)]]$ is the group ring of the pro-finite group $\pi_1^{\text{\'et}, \ell}(X_{\bar k}, \bar x)$ and $\mathfrak{gl}_n(R)$ is the ring of $n\times n$ matrices with entries in $R$.  Our second main result on the structure of the Galois action on $\pi_1^{\text{\'et}, \ell}(X_{\bar k}, \bar x)$ is:
\begin{theorem}\label{diagonalizing-convergent-rings-intro}
Let $\sigma$ be as in Theorem \ref{pi1-quasi-scalar-intro}.  Then there exists $r>0$ such that $\sigma$ acts (topologically) diagonalizably on $$\mathbb{Q}_\ell[[\pi_1^{\text{\'et}, \ell}(X_{\bar k}, \bar x)]]^{\leq \ell^{-r}},$$ i.e.~the span of the $\sigma$-eigenvectors is dense.
\end{theorem}
Loosely speaking, this means that the denominators of the $\sigma$-eigenvectors in $$\mathbb{Q}_\ell[[\pi_1^{\text{\'et}, \ell}(X_{\bar k}, \bar x)]]/\mathscr{I}^n$$ do not grow to quickly in $n$, where $\mathscr{I}$ is the augmentation ideal.

\emph{Step 3 (Section \ref{applications-section}).} Suppose $$\rho: \pi_1^{\text{\'et}, \ell}(X_{\bar k}, \bar x) \to GL(V)$$ is an arithmetic representation on a finite free $\mathbb{Z}_\ell$-module $V$. Then by a socle argument (Lemma \ref{subquotient-lemma}), we may assume that $\rho$ extends to a representation of $\pi_1^{\text{\'et}, \ell}(X_{k'}, \bar x)$ for some $k'/k$ finite.  In particular, for $m$ such that $\sigma^m\in G_{k'}\subset G_k$, $\sigma^m$ acts on $\mathfrak{gl}(V)$ so that the morphism $$\rho: \mathbb{Z}_\ell[[\pi_1^{\text{\'et}, \ell}(X_{\bar k}, \bar x)]]\to \mathfrak{gl}(V)$$ is $\sigma^m$-equivariant.

Let $r$ be as in Theorem \ref{diagonalizing-convergent-rings-intro}, and suppose $\rho$ is trivial mod $\ell^n$ for $n>r$.  Thus by Step $2$, we obtain a $\sigma^m$-equivariant map $$\tilde{\rho}: \mathbb{Q}_\ell[[\pi_1^{\text{\'et}, \ell}(X_{\bar k}, \bar x)]]^{\ell^{-r}}\to \mathfrak{gl}(V\otimes \mathbb{Q}_\ell).$$  But $\mathfrak{gl}(V\otimes \mathbb{Q}_\ell)$ is a finite-dimensional vector space; thus $\sigma^m$ has only finitely many eigenvalues.  But by Theorem \ref{diagonalizing-convergent-rings-intro}, and our choice of $\sigma$ in Theorem \ref{pi1-quasi-scalar-intro}, this implies that for $N\gg 0$, $$\tilde{\rho}(W^{-N}\mathbb{Q}_\ell[[\pi_1^{\text{\'et}, \ell}(X_{\bar k}, \bar x)]]^{\ell^{-r}})=0,$$ where again $W^{\bullet}$ denotes the weight filtration.  It is not hard to see that the $W$-adic topology and the $\mathscr{I}$-adic topology on $\mathbb{Q}_\ell[[\pi_1^{\text{\'et}, \ell}(X_{\bar k}, \bar x)]]$ agree, where $\mathscr{I}$ is the augmentation ideal. Hence for some $N'\gg 0$, $$\tilde{\rho}(\mathscr{I}^{N'}\cap \mathbb{Q}_\ell[[\pi_1^{\text{\'et}, \ell}(X_{\bar k}, \bar x)]]^{\leq \ell^{-r}})=0,$$ which implies $\rho$ is unipotent.

\subsection{Auxiliary Results}
We give some other applications of these techniques to certain very special geometric situations.

In particular, the Lang-N\'eron theorem~\cite[Theorem 2.1]{lang-neron} shows that if the representation $\rho$ arose from a family of Abelian varieties  over $X$, then the trace of the generic fiber of this family must be non-trivial.  This is our contribution to the geometric torsion conjecture.  To be precise:
\begin{corollary}\label{AV-corollary}
Let $X, K, \ell$ be as in Corollary \ref{p1-cor-intro}. Let $\eta$ be the generic point of $X$.  Then if $A$ is a non-zero Abelian scheme over $X$ with full level $\ell$ structure, we have $$\on{Tr}_{\mathbb{C}(X)/\mathbb{C}}A_{\eta}\not=0.$$
\end{corollary}
In other words, any family of Abelian schemes over $X$ with trivial $\ell$-torsion for $\ell$ large must have non-trivial trace.  An analysis of Theorem \ref{main-theorem-intro} gives an analogous result for any $X$.

We also give several results on varieties with nilpotent fundamental group.  See Section \ref{nilpotent-pi1} for details.
\subsection{Comparison to existing results} 
We now compare our results to results already in the literature.  To our knowledge, most prior results along these lines employ complex-analytic techniques, which began with Nadel \cite{nadel} and were built on by Noguchi \cite{noguchi} and Hwang-To \cite{hwang-to}.  There has been a recent flurry of interest in this subject, notably two recent beautiful papers by Bakker-Tsimerman \cite{bak2, bak1}, which alerted the author to this subject.

These beautiful results all use the hyperbolicity of $\mathscr{A}_{g,n}$ (the moduli space of principally polarized Abelian varieties with full level $n$ structure) to obstruct maps from curves of low gonality (of course, this description completely elides the difficult and intricate arguments in those papers).  The paper \cite{hwang-to} also proves similar results for maps into other locally symmetric spaces.

The papers \cite{nadel, noguchi, hwang-to} together show that  there exists an integer $N=N(g, d)$ such that if $A$ is a $g$-dimensional Abelian variety over a curve $X/\mathbb{C}$ with gonality $d$, then $A$ cannot have full level $N$ torsion unless it is a constant Abelian scheme.  The main deficiency of our result in comparison to these is that we do not give any uniformity in the gonality of $X$.  For example, if $X=\mathbb{P}^1\setminus \{x_1, \cdots, x_n\}$, our results depend on the cross-ratios of the $x_i$.

On the other hand, our result is uniform in $g$, whereas the best existing results (to our knowledge) are at least quadratic in $g$.  Thus for any given $X$ of genus zero, say, our results improve on those in the literature for $g$ large.  To our knowledge, this sort of uniformity in $g$ was not previously expected, and is quite interesting.  Moreover in many cases our result is sharp.  See Section \ref{sharp-section} for further remarks along these lines.

Our results also hold for arbitrary representations which arise from geometry, rather than just those of weight $1$ (e.g.~those that arise from Abelian varieties).  

Finally, to our knowledge, Theorem \ref{p1-main-theorem} is the first result (apart from \cite{poonen}, which bounds torsion of elliptic curves over function fields) along these lines which works in positive characteristic.  

More arithmetic work in this subject has been done by Abramovich-Madapusi Pera-Varilly-Alvarado \cite{abr1}, and Abramovich-Varilly-Alvarado \cite{abr2}, which relate the geometric versions of the torsion conjecture to the arithmetic torsion conjecture, assuming various standard conjectures relating hyperbolicity and rational points. Cadoret has also proven beautiful related arithmetic results (see e.g.~\cite{cadoret-tamagawa}).

\subsection{Galois action on the fundamental group}

The technical workhorse of this paper is a study of the action of the Galois group of a finitely generated field $k$ on the geometric fundamental group of a variety $X/k$.  We now compare our results on this subject to those in the literature.  In his inspiring paper \cite{deligne}, Deligne studies the fundamental group of $$\mathbb{P}^1\setminus \{0,1,\infty\}$$ and in particular, analyzes the action of $\on{Gal}(\overline{\mathbb{Q}}/\mathbb{Q})$ on a certain metabelian quotient of $\pi_1^{\text{\'et}}(\mathbb{P}^1\setminus \{0,1,\infty\})$.  Various other authors---notably Goncharov \cite{goncharov}, Ihara \cite{ihara2, ihara1}, Anderson-Ihara \cite{anderson-ihara1, anderson-ihara2}, Wojtkowiak \cite{wojtkowiakI, wojtkowiakII, wojtkowiakIII, wojtkowiakIV, wojtkowiakV}---have studied the action of $G_\mathbb{Q}$ on the \emph{pro-unipotent} completion of the fundamental group of $\mathbb{P}^1\setminus D$, for some effective divisor $D\subset \mathbb{P}^1$, and various quotients thereof.  In light of the $p$-adic comparison theorem proven by Andreatta, Kim, and Iovita \cite{kim-andreatta}, the work of Furusho \cite{furosho1, furosho2} also gives some information about the action of $G_{\mathbb{Q}_p}$  on   $\pi_1^{\text{\'et}}(\mathbb{P}^1\setminus \{0,1,\infty\})$.   

Deligne \cite[Section 19]{deligne} studies the action of $\on{Gal}(\overline{\mathbb{Q}}/\mathbb{Q})$ on a certain $\mathbb{Z}_\ell$-Lie algebra associated to a metabelian quotient of $\pi_1^{\text{\'et}}(\mathbb{P}^1\setminus \{0,1,\infty\})$.  In particular, he shows that certain ``polylogarithmic" extension classes are torsion, with order given by the valuations of special values of the Riemann zeta function at negative integers.  Our results are much blunter than these.  On the other hand, we do give results for the entire fundamental group, rather than a (finite rank) metabelian quotient.  Thus we are able to study \emph{integral, non-unipotent} aspects of the representation theory of fundamental groups.

Certain aspects of this work are also intimately related to the so-called $\ell$-adic iterated integrals of Wojtkowiak \cite{wojtkowiakI, wojtkowiakII, wojtkowiakIII, wojtkowiakIV, wojtkowiakV}.  In particular, given two rational points or rational tangential basepoints $x_i, x_j$ (see Section \ref{pi1-preliminaries}) on a normal variety $X$ over a finitely generated field $k$, and a special element $\sigma\in G_k$, we define the  \emph{canonical path} from $x_i$ to $x_j$ (Proposition-Construction \ref{canonical-paths-prop-constr}) to be the unique element $$p_{i,j}^{\sigma}\in \mathbb{Q}_\ell[[\pi_1^{\text{\'et}, \ell}(X_{\bar k}; \bar x_i, \bar x_j)]]$$ which is fixed by $\sigma$ and maps to $1$ under the augmentation map.  Here we define $\mathbb{Q}_\ell[[\pi_1^{\text{\'et}, \ell}(X_{\bar k}; \bar x_i, \bar x_j)]]$ to be a certain completion of the free pro-finite $\mathbb{Q}_\ell$-vector space on the \'etale torsor of paths from $x_i$ to $x_j$ (see Definition \ref{pro-l-groupoid-defn} and the surrounding remarks).  In Section \ref{integral-l-adic-periods-section}, we compute certain ``integral $\ell$-adic periods" associated to these canonical paths; these are related to Wojtkowiak's $\ell$-adic multiple polylogarithms.

In particular, if $p_{i,j}^n$ is the image of $p_{i,j}$ in $$\mathbb{Q}_\ell[[\pi_1^{\text{\'et}, \ell}(X_{\bar k}; \bar x_i, \bar x_j)]]/\mathscr{I}^n$$ (here $\{\mathscr{I}^n\}$ is the $\mathscr{I}$-adic filtration induced by viewing $\mathbb{Q}_\ell[[\pi_1^{\text{\'et}, \ell}(X_{\bar k}; \bar x_i, \bar x_j)]]$ as a left $\mathbb{Q}_\ell[[\pi_1^{\text{\'et}, \ell}(X_{\bar k}, \bar x_i)]]$-module), then there exists some $b^0_n\in \mathbb{Z}_{>0}$ such that $$\ell^{b^0_n}\cdot p_{i,j}^n\in \mathbb{Z}_\ell[[\pi_1^{\text{\'et}, \ell}(X_{\bar k}; \bar x_i, \bar x_j)]]/\mathscr{I}^n.$$  We call these positive integers $b^0_n$ the ``integral $\ell$-adic periods" associated to canonical paths; they are $\ell$-adic analogues of iterated integrals or multiple polylogarithms.  They are also $\ell$-adic analogues of Furusho's $p$-adic multiple zeta values \cite{furosho1, furosho2}.  Bounding these $b_n^0$ (as we do in Section \ref{integral-l-adic-periods-section}) is key to our results in genus zero. 

One may also view this work as an application of anabelian geometry to Diophantine questions (in particular, about function-field valued points of $\mathscr{A}_{g,n}$ and other period domains).  There is a tradition of such applications---see e.g. Kim's beautiful paper \cite{kim} and work of Wickelgren \cite{wickelgren, wickelgren2}, for example.  In particular, we believe our ``integral $\ell$-adic periods" to be related to Wickelgren's work on Massey products.
\subsection{Denominators of associators}
In the case that $$X=\mathbb{P}^1\setminus\{0,1,\infty\}$$ and $x_i, x_j$ are tangential basepoints located at $0$ and $1$, the canonical paths $p_{i,j}$ are also $\ell$-adic analogues of Drinfel'd's KZ associator \cite{drinfeld1, drinfeld2} and Furusho's $p$-adic Drinfel'd associator \cite{furosho1, furosho2}.  Our bounds on the integers $b^0_n$ are closely related to work of Alekseev-Podkopaeva-\v{S}evera \cite{rational-associators} bounding the denominators of coefficients of rational Drinfel'd associators.  It would be interesting to better understand the relationship between the work in this paper and e.g.~Alekseev-Torossian's work on the Kashiwara-Vergne conjecture \cite{alekseev-torossian}.
\subsection{Integral Hodge theory---beyond the pro-unipotent fundamental group}
Aside from our applications of anabelian techniques to questions about monodromy, we believe the main contributions of this paper to be 
\begin{enumerate}
\item the introduction of the rings $$\mathbb{Q}_\ell[[\pi_1^{\text{\'et}, \ell}(X_{\bar k}, \bar x)]]^{\leq \ell^{-r}},$$ and
\item the study of the action of $G_k$, for $k$ a finitely generated field, on these rings.
\end{enumerate}
Let $k$ be a field and $X$ a variety over $k$.  Almost all previous work (with the exception of \cite[Section 19]{deligne}) studying the Galois action on $$G(X):=\pi_1^{\text{\'et}, \ell}(X_{\bar k}, \bar x)$$ has focused on the \emph{pro-unipotent} completion of $G(X)$; this completion is a brutal operation which remembers only the \emph{unipotent} representation theory of $G(X)$, by definition.  The rings $\mathbb{Q}_\ell[[\pi_1^{\text{\'et}, \ell}(X_{\bar k}, \bar x)]]^{\leq \ell^{-r}},$ remember significantly more---indeed, as one ranges over all $r>0$, they remember all continuous representations of $G(X)$, and as evidenced by our applications, they allow one to study the \emph{integral structure} of such representations.  In future work, we hope to elucidate connections to integral structures on the Mumford-Tate group.  We also hope to pursue integral $p$-adic Hodge theory of the fundamental group.
\subsection{Notation}
Throughout, a variety $X/k$ will be a finite-type, separated, integral $k$-scheme.  When we write $H^i(X_{\bar k}, \mathbb{Z}_\ell)$ we will always mean the $\ell$-adic cohomology, i.e.~$$\varprojlim_r H^i(X_{\bar k, \text{\'et}}, \mathbb{Z}/\ell^r\mathbb{Z}),$$ and $$H^i(X_{\bar k}, \mathbb{Q}_\ell):=H^i(X_{\bar k}, \mathbb{Z}_\ell)\otimes \mathbb{Q}_\ell.$$
\subsection{Guide to reading the paper}
In Section \ref{pi1-preliminaries}, we will summarize some preliminaries on the \'etale fundamental group, inertia, and (rational) tangential basepoints.  Section \ref{hopf-preliminaries-section} will summarize standard facts about $\ell$-adic group algebras for pro-$\ell$ groups, and slightly less standard facts about so-called Hopf groupoids associated to pro-$\ell$ groupoids.  Section \ref{Ql-galois-section} will describe the $G_k$-action on the $\mathbb{Q}_\ell$-pro-unipotent fundamental group of a variety. Section \ref{integral-l-adic-periods-section} refines this analysis to an analysis of the $G_k$-action on the $\mathbb{Z}_\ell$-group ring of the pro-$\ell$ quotient of the fundamental group.   Section \ref{convergent-section} defines convergent group rings and Hopf groupoids and analyzes their properties.  Section \ref{applications-section} gives applications to restricting possible geometric monodromy representations, including those mentioned in the introduction.  Finally, Section \ref{hopf-groupoids} is an appendix developing some useful abstract theory of Hopf groupoids.

Expert readers may skip Sections \ref{pi1-preliminaries} and \ref{hopf-preliminaries-section}, referring back to them only as necessary for details or statements.  The reader only interested in applications may read Sections \ref{Ql-galois-section}, \ref{integral-l-adic-periods-section}, \ref{convergent-section}, and \ref{applications-section} for an essentially self-contained presentation, referring only to other sections as necessary.
\subsection{Acknowledgments} 
This work owes a great deal to discussions with Johan de Jong, Ravi Vakil, David Hansen, Daniel Halpern-Leistner, and many others.  Its genesis was a beautiful question by Lisa S.~on Mathoverflow \cite{lisa}, a talk by Benjamin Bakker on his work with Jacob Tsimerman \cite{bak2, bak1}, and some discussions with Akshay Venkatesh and Brian Conrad in the author's final year of graduate school.  I am grateful to George Boxer for suggesting the simple example of a representation that does not come from geometry in Example \ref{ridic-example}. Finally, Will Sawin suggested the use of Theorem \ref{H1-quasi-scalar} after reading an earlier version of this paper, and pointed the author to \cite{bogomolov}, which dramatically improved the main results of the paper --- I am deeply grateful to him.

\section{Preliminaries on fundamental groups}\label{pi1-preliminaries}

We briefly recall some useful notions about the \'etale fundamental group of a variety.  The standard reference for the fundamental group is \cite{SGA1}; the standard reference on tangential basepoints is \cite{deligne}.

Recall that the formalism of \cite[Expos\'e V]{SGA1} associates a profinite group to any Galois category.  A Galois category is a category $\mathscr{C}$ equipped with a functor $F:\mathscr{C}\to \on{Sets}$ satisfying certain axioms \cite[Expos\'e V.4]{SGA1}; the functor $F$ is called a \emph{fiber functor}.  

The typical application of this formalism is the following:  let $S$ be a connected, locally Noetherian scheme, and $\bar s: \on{Spec}(\Omega)\to S$ a geometric point of $S$ (that is, $\Omega$ is an algebraically closed field).  Then the category $\text{F\'et}_S$ of finite \'etale $S$-schemes equipped with the functor $F_{\bar s}: X \mapsto X_{\bar s}$ is a Galois category, and the \'etale fundamental group $\pi_1^{\text{\'et}}(S, \bar s)$ is defined to be the pro-finite group $\on{Aut}(F_{\bar s})$.  Similarly, if $\bar s_1, \bar s_2$ are two geometric points of $S$, one defines the \'etale path torsor $\pi_1^{\text{\'et}}(S; \bar s_1, \bar s_2)$ to be the profinite set $\on{Isom}(F_{\bar s_1}, F_{\bar s_2}).$  

This construction is functorial in $(S, \bar s)$.  If $S$ is a $k$-scheme, $\epsilon: k\to \bar k$ is an embedding of $k$ in an algebraic closure $ \bar k$ of $k$, and $\bar s: \on{Spec}(\bar k)\to S$ is a geometric point, the induced morphism $\pi_1^{\text{\'et}}(S, \bar s)\to \pi_1^{\text{\'et}}(\on{Spec}(k), \epsilon)=\on{Gal}(\bar k/k)$ fits into a fundamental exact sequence \cite[Th\'eor\`eme 6.1]{SGA1}: 

\begin{equation}\label{fundamental-exact-sequence}
1\to \pi_1^{\text{\'et}}(S_{\bar k}, \bar s)\to \pi_1^{\text{\'et}}(S, \bar s)\to \on{Gal}(\bar k/k)\to 1.
\end{equation}

If $s\in S(k)$ is a rational point and $\bar s: \on{Spec}(\bar k)\to S$ is the induced map, functoriality of $\pi_1^\text{\'et}$ gives a splitting of the above sequence, $$\on{Gal}(\bar k/k)\to \pi_1^{\text{\'et}}(S, \bar s).$$

If $S/k$ is not proper but admits a compactification $\overline{S}$ whose boundary is a normal crossings divisor, there is another Galois category with associated fundamental group which will be relevant to us.  Namely, we replace $\text{F\'et}_S$ with $\text{F\'et}_S^{\text{tame}}$, the category of \'etale covers of $S$ tamely ramified along $\overline{S}\setminus S$.  Geometric points $\bar s$ of $S$ again give fiber functors on this category; the associated group $$\pi_1^{\text{tame}}(S, \bar s)=\on{Aut}(F_{\bar s})$$ is called the tame fundamental group \cite{tame-pi1}.  We define the pro-$\ell$ fundamental group of $S$ to be the maximal pro-$\ell$ quotient of the geometric fundamental group $\pi_1^{\text{\'et}}(S_{\bar k}, \bar s)$ and denote it $\pi_1^{\text{\'et}, \ell}(S_{\bar k}, \bar s)$.  Both of these examples come with fundamental exact sequences as in (\ref{fundamental-exact-sequence}) above, constructed by pushing out sequence (\ref{fundamental-exact-sequence}) along the natural map $\pi_1^{\text{\'et}}(S_{\bar k}, \bar s)\to \pi_1^\text{tame}(S_{\bar{k}}, \bar s)$ (resp.~$\pi_1^{\text{\'et}}(S_{\bar k}, \bar s)\to \pi_1^{\text{\'et}, \ell}(S_{\bar{k}}, \bar s)$).  For example, we have
$$\xymatrix{
1\ar[r]&  \pi_1^{\text{\'et}}(S_{\bar k}, \bar s)\ar[r] \ar@{->>}[d]& \pi_1^{\text{\'et}}(S, \bar s)\ar[r] \ar[d]& \on{Gal}(\bar k/k)\ar[r] \ar[d] & 1 \\
1\ar[r]&  \pi_1^{\text{\'et}, \ell}(S_{\bar k}, \bar s)\ar[r]& \pi_1^{\text{\'et}, \ell}(S, \bar s)\ar[r] & \on{Gal}(\bar k/k)\ar[r] & 1.
}$$
Here the left square is a pushout square, and $\pi_1^{\text{\'et}, \ell}(S, \bar s)$ is defined to be the pushout.  Note that it is not in general a pro-$\ell$ group.

\subsection{Tangential Basepoints}

Not all fiber functors $\text{F\'et}_S\to \on{Sets}$ come from $\bar k$-points of $S$, as in the previous paragraph; likewise, not all splittings of the fundamental exact sequence \eqref{fundamental-exact-sequence} above come from rational points. We will find it convenient to consider certain so-called \emph{tangential basepoints}, first constructed by Deligne \cite[15.13-15.27]{deligne}, which are fiber functors whose associated fundamental exact sequences split, as in the case of fiber functors obtained from rational points.

Let $X$ be a scheme, and let $R$ be a discrete valuation ring with fraction field $K$ and residue field $k$; let $\bar K$ be an algebraic closure of $K$.  An $R$-tangential basepoint of $X$ is a morphism $\on{Spec}(K)\to X$; typically we will take $X$ to be a $k$-scheme and $R=k[[t]]$.  In this case we say that the basepoint is a \emph{rational tangential basepoint}.  Choose an embedding $K\hookrightarrow \overline{K}$ and let $\bar x$ be the induced geometric point of $X$.  If we let $\mathscr{C}$ be the category of finite \'etale covers of $X$ (resp. tame covers), then the pair $(\mathscr{C}, F_{\bar x}: U\mapsto U_{\bar{x}})$ is a Galois category.  We define $\pi_1^{\text{\'et}}(X, \bar x),$ (resp.~$\pi_1^{\text{tame}}(X, \bar x)$) as before to be $\on{Aut}(F_{\bar x})$.  Likewise, we define $\pi_1^{\text{\'et}, \ell}(X_{\bar k}, \bar x)$ to be the maximal pro-$\ell$ quotient of $\pi_1^{\text{\'et}}(X, \bar x)$.

If $\bar x_1, \bar x_2$ are two tangential basepoints or geometric points of $X$, one defines the \'etale path torsor $$\pi_1^{\text{\'et}}(X; \bar x_1, \bar x_2):=\on{Isom}(F_{\bar x_1}, F_{\bar x_2}).$$

\subsection{Inertia and Galois Actions}\label{inertia-section}  The utility of tangential basepoints (at least in this paper) comes from the following features:

\begin{enumerate}

\item Let $I\subset \on{Gal}(\bar K/K)$ be the inertia group of $K$.  Then the tangential basepoint $\bar x$ as above induces a natural map $\tau: I\to \pi_1^{\text{\'et}}(X, \bar x)$.

\item If the natural map $\on{Gal}(\bar K/K)\to \on{Gal}(\bar k/k)$  splits (e.g. if $R=k[[t]]$, or if $k$ is finite), and $X$ is a $k$-scheme, then the fundamental exact sequence $$1\to \pi_1^{\text{\'et}}(X_{\bar k}, \bar x)\to \pi_1^{\text{\'et}}(X, \bar x)\to \on{Gal}(\bar k/k)\to 1$$ is split by the composition $$\on{Gal}(\bar k/k)\to \on{Gal}(\bar K/K)\to \pi_1^\text{\'et}(X, \bar x).$$  The action of $\on{Gal}(\bar k/k)$ on (the prime-to-$p$ quotient of) $I$ is well-understood, so this gives us a subgroup of $\pi_1^{\text{\'et}}(X, \bar x)$ on which we know the Galois action, namely the image of $\tau$.

\end{enumerate}

In particular, let $X$ be a $k$-scheme with a normal crossings compactification, and let $$\bar x: \on{Spec}(k((t)))\to X$$ be a tangential basepoint.  Let $L$ be the maximal tame extension of $k((t))$ Let $I^{\text{tame}}\subset \on{Gal}(L/k((t)))$ be the tame inertia group of $k((t))$.  Then there is a diagram of exact sequences

$$\xymatrix{
1\ar[r] & I^{\text{tame}}\ar[r]\ar[d] & \on{Gal}(L/k((t)))\ar[r]\ar[d]& \on{Gal}(\bar k/k)\ar[r] \ar[d]^=& 1\\
1\ar[r] & \pi_1^{\on{tame}}(X_{\bar k}, \bar x)\ar[r] & \pi_1^{\text{tame}}(X, \bar x)\ar[r] & \on{Gal}(\bar k/k)\ar[r] & 1 
}$$

so that the splitting $\on{Gal}(\bar k/k)\to \on{Gal}(L/k((t)))$ induces a splitting $$\on{Gal}(\bar k/k)\to \pi_1^{\text{tame}}(X, \bar x).$$  As a $\on{Gal}(\bar k/k)$-module, $$I^\text{tame}=\prod_\ell \mathbb{Z}_\ell(1),$$ where the product is taken over all $\ell$ prime to the characteristic of $k$.

Composing with the map $\pi_1^{\text{tame}}(X_{\bar k}, \bar x)\to \pi_1^{\text{\'et}, \ell}(X_{\bar k}, \bar x)$, we obtain a diagram 

$$\xymatrix{
1\ar[r] & \mathbb{Z}_\ell(1) \ar[r] \ar[d]^{\iota_x} & \mathbb{Z}_\ell(1)\rtimes \on{Gal}(\bar k/k) \ar[r]\ar[d] & \on{Gal}(\bar k/k) \ar[r] \ar[d] & 1\\
1 \ar[r]& \pi_1^{\text{\'et}, \ell}(X_{\bar k}, \bar x) \ar[r] & \pi_1^{\text{\'et}, \ell}(X_{\bar k}, \bar x) \rtimes \on{Gal}(\bar k/k) \ar[r] & \on{Gal}(\bar k/k) \ar[r] & 1
}$$

of split exact sequences of fundamental groups.  Observe that if the tangential basepoint $\bar x$ extends to a map $$\bar{x}': \on{Spec}(R)\to X,$$ the maps $I^{\text{tame}}\to \pi_1^{\text{tame}}(X_{\bar k}, \bar x)$ and $\iota_x: \mathbb{Z}_\ell(1)\to \pi_1^{\text{\'et}, \ell}(X_{\bar k}, \bar x)$ above are trivial; if no such extension exists, however, they may be non-trivial.  In this latter case, we say that $\bar x$ is a tangential basepoint \emph{at a puncture} of $X$.

\section{Preliminaries on group algebras and Hopf groupoids}\label{hopf-preliminaries-section}

Let $\ell$ be a prime.  The purpose of this section is to fix notation and record some basic facts about the group rings of pro-$\ell$ groups and an analogous notion for groupoids.  We will later apply these results to the pro-$\ell$ \'etale fundamental groupoid of a variety.

\subsection{Group algebras} 

We first collect some results about the completeness of $\ell$-adic group algebras for pro-$\ell$ groups, with respect to the augmentation ideal.

\begin{lemma}\label{augmentation-nilpotent-ell-group}
Let $G$ be a finite $\ell$-group, and $k$ a field of characteristic $\ell$.  Then the augmentation ideal $\mathscr{I}_G$ in $k[G]$ is nilpotent.  
\end{lemma}

\begin{proof}
Finite $\ell$-groups are solvable, so there exists a normal  subgroup $H\subset G$ with $G/H$ nontrivial and Abelian of exponent $\ell$.  Then for each $g\in G/H$, $(g-1)^\ell=0$.  As the augmentation ideal $\mathscr{I}_{G/H}\subset k[G/H]$ is generated by such elements, and is finitely generated, $\mathscr{I}_{G/H}$ is nilpotent.

Now consider the exact sequence $$0\to \mathscr{I}_Hk[G]\to k[G]\to k[G/H]\to 0$$ (here $\mathscr{I}_H$ is the augmentation ideal in $k[H]$; $\mathscr{I}_Hk[G]$ is a two-sided ideal as $H$ is normal, and indeed $\mathscr{I}_Hk[G]=k[G]\mathscr{I}_H$).  Now for some $N$, $\mathscr{I}_G^N$ maps to zero in $k[G/H]$ by the argument of the previous paragraph; hence $\mathscr{I}_G^N\subset \mathscr{I}_{H}k[G]$.

But $\mathscr{I}_Hk[G]=k[G]\mathscr{I}_H$, so $(\mathscr{I}_Hk[G])^M=\mathscr{I}_H^Mk[G]$.  For sufficiently large $M$, this is zero by induction on $|G|$.
\end{proof}

For $G$ a pro-$\ell$ group, define $$\mathbb{Z}_\ell[[G]]:=\varprojlim \mathbb{Z}_\ell[G/U],$$ where the inverse limit is taken over all open normal subgroups $U\subset G$. There is a natural augmentation map $$\epsilon: \mathbb{Z}_\ell[[G]]\to\mathbb{Z}_\ell,$$ obtained as the inverse limit of the usual augmentation maps $\mathbb{Z}_\ell[G/U]\to \mathbb{Z}_\ell.$

\begin{lemma}\label{IG-complete}
Suppose $G$ is a pro-$\ell$ group.  Let $\mathscr{I}_G\subset \mathbb{Z}_\ell[[G]]$ be the augmentation ideal, $\mathscr{I}_G:=\ker(\epsilon)$.  Then $\mathbb{Z}_\ell[[G]]$ is complete with respect to the $\mathscr{I}_G$-adic topology.
\end{lemma}

\begin{proof}
First, observe that if $H$ is a finite $\ell$-group, Lemma \ref{augmentation-nilpotent-ell-group} implies that for any $r$, the augmentation ideal $\mathscr{I}_H$ of $$\mathbb{Z}/\ell^r\mathbb{Z}[H]$$ is nilpotent.  Indeed, the lemma implies that for some $M$, $\mathscr{I}_H^M\subset (\ell)$; but the ideal $(\ell)$ is clearly nilpotent.  Hence $\mathbb{Z}_\ell[H]$ is complete with respect to its augmentation ideal, as 

\begin{align*}
\mathbb{Z}_\ell[H] &=\varprojlim_r \mathbb{Z}/\ell^r\mathbb{Z}[H] \\
&=\varprojlim_r \varprojlim_s \mathbb{Z}/\ell^r\mathbb{Z}[H]/\mathscr{I}_H^s \\
&=\varprojlim_s \varprojlim_r \mathbb{Z}/\ell^r\mathbb{Z}[H]/\mathscr{I}_H^s \\
&=\varprojlim_s \mathbb{Z}_\ell[H]/\mathscr{I}_H^s.
\end{align*}

But now a similar interchange of limits argument implies that $\mathbb{Z}_\ell[[G]]$ is complete with respect to its augmentation ideal.  Indeed, we have

$$\mathbb{Z}_\ell[[G]]=\varprojlim_{G\twoheadrightarrow H}\mathbb{Z}_\ell[H]=\varprojlim_{G\twoheadrightarrow H}\varprojlim_s\mathbb{Z}_\ell[H]/\mathscr{I}_H^s.$$  As before, interchanging the limits gives the result.
\end{proof}

If $R$ is a $\mathbb{Z}_\ell$-algebra , let $$R[[G]]:=\varprojlim_n \left(\mathbb{Z}_\ell[[G]]/\mathscr{I}_G^n\otimes_{\mathbb{Z}_\ell} R\right).$$  By Lemma \ref{IG-complete}, we obtain our original definition of $\mathbb{Z}_\ell[[G]]$ if we set $R=\mathbb{Z}_\ell$, so the overloaded notation is harmless. 

We will primarily study $$\mathbb{Z}_\ell[[G]], \mathbb{Q}_\ell[[G]];$$ observe that the natural injection $$\mathbb{Z}_\ell[[G]]\otimes \mathbb{Q}_\ell\to \mathbb{Q}_\ell[[G]]$$  is not an isomorphism. The image of this injection is the set of elements with ``bounded denominators."

Let $$\epsilon: R[[G]]\to R$$ be the augmentation map, obtained as the inverse limit of the maps $$\epsilon_U: R[G/U]\to R$$ $$gU\mapsto 1$$ for $gU\in G/U$.  The augmentation ideal $\mathscr{I}_G$ is the kernel of $\epsilon$.

\begin{lemma}\label{augmentation-abelianization-group-ring}
Suppose $G$ is a pro-$\ell$ group and $R$ is a $\mathbb{Z}_\ell$-algebra.  The map $g\mapsto g-1$ induces an isomorphism $$G^{ab}\otimes_{\mathbb{Z}_\ell} R\to \mathscr{I}_G/\mathscr{I}_G^2.$$ 
\end{lemma}

\begin{proof}
We first check that the map is well-defined.  Observe that if $g,h\in G$, $$(g-1)(h-1)=gh-g-h+1=(gh-1)-(g-1)-(h-1)\in \mathscr{I}_G^2.$$ Hence $$gh-1=(g-1)+(h-1)=hg-1\bmod \mathscr{I}_G^2,$$ so the map $$G\to \mathscr{I}_G/\mathscr{I}_G^2$$ $$g\mapsto g-1$$ does indeed factor through $G^{ab}$; as $\mathscr{I}_G/\mathscr{I}_G^2$ is an $R$-module, we obtain a map $G^{ab}\otimes_{\mathbb{Z}_\ell} R\to \mathscr{I}_G/\mathscr{I}_G^2$, as claimed.

The inverse map is induced by the map $\mathscr{I}_G\to G^{ab}$ which sends the basis elements $g-1$ to $g$; this factors through $\mathscr{I}_G/\mathscr{I}_G^2$ as $(g-1)(h-1)=(gh-1)-(g-1)-(h-1)$ maps to $ghg^{-1}h^{-1}\in [G,G]$.
\end{proof}

\subsection{Groupoids and Hopf groupoids} \label{hopf-groupoids-section} 

Now let $\mathcal{G}$ be a groupoid (a category in which every morphism is an isomorphism).  For objects $a, b\in \mathcal{G}$, we let $\mathcal{G}(a, b)$ denote the set of morphisms from $a$ to $b$.  

\begin{defn}\label{pro-l-groupoid-defn}
A \emph{pro-$\ell$ groupoid} is a groupoid $\mathcal{G}$ such that 

\begin{enumerate}

\item $\mathcal{G}(a,b)$ is a profinite set for each $a, b$, 

\item  the composition maps $\mathcal{G}(a, b)\times \mathcal{G}(b,c)\to \mathcal{G}(a, c)$, and the inversion maps $\mathcal{G}(a,b)\to \mathcal{G}(b,a)$  are continuous with respect to the profinite topology, and
\item the groups $\mathcal{G}(a,a)$ are pro-$\ell$ groups for each $a$.

\end{enumerate}
\end{defn}

For each $a,b$ in a pro-$\ell$ groupoid $\mathcal{G}$, we define 

$$\mathbb{Z}_\ell[[\mathcal{G}(a,b)]]=\varprojlim_{\mathcal{G}(a,b)\twoheadrightarrow H} \mathbb{Z}_\ell[H]$$ 
where the inverse limit is taken over continuous surjections $\mathcal{G}(a, b)\twoheadrightarrow H$, with $H$ a (discrete) finite set.  The composition $\mathcal{G}(a, b)\times \mathcal{G}(b, c)\to \mathcal{G}(a, c)$ induces a continuous map 

$$\nabla: \mathbb{Z}_\ell[[\mathcal{G}(a, b)]]\otimes_{\mathbb{Z}_\ell} \mathbb{Z}_\ell[[\mathcal{G}(b, c)]]\to \mathbb{Z}_\ell[[\mathcal{G}(a, c)]].$$

Similarly, there is a natural inversion map 

$$S: \mathbb{Z}_\ell[[\mathcal{G}(a, b)]]\to \mathbb{Z}_\ell[[\mathcal{G}(b, a)]],$$ 

sending $g$ to $g^{-1}$, and a ``comultiplication"

 $$\Delta: \mathbb{Z}_\ell[[\mathcal{G}(a, b)]]\to \mathbb{Z}_\ell[[\mathcal{G}(a,b)]]\otimes_{\mathbb{Z}_\ell} \mathbb{Z}_\ell[[\mathcal{G}(a,b)]]$$ 
 
 sending $g\in \mathcal{G}(a, b)$ to $g\otimes g$ and extended linearly.  There are obvious co-unit maps  
 
 $$\epsilon: \mathbb{Z}_\ell[[\mathcal{G}(a, b)]]\to \mathbb{Z}_\ell,$$ 
 
 and for $a=b$, unit maps 
 
 $$\eta: \mathbb{Z}_\ell\to \mathbb{Z}_\ell[[\mathcal{G}(a,a)]].$$ 
 
  These maps satisfy axioms analogous to those of a Hopf algebra.  In fact, the $\mathbb{Z}_\ell[[\mathcal{G}(a,b)]]$ form a \emph{Hopf groupoid}; see Section \ref{hopf-groupoids} for an appendix developing the necessary theory of such objects. In particular, the following diagrams commute:

 $$\xymatrix{
 \mathbb{Z}_\ell[[\mathcal{G}(a, b)]]\otimes \mathbb{Z}_\ell[[\mathcal{G}(a, b)]] \ar[rr]^{S\otimes \text{id}}& & \mathbb{Z}_\ell[[\mathcal{G}(b, a)]]\otimes \mathbb{Z}_\ell[[\mathcal{G}(a, b)]] \ar[d]^\nabla\\
\mathbb{Z}_\ell[[\mathcal{G}(a, b)]] \ar[u]^\Delta \ar[r]^\epsilon & \mathbb{Z}_\ell  \ar[r]^\eta & \mathbb{Z}_\ell[[\mathcal{G}(b, b)]]
}$$

 $$\xymatrix{
 \mathbb{Z}_\ell[[\mathcal{G}(a, b)]]\otimes \mathbb{Z}_\ell[[\mathcal{G}(a, b)]] \ar[rr]^{\text{id}\otimes S}& & \mathbb{Z}_\ell[[\mathcal{G}(a, b)]]\otimes \mathbb{Z}_\ell[[\mathcal{G}(b, a)]] \ar[d]^\nabla\\
\mathbb{Z}_\ell[[\mathcal{G}(a, b)]] \ar[u]^\Delta \ar[r]^\epsilon & \mathbb{Z}_\ell  \ar[r]^\eta & \mathbb{Z}_\ell[[\mathcal{G}(a, a)]].
}$$

In Corollary \ref{I-adic-filtration} of the Appendix, we construct an $\mathscr{I}$-adic filtration on each $\mathbb{Z}_\ell[[\mathcal{G}(a, b)]]$, which we may view as either the $\mathscr{I}(a,a)$-adic filtration on $\mathbb{Z}_\ell[[\mathcal{G}(a, b)]]$, viewed as a left $\mathbb{Z}_\ell[[\mathcal{G}(a,a)]]$-module, or as the $\mathscr{I}(b,b)$-adic filtration on $\mathbb{Z}_\ell[[\mathcal{G}(a,b)]]$, viewed as a right $\mathbb{Z}_\ell[[\mathcal{G}(b,b)]]$-module.  As in Lemma \ref{augmentation-abelianization-group-ring}, a proof identical to that of Lemma \ref{augmentation-abelianization-groupoid} shows that $\mathscr{I}(a,b)/\mathscr{I}^2(a,b)$ is canonically isomorphic to $\mathcal{G}(a,a)^{\text{ab}}\otimes \mathbb{Z}_\ell$ if $\mathcal{G}(a,b)$ is non-empty.

As in the group algebra case, if $R$ is a $\mathbb{Z}_\ell$-algebra, we define $R[[\mathcal{G}]]=\{R[[\mathcal{G}(a, b)]]\}$ via $$R[[\mathcal{G}(a, b)]]:=\varprojlim_n \left(\mathbb{Z}_\ell[[\mathcal{G}(a, b)]]/\mathscr{I}^n(a,b)\otimes R\right).$$

\section{Galois action on the $\mathbb{Q}_\ell$-fundamental group}\label{Ql-galois-section}
\subsection{Generalities on the Galois Action on $\pi_1^{\text{\'et}, \ell}$}  Let $X$ be a variety over a field $k$; let $\ell$ be a prime different from $\on{char}(k)$. We now study the action of $G_k:=\on{Gal}(\bar k/k)$ on the geometric (pro-$\ell$) fundamental group of $X$ $$\pi_1^{\text{\'et}, \ell}(X_{\bar k}, \bar x)$$ for some rational point or rational tangential basepoint $x$, and on the group algebra $\mathbb{Q}_\ell[[\pi_1^{\text{\'et}, \ell}(X_{\bar k}, \bar x)]]$.  After some preliminaries on the $\mathscr{I}$-adic filtration of $\mathbb{Q}_\ell[[\pi_1^{\text{\'et}, \ell}(X_{\bar k}, \bar x)]]$, we define the weight filtration $W^\bullet$ on this pro-finite-dimensional vector space.

The main results of this section
\begin{enumerate}
\item prove that the action of certain elements of $G_k$, called \emph{quasi-Frobenii} (including Frobenii at primes of good reduction), on $\mathbb{Q}_\ell[[\pi_1^{\text{\'et}, \ell}(X_{\bar k}, \bar x)]]$ is semi-simple (Theorem \ref{semisimple-pi1}), and
\item  prove the existence of certain special elements in $\on{Gal}(\bar k/k)$, called \emph{quasi-scalars}, whose action on $\mathbb{Q}_\ell[[\pi_1^{\text{\'et}, \ell}(X_{\bar k}, \bar x)]]$ is especially well-adapted to the weight filtration (Corollary \ref{pi1-quasi-scalar}).
\end{enumerate}

Let $\{x_1, \cdots, x_n\}$ be a collection of rational points of $X$, or rational tangential basepoints.  Let $k\to \bar{k}$ be a choice of algebraic closure of $k$, and let $k((t))\to \overline{k((t))}$ be a compatible choice of algebraic closure of $k((t))$.  Then the \emph{geometric fundamental groupoid} of $X$ associated to $\{x_1, \cdots, x_n\}$ has objects $\{x_1, \cdots, x_n\}$ and morphisms between $x_i$ and $x_j$ given by $$\pi_1^{\text{\'et}}(X_{\bar k}; \bar x_i, \bar x_j).$$  The \emph{geometric pro-$\ell$ fundamental groupoid} $\Pi_1(X_{\bar k})$ is the maximal pro-$\ell$ quotient of the geometric fundamental groupoid; we denote the morphisms from $x_i$ to $x_j$ by $$\pi_1^{\text{\'et},\ell}(X_{\bar k}; \bar x_i, \bar x_j).$$

We define the \emph{arithmetic pro-$\ell$ fundamental groupoid} $$\{\pi_1^{\text{\'et}, \ell}(X; \bar x_i, \bar x_j)\}$$ as having objects $\{x_1, \cdots, x_n\}$ as before and with morphisms $x_i$ to $x_j$ defined by the pushout diagram

$$\xymatrix{
\pi_1^{\text{\'et}}(X_{\bar k}; \bar x_i, \bar x_j) \ar[r] \ar[d] & \pi_1^{\text{\'et}}(X;\bar x_i, \bar x_j) \ar[d] \\
\pi_1^{\text{\'et},\ell}(X_{\bar k}; \bar x_i, \bar x_j) \ar[r] & \pi_1^{\text{\'et}, \ell}(X; \bar x_i, \bar x_j).
}$$

While the geometric pro-$\ell$ fundamental groupoid is a pro-$\ell$ groupoid, one should note that the arithmetic pro-$\ell$ groupoid is not a pro-$\ell$ groupoid in general.

As each $x_i$ is a rational point or rational tangential basepoint, the natural maps $$\pi_1^{\text{\'et}}(X, \bar x_i)\to G_k$$ have sections $s_i: G_k\to \pi_1^{\text{\'et}}(X, \bar x_i)$ (induced by the maps $x_i\to X$).  Hence there are also natural sections to the maps $$\pi_1^{\text{\'et}, \ell}(X, \bar x_i)\to G_k.$$    

These sections give an action of $$G_k:=\on{Gal}(\bar k/k)$$ on $\pi_1^{\text{\'et}}(X_{\bar k}; \bar x_i, \bar x_j)$ and hence on $\pi_1^{\text{\'et}, \ell}(X_{\bar k}; \bar x_i, \bar x_j)$, defined as follows.  Let $\sigma\in G_k$ be any element; then $\sigma$ acts on $g\in \pi_1^{\text{\'et}}(X; \bar x_i, \bar x_j)$ via $$g\mapsto s_i(\sigma)\cdot g\cdot s_j(\sigma)^{-1}.$$  One may check directly that this action preserves the subset $$\pi_1^{\text{\'et}}(X_{\bar k}; \bar x_i, \bar x_j)\subset \pi_1^{\text{\'et}}(X; \bar x_i, \bar x_j),$$ and hence gives an action on this subset.  As the quotient map of groupoids $$\{\pi_1^{\text{\'et}}(X_{\bar k}; \bar x_i, \bar x_j)\}\to \{\pi_1^{\text{\'et}, \ell}(X_{\bar k}; \bar x_i, \bar x_j)\}$$ is canonical, there is an induced action of $G_k$ on $\pi_1^{\text{\'et},\ell}(X_{\bar k}; \bar x_i, \bar x_j).$

The Hopf groupoid associated to the geometric pro-$\ell$ fundamental groupoid will be denoted $\mathbb{Z}_\ell[[\Pi_1(X_{\bar k})]]$; its objects are the rational or rational tangential basepoints $\{x_1, \cdots, x_n\}$ above, with morphisms between $x_i$ and $x_j$ given by $$\mathbb{Z}_\ell[[\pi_1^{\text{\'et},\ell}(X_{\bar k}; \bar x_i, \bar x_j)]].$$  As in subsection \ref{hopf-groupoids-section} or the appendix (Section \ref{hopf-groupoids}), we denote the filtration by powers of the augmentation ideal by $\{\mathscr{I}^n(x_i, x_j)\}$.  If $i=j$, we denote powers of the augmentation ideal by $\mathscr{I}^n(x_i)$.
\begin{prop}
Under the action of $G_k$ described above, the Galois module $\mathscr{I}(x_i, x_j)/\mathscr{I}^2(x_i,x_j)$ is canonically identified with $\pi_1^{\text{\'et}, \ell}(X_{\bar k}, \bar x_i)^{\text{ab}}$ through the map described in Lemma \ref{augmentation-abelianization-groupoid}.
\end{prop}
\begin{proof}
A direct computation shows that the map defined in Lemma \ref{augmentation-abelianization-groupoid} is Galois-equivariant.
\end{proof}
\begin{remark}\label{galois-action-remark}
If $X$ is geometrically connected, there is a canonical (Galois-equivariant) identification $$\text{Hom}(\pi_1^{\text{\'et}, \ell}(X_{\bar k}, \bar x_i), \mathbb{Z}_\ell)\simeq H^1(X_{\bar k}, \mathbb{Z}_\ell).$$  Thus if $\pi_1^{\text{\'et}, \ell}(X_{\bar k}, \bar x_i)^\text{ab}$ is torsion-free, the Galois action on $\mathscr{I}(x_i, x_j)/\mathscr{I}^2(x_i, x_j)$ is completely determined by the Galois action on $H^1(X_{\bar k}, \mathbb{Z}_\ell).$  Since the natural multiplication map $$(\mathscr{I}(x_i)/\mathscr{I}^2(x_i))^{\otimes {n-1}}\otimes (\mathscr{I}(x_i, x_j)/\mathscr{I}^2(x_i, x_j))\to \mathscr{I}^n(x_i, x_j)/\mathscr{I}^{n+1}(x_i, x_j)$$ is surjective and Galois-equivariant, knowledge of the Galois action on $H^1(X_{\bar k}, \mathbb{Z}_\ell)$ and a presentation of the fundamental group essentially determines the action of the Galois group on the graded module associated to the $\mathscr{I}$-adic filtration on the geometric pro-$\ell$ fundamental groupoid.
\end{remark}

\begin{example}\label{general-curve-example}
Let $X$ be a smooth, geometrically connected $k$-curve, and let $\overline{X}$ be the unique smooth, proper, geometrically connected $k$-curve containing $X$ as an open subscheme.  Let $\ell$ be  a prime different from $\on{char}(k)$ and let $g$ be the genus of $\overline{X}$, and let $d=\#|\overline{X}(\bar k)\setminus X(\bar k)|$. Then for any geometric point or geometric tangential basepoint $\bar x$ of $X$, $\pi_1^{\text{\'et}, \ell}(X_{\bar k}, \bar x)$ is the free pro-$\ell$ group on generators $$a_1, b_1, \cdots, a_g, b_g, c_1, c_2, \cdots, c_d$$ subject to the relation $$\left(\prod_{i=1}^g [a_i, b_i]\right)\left(\prod_{j=1}^d c_j\right)=1.$$  Suppose $d\geq 1$, so this group is in fact a free pro-$\ell$ group.  Thus the completed group algebra is isomorphic to the algebra of (non-commutative) power series in $2g+d-1$ variables.  So the multiplication map $$(\mathscr{I}(x_i)/\mathscr{I}^2(x_i))^{\otimes {n-1}}\otimes (\mathscr{I}(x_i, x_j)/\mathscr{I}^2(x_i, x_j))\to \mathscr{I}^n(x_i, x_j)/\mathscr{I}^{n+1}(x_i, x_j)$$ is an isomorphism for any $x_i, x_j$, so Remark \ref{galois-action-remark} completely determines the Galois action on the graded module associated to the $\mathscr{I}$-adic filtration on $\mathbb{Z}_\ell[[\pi_1^{\text{\'et}, \ell}(X_{\bar k}; \bar x_i, \bar x_j)]].$
\end{example}

\begin{example}
We make Example \ref{general-curve-example} explicit in the case that $X=\mathbb{P}^1\setminus\{0,1, \infty\}$.  In this case $$H^1(X_{\bar k}, \mathbb{Z}_\ell)\simeq \mathbb{Z}_\ell(-1)^{\oplus 2}$$ as a $G_k$-module.  Hence $$\mathscr{I}^n(x_i, x_j)/\mathscr{I}^{n+1}(x_i, x_j)\simeq \mathbb{Z}_\ell(n)^{\oplus 2^n}.$$
\end{example}
Having described the $G_k$-module structure on $\mathscr{I}^n(x_i, x_j)/\mathscr{I}^{n+1}(x_i, x_j)$ in terms of the $\ell$-adic cohomology of $X$, we spend the rest of the section describing studying the extension data $$0\to \mathscr{I}^{n+1}(x_i, x_j)/\mathscr{I}^{n+2}(x_i, x_j)\to \mathscr{I}^n(x_i, x_j)/\mathscr{I}^{n+2}(x_i, x_j)\to \mathscr{I}^n(x_i, x_j)/\mathscr{I}^{n+1}(x_i, x_j)\to 0,$$ where we view the above as an exact sequence of Galois modules. 
\subsection{The weight filtration}
In this section we define and record some basic facts about the weight filtration on $$\mathbb{Q}_\ell[[\Pi_1(X_{\bar k})]]:=\varprojlim_n \mathbb{Z}_\ell[[\Pi_1(X_{\bar k})]]/\mathscr{I}^n\otimes \mathbb{Q}_\ell.$$  
\begin{defn}
Suppose that $k$ is a finite field of characteristic prime to $\ell$, and let $\sigma\in G_k$ be the Frobenius element.  Let $\rho: G_k\to GL_n(\mathbb{Q}_\ell)$ be a continuous representation.  For $i\in \mathbb{Z}$, we say that $\rho$ is \emph{pure of weight $i$} if the eigenvalues of $\rho(\sigma)$ are algebraic integers with (archimedean) absolute value $q^{-i/2}$, for all embeddings $\overline{\mathbb{Q}}\to \mathbb{C}$.  We say it is \emph{mixed} if it is a direct sum of pure representations of any integer weight.

Now suppose that $k$ is a finitely generated field.  We say that a representation $\rho: G_k\to GL_n(\mathbb{Q}_\ell)$ is pure of weight $i$ if there exists a ring $R\subset k$, finitely generated over $\mathbb{Z}$, with $k=\on{Frac}(R)$, such that $\rho$ is unramified over $R$ (i.e. $\rho$ factors through the natural map $G_k\to \pi_1^{\text{\'et}}(\on{Spec}(R))$) and such that for all closed points $\mathfrak{p}$ of $\on{Spec}(R)$, the restriction of $\rho$ to $G_{\kappa(\mathfrak{p})}$ (which is well-defined up to conjugacy) is pure of weight $i$.  We say $\rho$ is mixed if it admits an increasing filtration $W^\bullet$ such that $W^i/W^{i-1}$ is pure of weight $i$; such a filtration is unique, and we call it the \emph{weight filtration}..
\end{defn}

We now define a filtration on $\mathbb{Q}_\ell[[\Pi_1(X_{\bar k})]]$, which is the weight filtration if $k$ is a finitely generated field.

We suppose $X$ is smooth with a simple normal crossings compactification $\overline{X}$.  As usual, we let $$\mathscr{I}\subset \mathbb{Q}_\ell[[\Pi_1(X_{\bar k})]]$$ be the augmentation ideal and we let $$\mathscr{J}\subset \mathbb{Q}_\ell[[\Pi_1(X_{\bar k})]]$$ be the kernel of the natural map $$\mathbb{Q}_\ell[[\Pi_1(X_{\bar k})]]\to\mathbb{Q}_\ell[[\Pi_1(\overline{X}_{\bar k})]]$$ induced by the inclusion $X\hookrightarrow \overline{X}$.  For $i\leq 0$, we set $$W^{-i}\mathbb{Q}_\ell[[\Pi_1(X_{\bar k})]]=\mathbb{Q}_\ell[[\Pi_1(X_{\bar k})]].$$ For $i>0$, we set $$W^{-i}\mathbb{Q}_\ell[[\Pi_1(X_{\bar k})]]=\mathscr{I}\cdot W^{-i+1}\mathbb{Q}_\ell[[\Pi_1(X_{\bar k})]]+\mathscr{J}\cdot W^{-i+2}\mathbb{Q}_\ell[[\Pi_1(X_{\bar k})]].$$  
\begin{prop}\label{weight-prop}
If $k$ is a finitely generated field, $$W^{-i}/W^{-i-1}$$ is a pure $G_k$-representation of weight $-i$.  Moreover, the $W$-adic topology equals the $\mathscr{I}$-adic topology: we have 
 $$W^{-1}=\mathscr{I}, W^{-2}=\mathscr{J}+\mathscr{I}^2,$$
$$\mathscr{I}^n\subset W^{-n}$$ and $$W^{-2n-1}\subset \mathscr{I}^n.$$
\end{prop}
\begin{proof}
The statements comparing the weight filtration to the $\mathscr{I}$-adic filtrations are immediate from the definitions.

To see that the former statement holds, observe that for $i
\geq -1$, this is immediate from the comparison of $\mathscr{I}/\mathscr{I}^2$ to $H^1(X_{\bar k}, \mathbb{Z}_\ell)$ of Remark \ref{galois-action-remark}.  In general, the follows from the fact that the weight filtration we've defined is multiplicative.
\end{proof}
\subsection{Varieties over finite and finitely generated fields}
We first consider the case where $k$ is a finite field.  Let $X/k$ be a geometrically connected variety. The goal of this section is to study the action of $\on{Gal}(\overline{k}/k)$ on $$\mathbb{Q}_\ell[[\Pi_1(X_{\bar k}]]:=\varprojlim_n \mathbb{Z}_\ell[[\Pi_1(X_{\bar k})]]/\mathscr{I}^n\otimes \mathbb{Q}_\ell.$$  The main result of this section is that the $G_k$-action on $$\overline{\mathbb{Q}_\ell}[[\Pi_1(X_{\bar k})]]/\mathscr{I}^n$$ is semisimple in broad generality (i.e. if $X$ is smooth and admits a smooth compactification with simple normal crossings boundary). Indeed, we prove that if $k$ is any field and $\sigma\in G_k$ is a \emph{quasi-Frobenius} (Definition \ref{quasi-frobenii-defn}), then its action on $\mathbb{Q}_\ell[[\Pi_1(X_{\bar k})]]$ is semi-simple.  In the next section we will analyze the extent to which the action on $$\overline{\mathbb{Z}_\ell}[[\Pi_1(X_{\bar k})]]/\mathscr{I}^n$$ fails to be semisimple.

We begin by studying the action of $G_k$ on $H^1(X, \mathbb{Q}_\ell)$.
\begin{lemma}\label{H1-semisimple}
Let $k$ be a finite field, let $X/k$ be a smooth, geometrically connected variety, and let $\overline{X}$ be a smooth compactification of $X$ with simple normal crossings boundary.  Then $G_k$ acts semi-simply on $$H^1(X_{\bar k}, \mathbb{Q}_\ell).$$  Furthermore, this $G_k$-representation is mixed of weights $1$ and $2$.
\end{lemma}
\begin{proof}
Let $j: X\to \overline{X}$ be the embedding.  Let $D_1, \cdots, D_n$ be the components of $\overline{X}\setminus X$.  Then the Leray spectral sequence for $Rj_*$ gives $$0\to H^1(\overline{X}_{\bar k}, \mathbb{Q}_\ell)\to H^1(X_{\bar k}, \mathbb{Q}_\ell)\to \bigoplus_{i\in \{1, \cdots, n\}} H^0(D_i, \mathbb{Q}_\ell)\otimes \mathbb{Q}_\ell(-1)\to H^2(\overline{X}_{\bar k}, \mathbb{Q}_\ell)\to \cdots$$
Now $H^1(\overline{X}_{\bar k}, \mathbb{Q}_\ell)$ is pure of weight $1$ by the Weil conjectures, and $G_k$ acts on it semisimply \cite{tate2}.  On the other hand, let $$V=\ker\left(\bigoplus_{i\in \{1, \cdots, n\}} H^0(D_i, \mathbb{Q}_\ell)\otimes \mathbb{Q}_\ell(-1)\to H^2(\overline{X}_{\bar k}, \mathbb{Q}_\ell)\right).$$  $V$ is manifestly pure of weight $2$, with semisimple $G_k$-action.  But the short exact sequence $$0\to  H^1(\overline{X}_{\bar k}, \mathbb{Q}_\ell)\to H^1(X_{\bar k}, \mathbb{Q}_\ell)\to V\to 0$$ splits (canonically), as the first and last term have different weights.
\end{proof}
The main theorem describing the action of $G_k$ on $\overline{\mathbb{Q}_\ell}[[\Pi_1(X_{\bar k})]]/\mathscr{I}^n$ is a ``non-Abelian" generalization of the above lemma.  Before stating it, we axiomatize the properties of the Frobenius action on $H^1(X_{\bar k}, \mathbb{Q}_\ell)$, as we will later apply these results to elements of $G_k$ which are not Frobenii.
\begin{defn}[Quasi-Frobenii]\label{quasi-frobenii-defn}
Let $(V, F^\bullet)$ be a  vector space over a field $k$ with an increasing filtration indexed by $\mathbb{Z}$, and $\sigma: V\to V$ a semisimple automorphism preserving the filtration.  Let $(\on{Sym}^*(V), F^\bullet)$ be the induced filtration on the symmetric algebra of $V$.  Then we say that $\sigma$ is a \emph{quasi-Frobenius} if the set of $\sigma$ equivariant morphisms $$\on{Hom}_{\sigma}(\on{gr}^i(\on{Sym}^*(V)), \on{gr}^j(\on{Sym}^*(V))=0$$ for all $i\not=j$.  If $\lambda\in \bar k$ is an eigenvalue of $\sigma$ acting on $\on{gr}^i(\on{Sym}^*(V))\otimes \bar k$, we say that $\lambda$ has \emph{weight} $i$; likewise, if $v$ is an eigenvector with eigenvalue $\lambda$, we say that $v$ has \emph{weight} $i$.
\end{defn}
In other words, if $\sigma\in GL(V)$ is a quasi-Frobenius with respect to a filtration $F^\bullet$, no $\sigma$-eigenvalue appears in two different associated-graded pieces of the filtration induced on $\on{Sym}^*(V)$.  Concretely, this means that if $S_i$ is the set of eigenvalues of $\sigma$ appearing in $\on{gr}^iV$, we have that $$\prod_i \prod_{\lambda\in S_i} \lambda^{a_{\lambda, i}}\not=\prod_i \prod_{\lambda\in S_i} \lambda^{b_{\lambda, i}}$$ unless $$\sum ia_{\lambda, i}=\sum ib_{\lambda, i}.$$
We say that the monomial $$\prod_i \prod_{\lambda\in S_i} \lambda^{a_{\lambda, i}}$$ has weight $$\sum ia_{\lambda, i}.$$
\begin{example}
The motivating example for this definition is the following:  Let $k$ be a finite field of cardinality $q$, $X/k$ a variety, and $\ell$ a prime different from the characteristic of $k$.  Then the Frobenius automorphism of $$(H^*(X_{\bar k}, \mathbb{Q}_\ell), W^\bullet)$$ is a quasi-Frobenius, where $W^\bullet$ is the weight filtration, as long as Frobenius acts semisimply on $H^*(X_{\bar k}, \mathbb{Q}_\ell)$.  Indeed, the eigenvalues of Frobenius on $\on{gr}^iH^*(X_{\bar k}, \mathbb{Q}_\ell)$ are $q$-Weil numbers of weight $i$ (by the definition of the weight filtration), so there can be no Frobenius-equivariant maps between pieces of $H^*$ with distinct weights.

In particular, by Lemma \ref{H1-semisimple}, the Frobenius automorphism of $H^1(X_{\bar k}, \mathbb{Q}_\ell)$, with $X$ smooth, geometrically connected, and admitting a simple normal crossings compactification, is a quasi-Frobenius relative to the weight filtration.
\end{example}

\begin{theorem}\label{semisimple-pi1}
Let $k$ be a field, let $X/k$ be a smooth, geometrically connected variety, and let $\overline{X}$ be a smooth compactification of $X$ with simple normal crossings boundary.  Let $x_i, x_j$ be a rational points or rational tangential basepoints of $X$.  Suppose that $\sigma\in G_k$ is a quasi-Frobenius for $H^1(X_{\bar k}, \mathbb{Q}_\ell)$, relative to the weight filtration.  Then $\sigma$ acts semi-simply on $$\overline{\mathbb{Q}_\ell}[[\pi_1^{\text{\'et}, \ell}(X_{\bar k}; \bar x_i, \bar x_j)]]/\mathscr{I}^n(x_i, x_j).$$
\end{theorem} 

\begin{remark}
Theorem \ref{semisimple-pi1} is trivial if $H^1(X_{\bar k}, \mathbb{Q}_\ell)$ is pure, e.g.~if $X$ is proper or if $H^1(\overline{X}_{\bar k}, \mathbb{Q}_\ell)=0$.  Indeed, in this case, let $\iota$ be the weight of $H^1(X_{\bar k}, \mathbb{Q}_\ell)$.  Then $$\mathscr{I}(x_i, x_j)/\mathscr{I}^2(x_i, x_j)$$ is pure of weight $-\iota$, and hence $$\mathscr{I}^n(x_i, x_j)/\mathscr{I}^{n+1}(x_i, x_j)$$ is pure of weight $-n\iota$ (see Remark \ref{galois-action-remark}).  Thus the short exact sequences of $\sigma$-modules $$0\to \mathscr{I}^{n+1}(x_i, x_j)/\mathscr{I}^{n+2}(x_i, x_j)\to \mathscr{I}^n(x_i, x_j)/\mathscr{I}^{n+2}(x_i, x_j)\to \mathscr{I}^n(x_i, x_j)/\mathscr{I}^{n+1}(x_i, x_j)\to 0$$ split $\sigma$-equivariantly.  But $\sigma$ acts semisimply on each $$\mathscr{I}^{n}(x_i, x_j)/\mathscr{I}^{n+1}(x_i, x_j),$$ as it does so on $$\mathscr{I}(x_i, x_j)/\mathscr{I}^2(x_i, x_j)=H^1(X_{\bar k}, \mathbb{Q}_\ell)^\vee$$ by assumption, as each of these $\sigma$-modules is a quotient of a tensor power of $\mathscr{I}/\mathscr{I}^2$.
\end{remark}

Before beginning the proof of Theorem \ref{semisimple-pi1}, we need to introduce a new notion.  
\begin{propconstr}[Canonical Paths]\label{canonical-paths-prop-constr}
Let $k$ be a field, let $X/k$ be a smooth, geometrically connected variety, and let $\overline{X}$ be a smooth compactification of $X$ with simple normal crossings boundary.  Let $x_i, x_j$ be a rational points or rational tangential basepoints of $X$.  Suppose that $\sigma\in G_k$ acts as a quasi-Frobenius on $H^1(X_{\bar k}, \mathbb{Q}_\ell)$ relative to the weight filtration.  Then there is a unique element $$p_{i,j}\in \mathbb{Q}_\ell[[\pi_1^{\text{\'et}, \ell}(X_{\bar k}; \bar x_i, \bar x_j)]]$$ such that 
\begin{enumerate}
\item $p_{i,j}$ is $\sigma$-invariant, and
\item $\epsilon(p_{i,j})=1$.  
\end{enumerate}
(Here $\epsilon: \mathbb{Q}_\ell[[\pi_1^{\text{\'et}, \ell}(X_{\bar k}; \bar x_i, \bar x_j)]]\to \mathbb{Q}_\ell$ is the augmentation map.)
\end{propconstr}
\begin{proof}
As $\mathbb{Q}_\ell[[\pi_1^{\text{\'et}, \ell}(X_{\bar k}; \bar x_i, \bar x_j)]]$ is defined so that $$\mathbb{Q}_\ell[[\pi_1^{\text{\'et}, \ell}(X_{\bar k}; \bar x_i, \bar x_j)]]=\varprojlim_n \mathbb{Q}_\ell[[\pi_1^{\text{\'et}, \ell}(X_{\bar k}; \bar x_i, \bar x_j)]]/\mathscr{I}^n(x_i, x_j),$$ it is enough to show that there are unique  elements $$p_{i,j}^n\in \mathbb{Q}_\ell[[\pi_1^{\text{\'et}, \ell}(X_{\bar k}; \bar x_i, \bar x_j)]]/\mathscr{I}^n(x_i, x_j)$$ satisfying the desired conditions; that is, we wish to find a canonical $\sigma$-equivariant splitting of the augmentation map $$\epsilon_n: \mathbb{Q}_\ell[[\pi_1^{\text{\'et}, \ell}(X_{\bar k}; \bar x_i, \bar x_j)]]/\mathscr{I}^n(x_i, x_j)\to \mathbb{Q}_\ell.$$  But the kernel of this map is $$\mathscr{I}(x_i, x_j)/\mathscr{I}^n(x_i, x_j).$$  By Lemma \ref{H1-semisimple}, the weights of $$\mathscr{I}(x_i, x_j)/\mathscr{I}^2(x_i, x_j)$$ are in the interval $[-2, -1]$; hence the weights of $$\mathscr{I}^r(x_i, x_j)/\mathscr{I}^{r+1}(x_i, x_j)$$ are in the interval $[-r, -2r]$.  Thus the kernel of the augmentation map $\epsilon_n$ has strictly negative weight.  But the target of the augmentation map has weight zero; hence the augmentation map splits uniquely, as desired.
\end{proof}
We call the element $p_{i,j}$ above \emph{the $\sigma$-canonical path} between $x_i$ and $x_j$.
\begin{remark}
If $x_i=x_j$, $p_{i,j}$ is just the identity of $\mathbb{Q}_\ell[[\pi_1^{\text{\'et}, \ell}(X_{\bar k}, \bar x_i)]].$  Thus $$p_{i,j}=p_{i,j}^{-1}$$ $$p_{i,j}\circ p_{j,k}=p_{i,k}$$ (by the unicity of the canonical paths, and Galois-equivariance of the composition maps).  Moreover, it is not hard to see that the $p_{i,j}$ are \emph{group-like} in the sense of Definition \ref{grouplike}.
\end{remark}
\begin{proof}[Proof of Theorem \ref{semisimple-pi1}]
We first observe that it suffices to prove Theorem \ref{semisimple-pi1} in the case $x_i=x_j$.  Indeed, composition with $p_{i,j}$ induces a canonical $\sigma$-equivariant isomorphism $$\mathbb{Q}_\ell[[\pi_1^{\text{\'et}, \ell}(X_{\bar k}, \bar x_i)]]/\mathscr{I}^n(x_i)\to \mathbb{Q}_\ell[[\pi_1^{\text{\'et}, \ell}(X_{\bar k}; \bar x_i, \bar x_j)]]/\mathscr{I}^n(x_i, x_j).$$  Thus without loss of generality, $i=j$, and we denote $x_i$ by simply $x$.  Moreover, it is harmless to extend scalars to $\overline{\mathbb{Q}_\ell}$ everywhere.

Next, we claim it suffices to find a $\sigma$-equivariant splitting $s$ of the natural map $$\mathscr{I}(x)/\mathscr{I}^n(x)\to \mathscr{I}(x)/\mathscr{I}^2(x).$$

Indeed, $\sigma$ acts semisimply on $\mathscr{I}(x)/\mathscr{I}^2(x)$ by assumption, so we may choose a basis of eigenvectors $\{v_1, \cdots, v_n\}$.  Then the monomials in $s(v_i)$ give a basis of eigenvectors for $$\mathbb{Q}_\ell[[\pi_1^{\text{\'et}, \ell}(X_{\bar k}, \bar x)]]/\mathscr{I}^n(x),$$ which proves semisimplicity.

We now construct a splitting $$s:\mathscr{I}(x)/\mathscr{I}^2(x)\to \mathscr{I}(x)/\mathscr{I}^n(x)$$ as desired above.  Observe that as $\sigma$ is a quasi-Frobenius relative to the weight filtration, we have a  splitting of the weight filtration $$\mathscr{I}(x)/\mathscr{I}^2(x)=W_1\oplus W_2$$ canonically where $W_1$ has weight $-1$ and $W_2$ has weight $-2$, using the eigenspace decomposition of $\sigma$.  The projection $$\mathscr{I}(x)/\mathscr{I}^n(x)\to \mathscr{I}(x)/\mathscr{I}^2(x)\to W_1$$ has kernel with weights in $[-2n, -2]$ (again by Lemma \ref{H1-semisimple}), so there is a canonical splitting $$s_1: W_1\to \mathscr{I}(x)/\mathscr{I}^n(x).$$  It remains to construct a splitting $s_2: W_2\to \mathscr{I}(x)/\mathscr{I}^n(x)$; if $W_2=0$, the proof is complete.

Let $D_1, \cdots, D_n$ be the components of the simple normal crossings boundary $\overline{X}\setminus X$.  Let $j: X\to \overline{X}$ be the embedding, and recall that the Leray spectral sequence for $Rj_*$ gives an exact sequence
$$0\to H^1(\overline{X}_{\bar k}, \mathbb{Q}_\ell)\to H^1(X_{\bar k}, \mathbb{Q}_\ell)\to \bigoplus_{i\in \{1, \cdots, n\}} H^0(D_i, \mathbb{Q}_\ell)\otimes \mathbb{Q}_\ell(-1)\to \cdots.$$

 Dualizing gives an exact sequence $$\bigoplus_{i\in \{1, \cdots, n\}}H^0(D_i, \mathbb{Q}_\ell)^{\vee}\otimes \mathbb{Q}_\ell(1)\overset{\iota}{\to} \mathscr{I}(x)/\mathscr{I}^2(x)\to H^1(\overline{X}_{\bar k}, \mathbb{Q}_\ell)^\vee\to 0.$$  Comparing weights gives that $W_2$ is precisely the image of $\iota$ above.  Thus it suffices to lift $\iota$ to a map into $\mathscr{I}(x)/\mathscr{I}^n(x)$.  
 
See Figure \ref{weight2-schematic} for a schematic depiction of this lift.  The statement of the theorem is insensitive to replacing $k$ with a finite extension (by replacing $\sigma$ with a power), so we may assume that each $$D_i\setminus \bigcup_{j\not=i} D_j$$ contains a rational point $y_i$.  Let $y_i':\on{Spec}(k[[t]])\to \overline{X}$ be the germ of a curve in $\overline{X}$ with the closed point mapping to $y_i$ and transverse to $D_i$.  Let $\tilde y_i: \on{Spec}(k((t)))\to X$ be the associated rational tangential basepoint, and $\overline{\tilde{y_i}}$ a geometric point obtained by choosing an embedding $k((t))\hookrightarrow \overline{k((t))}$.  In Section \ref{inertia-section}, we constructed natural Galois-equivariant maps $$\iota_{\tilde y_i}: \mathbb{Z}_\ell(1)\to \pi_1^{\text{\'et}, \ell}(X_{\bar k}, \overline{\tilde y_i}).$$  Thus there are natural maps $$(\iota_{\tilde y_i}-1)\otimes \mathbb{Q}_\ell: \mathbb{Q}_\ell(1)\to \mathscr{I}(\tilde y_i)/\mathscr{I}^2(\tilde y_i).$$  Letting $r_i: \mathscr{I}(\tilde y_i)/\mathscr{I}^2(\tilde y_i)\to \mathscr{I}(x)/\mathscr{I}^2(x)$ be the canonical identification from Lemma \ref{augmentation-abelianization-groupoid}, (that is, the identification induced from conjugation by any element of $\pi_1^{\text{\'et}, \ell}(X_{\bar k}; \bar x, \overline{\tilde y_i})$), we obtain a map $$\bigoplus_{i\in \{1, \cdots, n\}} r_i \circ (\iota_{\tilde y_i}-1)\otimes \mathbb{Q}_\ell: \bigoplus_{i\in\{1, \cdots, n\}}\mathbb{Q}_\ell(1)\to \mathscr{I}(x)/\mathscr{I}^2(x).$$  The image of this map is may be canonically identified with the image of $\iota$ above; we rename this map $\iota'$.  
\begin{figure}[h]
\centering{
\resizebox{130mm}{!}{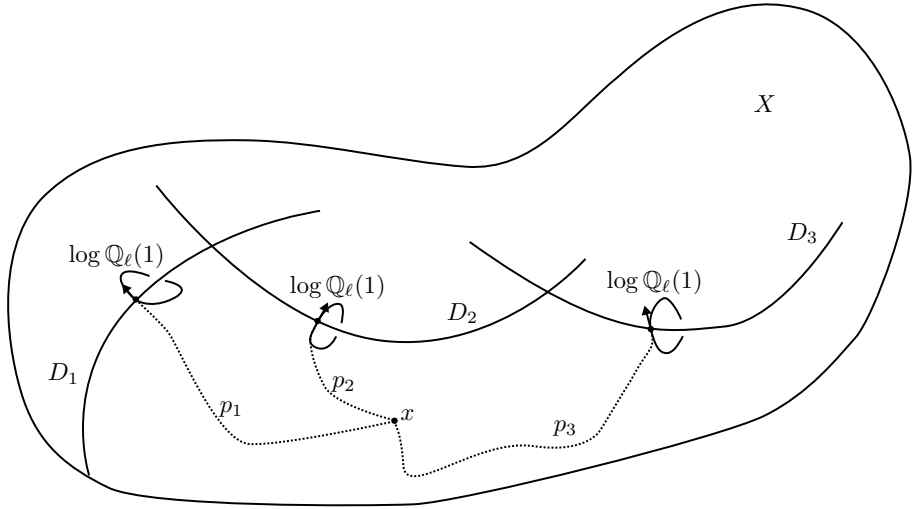}
\caption{A schematic depiction of the lift of the weight $-2$ part of $\mathscr{I}(x)/\mathscr{I}^2(x)$ to $\mathbb{Q}_\ell[[\pi_1^{\text{\'et}, \ell}(X_{\bar k}, \bar x)]]$}\label{weight2-schematic}
}
\end{figure} 
 
 Now it suffices to lift $\iota'$ to a map into $\mathscr{I}(x)/\mathscr{I}^n(x)$.  Let $p_i\in \mathbb{Q}_\ell[[\pi_1^{\text{\'et}, \ell}(X_{\bar k}; \bar x, \overline{\tilde y_i})]]$ be the $\sigma$-canonical path.  Then the map $$s_i: \mathbb{Q}_\ell(1)\to \mathscr{I}(x)/\mathscr{I}^n(x)$$ defined by $$s_i: \alpha\mapsto p_i^{-1}\log \iota_{\tilde y_i}(\alpha)p_i$$
 is $\sigma$-equivariant, and the map $$\bigoplus_{i\in \{1, \cdots, n\}} s_i: \bigoplus \mathbb{Q}_\ell(1)\to \mathscr{I}(x)/\mathscr{I}^n(x)$$ lifts $\iota'$, as desired.  Here $\log$ is defined by the usual power series $$\log(1+x)=\sum (-1)^{n+1} \frac{x^n}{n},$$ which converges as $\iota_{\tilde{ y}}(\alpha)-1$ is in the augmentation ideal of $\mathbb{Q}_\ell[[\pi_1^{\text{\'et}, \ell}(X_{\bar k}, \bar x)]].$
\end{proof}
\begin{corollary}\label{splitting-I-adic-filtration}
Let $X/k$ and $\sigma$ be as in Theorem \ref{semisimple-pi1}.  Then there is a $\sigma$-equivariant isomorphism $$\overline{\mathbb{Q}_\ell}[[\pi_1^{\text{\'et}, \ell}(X_{\bar k}; \bar x_i, \bar x_j)]]/\mathscr{I}^n(x_i, x_j)\simeq \bigoplus_{r=0}^{n-1} \mathscr{I}^r(x_i, x_j)/\mathscr{I}^{r+1}(x_i, x_j)\otimes \overline{\mathbb{Q}_\ell}.$$ 
\end{corollary}
\begin{proof}
Let $\{\bar v_r\}$ be a basis of $\sigma$-eigenvectors for $\mathscr{I}(x_i, x_j)/\mathscr{I}^2(x_i, x_j),$ and let $\{v_r\}$ be lifts to $\sigma$-eigenvectors in $\mathscr{I}(x_i, x_j)/\mathscr{I}^n(x_i, x_j)$.  Let $$p_{i, j}\in \overline{\mathbb{Q}_\ell}[[\pi_1^{\text{\'et}, \ell}(X_{\bar k}; \bar x_i, \bar x_j)]]/\mathscr{I}^n(x_i, x_j)$$ be the image of the $\sigma$-canonical path.  Then each $v_rp_{i,j}^{-1}$ is a $\sigma$-eigenvector in $\overline{\mathbb{Q}_\ell}[[\pi_1^{\text{\'et}, \ell}(X_{\bar k}, \bar x_i)]]/\mathscr{I}^n(x_i)$.  Then if $m$ is a monomial in the $v_rp_{i,j}^{-1}$, we have that $p_{i,j}m$ is a $\sigma$-eigenvector in $$\overline{\mathbb{Q}_\ell}[[\pi_1^{\text{\'et}, \ell}(X_{\bar k}; \bar x_i, \bar x_j)]]/\mathscr{I}^n(x_i, x_j);$$ as we range over all monomials $m$, this gives a diagonal basis for the action of $\sigma$.  If $m$ is a monomial of degree $r$, then $p_{i,j}m\in \mathscr{I}^r(x_i, x_j)$.  Let $\overline{m}$ denote its image in $\mathscr{I}^r(x_i, x_j)/\mathscr{I}^{r+1}(x_i, x_j)$.  Now the desired isomorphism is induced by the map sending $p_{i,j}m$ to $\overline{m}$.
\end{proof}
\subsection{Quasi-scalars}
The purpose of this subsection is to produce certain elements of $G_k$ which are quasi-Frobenii for $\mathbb{Q}_\ell[[\pi_1^{\text{\'et}, \ell}(X_{\bar k}; \bar x_i, \bar x_j)]]$ relative to the weight filtration; moreover, they act by scalars on associated graded pieces of the weight filtration.  The ideas in this section are an adaptation of \cite[Corollaire 2]{bogomolov} and were suggested to the author by Will Sawin \cite{sawin} after an earlier version of this manuscript was circulated.

\begin{defn}[Quasi-Scalar]
Let $(V, F^\bullet)$ be a vector space with an increasing filtration, and let $\sigma$ be an automorphism of $V$ preserving $F^\bullet$.  Then we say that $\sigma$ is a \emph{quasi-scalar} if there exists a non-zero, non-root-of-unity scalar $\alpha$ such that $\alpha$ acts on $\on{gr}^i(V)$ via $\alpha^i$.  We call $\alpha$ the \emph{degree} of the quasi-scalar $\sigma$.
\end{defn}

\begin{remark}
Note that all quasi-scalars are quasi-Frobenii. 
\end{remark}

\begin{example}
The motivation for this definition again comes from the Frobenius action on the weight filtration, in this case in the mixed Tate setting.  Suppose that $k$ is a finite field and $X/k$ a variety such that $H^*(X_{\bar k}, \mathbb{Q}_\ell)$ is mixed Tate.  Then if Frobenius acts semisimply on $H^*(X_{\bar k}, \mathbb{Q}_\ell)$, some power of Frobenius acts a quasiscalar on $H^*$ relative to the weight filtration. 
\end{example}

We now show that in many cases (i.e. if $X$ is smooth over a finitely generated field $k$ of characteristic zero) there exist $\sigma \in G_k$ such that $\sigma$ acts on $H^1(X_{\bar k}, \mathbb{Q}_\ell)$ via a quasi-scalar, relative to the weight filtration.  It will follow that $\sigma$ acts on $\mathbb{Q}_\ell[[\Pi_1(X_{\bar k})]]$ via a quasi-scalar as well.

\begin{theorem}\label{H1-quasi-scalar}
Let $k$ be a finitely generated field of characteristic zero, and let $X/k$ be a smooth, geometrically connected variety. Then there for all $\alpha\in \mathbb{Z}_\ell^\times$ sufficiently close to $1$, there exists $\sigma\in G_k$ such that $\sigma$ acts on $H^1(X_{\bar k}, \mathbb{Q}_\ell)$ as a quasi-scalar of degree $\alpha$ with respect to the weight filtration.
\end{theorem}

Before giving the proof, we observe the following corollary:

\begin{corollary}\label{pi1-quasi-scalar}
Let $k$ be a finitely generated field of characteristic zero, and let $X/k$ be a smooth, geometrically connected variety.   Then there for all $\alpha\in \mathbb{Z}_\ell^\times$ sufficiently close to $1$, there exists $\sigma\in G_k$ such that $\sigma$ acts on $\mathbb{Q}_\ell[[\pi_1^{\text{\'et}, \ell}(X_{\bar k}; \bar x_i, \bar x_j)]]$ as a quasi-scalar of degree $\alpha$ with respect to the weight filtration, for any two rational points or rational tangential basepoints $x_i, x_j$.
\end{corollary}
\begin{proof}
For notational simplicity we assume $x_i=x_j$; the proof is identical in the case the two basepoints are distinct.

Let $\sigma$ be as in Theorem \ref{H1-quasi-scalar}, i.e. an element of $G_k$ acting via a quasi-scalar of degree $\alpha$ on $H^1(X_{\bar k}, \mathbb{Q}_\ell)$.  We claim that $\sigma$ also acts via a quasi-scalar of degree $\alpha$ on $\mathbb{Q}_\ell[[\pi_1^{\text{\'et}, \ell}(X_{\bar k}, \bar x)]]$.

Indeed, $\sigma$ is a quasi-Frobenius relative to the weight filtration on $H^1$; hence by Theorem \ref{semisimple-pi1}, $\sigma$ acts semisimply on $\mathbb{Q}_\ell[[\pi_1^{\text{\'et}, \ell}(X_{\bar k}, \bar x)]]/\mathscr{I}^n$ for all $n$.  Thus there exists a unique $\sigma$-equivariant splitting $s$ of the map $$\mathscr{I}\to \mathscr{I}/\mathscr{I}^2;$$  Let $$T^*(s): T^*(\mathscr{I}/\mathscr{I}^2)\to \mathbb{Q}_\ell[[\pi_1^{\text{\'et}, \ell}(X_{\bar k}, \bar x)]]$$ be the induced $\sigma$-equivariant map from the tensor algebra on $\mathscr{I}/\mathscr{I}^2$ to the $\mathbb{Q}_\ell$-group algebra.  This map is surjective, respects the weight filtration, and $\sigma$ acts on the tensor algebra via a quasi-scalar of degree $\alpha$, so the proof is complete.
\end{proof}

We now prove the theorem.

\begin{proof}[Proof of Theorem \ref{H1-quasi-scalar}]
Let $V=H^1(X_{\bar k}, \mathbb{Q}_\ell)$, and let $n=\dim(V)$.  Let $G\subset GL(V)$ be the Zariski closure of the image of the natural map $$\rho: G_k\to GL(V).$$  We first show that there exists a quasi-scalar in $G(\mathbb{Q}_\ell)$ adapted to the weight filtration on $V$.  We let $\overline{X}$ be a compactification of $X$ with simple normal crossings boundary $D$; spread the pair $(\overline{X}, D)$ out over a finite-type $\mathbb{Z}$-algebra $R$.  Let $\mathfrak{p}\in \on{Spec}(R)$ be a closed point at which the pair $(\overline{X}, D)$ has good reduction, i.e. $\overline{X}_{\mathfrak{p}}$ is smooth and $D_{\mathfrak{p}}$ is a simple normal crossings divisor.  Let $q=\#|R/\mathfrak{p}|$.

We let $F\in G_k$ be a Frobenius at $\mathfrak{p}$.  By Lemma \ref{H1-semisimple}, there exists some finite extension $L\supset \mathbb{Q}_\ell$ so that $F$ acts diagonalizably on $V_L=H^1(X_{\bar k}, L).$  Let $\{v_i\}\subset V_L$ be a diagonal basis for $F$, with eigenvalues $\lambda_i$, so that $$F(v_i)=\lambda_i\cdot v_i.$$  Let $T\subset G_L$ be the Zariski-closure of $\{F^n\}$.  We now describe $T$ in terms of the $\{\lambda_i\}$.  Let $S \subset \mathbb{Z}^{n}$ be the set of $n$-tuples of integers $(\alpha_1, \cdots, \alpha_n)$ such that $$\prod \lambda_i^{\alpha_i}=1.$$  Then in the basis $\{v_i\}$, $T$ is the subgroup scheme of $GL(V)$ consisting of diagonal matrices
$$\begin{pmatrix}
\gamma_1 & 0 & 0 &\cdots & 0\\
0 & \gamma_2 & 0 & \ddots & 0\\
0 & 0 & \gamma_3 & \ddots & 0\\
\vdots & \ddots & \ddots &  \ddots & \vdots\\
0 & 0 & 0 & \cdots & \gamma_n
\end{pmatrix}$$
such that for $(\alpha_1, \cdots, \alpha_n)\in S,$  $$\prod \gamma_i^{\alpha_i}=1.$$   (That is, the $\gamma_i$ satisfy the same multiplicative relations as the $\lambda_i$.)  In particular, we have that if $E/\mathbb{Q}$ is the (finite) extension of $\mathbb{Q}$ generated by the $\lambda_i$, then for any $(\alpha_1, \cdots, \alpha_n)\in S,$  $$\prod \on{Nm}_{E/\mathbb{Q}}(\lambda_i)^{\alpha_i}=1$$  (as the norm map is a multiplicative homomorphism), and hence the matrix 
$$A:=\begin{pmatrix}
\on{Nm}_{E/\mathbb{Q}}(\lambda_1) & 0 & 0 &\cdots & 0\\
0 & \on{Nm}_{E/\mathbb{Q}}(\lambda_2)  & 0 & \ddots & 0\\
0 & 0 & \on{Nm}_{E/\mathbb{Q}}(\lambda_3)  & \ddots & 0\\
\vdots & \ddots & \ddots &  \ddots & \vdots\\
0 & 0 & 0 & \cdots & \on{Nm}_{E/\mathbb{Q}}(\lambda_n) 
\end{pmatrix}$$
is in $T$.  Then some some power $A^N$ of this matrix is a quasi-scalar in $T(L)$, adapted to the weight filtration.  Observe that $A$ has infinite order, as the diagonal entries are norms of $q$-Weil integers. 

But this matrix in fact lies in $$T(L)\cap GL(V)(\mathbb{Q}_\ell)\subset GL(V)(L).$$  Indeed, $A^N$ is in $T(L)$ by definition.  The choice of Frobenius $F$ splits the weight filtration on $V$ over $\mathbb{Q}_\ell$ (into generalized eigenspaces).  Let $$V=W_1\oplus W_2$$ be this splitting.  As $A^N$ commutes with $F$ by definition, $A^N$ is just the operator which acts by one scalar on $W_1$ and another on $W_2$; hence it lies in $GL(V)(\mathbb{Q}_\ell)$ as desired.  Hence $A^N$ in fact lies in $G(\mathbb{Q}_\ell)$.

Now, we produce the desired $\sigma\in G_k$.  Let $V=W_1\oplus W_2$ be the splitting defined previously, and consider the subtorus $T'$ of $G$ consisting of matrices in $G$ which act by a scalar $\alpha$ on $W_1$ and by the scalar $\alpha^2$ on $W_2$.  We have just shown that this subtorus $T'$ of $G$ is positive-dimensional, as it contains an element of infinite order.  But by the main result of \cite{bogomolov} (which in fact holds for any finitely generated field of characteristic zero --- see e.g. Serre's argument \cite[Letter to Ribet, 1/1/1981]{serre} for a reduction to the number field case proven in \cite{bogomolov}), the image of $G_k$ in $G(\mathbb{Q}_\ell)$ is open.  Hence, in particular, $\on{im}(G_k)\cap T'(\mathbb{Q}_\ell)$ is non-empty and open, proving the theorem.
\end{proof}
\begin{remark}\label{bound-observation-1}
Observe that the degree of the quasi-scalar we construct depends only on the index of the image of the representation $$G_k\to GL(H^1(X_{\bar k}, \mathbb{Z}_\ell))$$ in its Zariski closure.  In particular, if one believes standard adelic open image conjectures for Galois representations, one may find, for almost all $\ell$, a quasi-scalar of degree $q$, where $q$ is a topological generator of $\mathbb{Z}_\ell^\times$.
\end{remark}

\section{Integral $\ell$-adic periods}\label{integral-l-adic-periods-section}
Let $k$ be a finitely generated field, $X/k$ a smooth variety with simple normal crossings compactification, and $x$ a rational point or rational tangential basepoint of $X$.  We have shown (in Theorem \ref{semisimple-pi1}) that if $\sigma\in G_k$ is a quasi-Frobenius for $H^1(X_{\bar k}, \mathbb{Q}_\ell)$, then the weight filtration on $\mathbb{Q}_\ell[[\pi_1^{\text{\'et}, \ell}(X_{\bar k}, \bar x)]]$ splits $\sigma$-equivariantly.  However, the weight filtration of  $\mathbb{Z}_\ell[[\pi_1^{\text{\'et}, \ell}(X_{\bar k}, \bar x)]]$ does not split $\sigma$-equivariantly---the purpose of this section is to study its failure to do so.
\subsection{Generalities on integral $\ell$-adic periods}\label{integral-l-adic-generalities}
Combining Remark \ref{galois-action-remark} and Corollary \ref{splitting-I-adic-filtration}, we have a complete description of the action of a quasi-Frobenius $\sigma \in G_k$ on $$\overline{\mathbb{Q}_\ell}[[\Pi_1(X_{\bar k})]],$$ for $X$ a smooth variety with simple normal crossings compactification.  We now analyze the situation with $\overline{\mathbb{Q}_\ell}$ replaced with $\overline{\mathbb{Z}_\ell}$ (the normalization of $\mathbb{Z}_\ell$ in $\overline{\mathbb{Q}_\ell}$).  Here, the situation is more subtle.  In the previous section, we heavily used the fact that $\mathbb{Q}_\ell[\sigma]$-modules split into generalized eigenspaces; this is no longer true for $\mathbb{Z}_\ell[\sigma]$-modules.

\begin{example}\label{non-split-sequence}
Let $k$ be a positive integer.  Consider the following short exact sequence of $\mathbb{Z}_\ell$-modules with endomorphism:
$$\xymatrixcolsep{5pc}\xymatrixrowsep{6pc}\xymatrix{
0 \ar[r] & \mathbb{Z}_\ell \ar[r]^{\begin{pmatrix} 1 \\ 0\end{pmatrix}} \ar[d]^{\cdot 1} & \mathbb{Z}_\ell^2 \ar[r]^{\begin{pmatrix} 0 & 1\end{pmatrix}} \ar[d]^{\begin{pmatrix} 1 & 1 \\ 0 & 1+\ell^r \end{pmatrix}} & \mathbb{Z}_\ell \ar[r] \ar[d]^{\cdot (1+\ell^r)} & 0\\
0 \ar[r] & \mathbb{Z}_\ell \ar[r]^{\begin{pmatrix} 1 \\ 0\end{pmatrix}}  & \mathbb{Z}_\ell^2 \ar[r]^{\begin{pmatrix} 0 & 1\end{pmatrix}}  & \mathbb{Z}_\ell \ar[r]  & 0.
}$$
This sequence does not split; on tensoring with $\mathbb{Q}_\ell$ it does split, via the map $$\begin{pmatrix} \frac{1}{\ell^r} \\ 1 \end{pmatrix}: \mathbb{Q}_\ell\to\mathbb{Q}_\ell^2.$$
\end{example}
\begin{defn}\label{integral-periods-defn}
Let $$0\to W\to V\overset{\pi}{\to} V/W\to 0$$ be a short exact sequence of free $\overline{\mathbb{Z}_\ell}$-modules, and let $F: V\to V$ be an endomorphism preserving $W$.  Then we define the \emph{integral $\ell$-adic period} $b_{V, W}$ of the triple $(V, W, F)$ to be the infimum among all $v_\ell(a)\in \mathbb{Q}$, $a\in \overline{\mathbb{Z}_\ell}$ such that there exists an $F$-equivariant map $s: a \cdot V/W\to V$ with $\pi\circ s=\on{id}$, or $\infty$ if no such $a$ exists.  That is, we are taking the minimum valuation of $a$ such that there exists a commutative diagram 
$$\xymatrix{
 		&		&				& V/W \ar[d]^{\cdot a} \ar[dl]_s &\\
0\ar[r] & W\ar[r] & V\ar[r]^\pi & V/W \ar[r] & 0.
}$$
\end{defn}
For example, the integral $\ell$-adic period of the sequence in Example \ref{non-split-sequence} is $r$.  For any split sequence of modules with endomorphism, the integral $\ell$-adic period is zero. 

The point of this definition is that if $b_{V,W}$ is finite, a choice of $s$ as Definition \ref{integral-periods-defn} induces an isomorphism $$\tilde s: (W\oplus V/W)\otimes \overline{\mathbb{Q}_\ell}\overset{\sim}{\to}V\otimes \overline{\mathbb{Q}_\ell}.$$ The matrix entries defining this isomorphism with respect to  $\overline{\mathbb{Z}_\ell}$-bases of both sides have valuation bounded below by $-b_{V, W}$.  We view $\tilde s$ as a kind of ``comparison isomorphism" analogous to Hodge-theoretic comparison isomorphisms, and so these matrix entries are analogous to periods.
\begin{remark}
Suppose that $F$ acts semisimply on $V\otimes \overline{\mathbb{Q}_\ell}$, so that the Zariski-closure of $\langle F^n\rangle$ in $GL_n(V\otimes \overline{\mathbb{Q}_\ell})$ is a torus $T$.  Then as one runs over all $F$-stable submodules of $V$, the associated $\ell$-adic periods control the reduction type of $T$.  For example, in Example \ref{non-split-sequence}, we see that for $k>0$, the torus in question has additive reduction.  If $(V, F)$ arises from geometry (e.g.~$F$ encodes the action of Frobenius on some $\ell$-adic cohomology group or fundamental group), then the torus $T$ is referred to as a Frobenius torus.  

We plan to elaborate on the integral structure of Frobenius tori and Mumford-Tate groups in future work; we will not require these observations in this paper, but we feel they strengthen the analogy between integral $\ell$-adic periods and periods arising from other notions of Hodge theory.  In particular, one expects integral $\ell$-adic periods to be controlled by an integral structure on the Mumford-Tate group, at least in cases when the integral Tate conjecture is satisfied.
\end{remark}

Suppose that $V$ is a mixed representation of $G_k$ over $\mathbb{Q}_\ell$, and that $T\subset V$ is a $G_k$-stable sublattice.  Then there is a weight filtration $W^\bullet V$, and we denote its intersection with $T$ by $W^\bullet T$.  

The main purpose of this section is to fix an element $\sigma\in G_k$ and study the integral $\ell$-adic periods associated to the short exact sequences (for $s>t>u$)
\begin{equation}\label{weight-sequence}
0\to W^tT/W^{u}T\to W^sT/W^{u}T\to W^sT/W^tT\to 0
 \end{equation}
  with respect to the action of $\sigma$.  In our case, we will take $$T=\mathbb{Z}_\ell[[\pi_1^{\text{\'et}, \ell}(X_{\bar k}; x_i, x_j)]]/W^{-N}\mathbb{Z}_\ell[[\pi_1^{\text{\'et}, \ell}(X_{\bar k}; \bar x_i, \bar x_j)]]$$ to be a quotient of the free pro-finitely generated module on the path torsor of a variety, and $\sigma\in G_k$ to be an element which acts on $T$ via a quasi-scalar relative to the weight filtration.  See the next subsection (\ref{weight-2-subsection}) for a definition of the weight filtration on $\mathbb{Z}_\ell[[\pi_1^{\text{\'et}, \ell}(X_{\bar k}; \bar x_i, \bar x_j)]]$.
  \begin{remark}
  The integral $\ell$-adic periods associated to sequence \ref{weight-sequence} in the case $s=0, t=-1$ are related to Wojtkowiak's $\ell$-adic iterated integrals \cite{wojtkowiakI, wojtkowiakII, wojtkowiakIII, wojtkowiakIV, wojtkowiakV}.  These periods measure the extent to which $\ell$-power denominators are required to construct the canonical paths of Proposition-Construction \ref{canonical-paths-prop-constr}.  As in the proof of Theorem \ref{semisimple-pi1}, these periods control the weight $-2$ part of $\overline{\mathbb{Z}_\ell}[[\Pi_1(X_{\bar k})]]/\mathscr{I}^n$. 
  \end{remark}
  \begin{remark}
  The main results of \cite[Section 19]{deligne} may be viewed as a computation of the periods associated to sequence \ref{weight-sequence} with $X=\mathbb{P}^1\setminus\{0,1,\infty\}$ and $\Pi_1^{\text{\'et}, \ell}(X_{\bar k})$ replaced with a certain metabelian quotient thereof.
  \end{remark}
  We discuss some generalities about iterated $\ell$-adic periods before proceeding to bound the periods associated to sequences \ref{weight-sequence}. First, we characterize integral $\ell$-adic periods in terms of $\on{Ext}^1$ groups.
\begin{lemma}\label{ext-lemma}
Let $R$ be a ring and let $M, N$ be right $R$-modules.  Suppose $\alpha\in \text{Ext}^1_R(M, N)$ is a class representing a short exact sequence $$0\to N\to E\to M\to 0.$$   Then for $r\in  \on{End}_R(M)$, the short exact sequence in the top row of  
$$\xymatrix{
0 \ar[r] & N\ar[r] \ar[d]& M\times_M E\ar[r] \ar[d]& M\ar[r]\ar[d]^{r}& 0\\
0 \ar[r] & N\ar[r] & E\ar[r] & M\ar[r] &0
}$$
splits if and only if $r\cdot \alpha=0$.  Here the right square in the diagram above is Cartesian.
\end{lemma}
\begin{proof}
This is immediate from the Yoneda description of the $\on{End}_R(M)$-module structure on $\on{Ext}^1_R(M,N)$. 
\end{proof}
  \begin{corollary}\label{ext-interpretation}
  Let $V$ be a finitely generated, free $\overline{\mathbb{Z}_\ell}$-module, $W\subset V$ a free submodule, and $F: V\to V$ an endomorphism preserving $W$ (so that $V, W$ are $\overline{\mathbb{Z}_\ell}[F]$-modules).  Then the short exact sequence $$0\to W\to V\to V/W\to 0$$ corresponds to a class $$\alpha\in \on{Ext}^1_{\overline{\mathbb{Z}_\ell}[F]}(V/W, W).$$  If $\alpha$ is not a $\overline{\mathbb{Z}_\ell}$-torsion class, the integral $\ell$-adic period of the triple $(V, W, F)$ is infinite; otherwise, it is the infimum of all $v_\ell(b)$ such that $b\cdot\alpha=0$, with $b\in \overline{\mathbb{Z}_\ell}$.
  \end{corollary}
  \begin{proof}
  This is follows from Lemma \ref{ext-lemma}, setting $r=b$.
  \end{proof}
  Our main technique to bound integral $\ell$-adic periods associated to $(V, W, F)$ will be to bound the annihilator of $$\on{Ext}^1_{\overline{\mathbb{Z}_\ell}[F]}(V/W, W)$$ in $\overline{\mathbb{Z}_\ell}$.  We briefly record two lemmas which allow us to compute these annihilators.
  \begin{lemma}\label{ext-computation}
  Let $W$ be a $\overline{\mathbb{Z}_\ell}[F]$-module.  Then there is an exact sequence $$W\overset{1-F}{\longrightarrow} W\to \on{Ext}^1_{\overline{\mathbb{Z}_\ell}[F]}(\overline{\mathbb{Z}_\ell}, W)\to 0$$ where $F$ acts on $\overline{\mathbb{Z}_\ell}$ by the identity. 
  \end{lemma}
  \begin{proof} 
  The sequence $$0\to \overline{\mathbb{Z}_\ell}[F]\overset{1-F}{\longrightarrow} \overline{\mathbb{Z}_\ell}[F]\to \overline{\mathbb{Z}_\ell}\to 0$$ is a free resolution of $\overline{\mathbb{Z}_\ell}$ as a $\overline{\mathbb{Z}_\ell}[F]$-module; using it to compute the $\on{Ext}$ group in question gives the result.
  \end{proof}
  \begin{lemma}
  Let $W'$ be an $\ell^k$-torsion $\overline{\mathbb{Z}_\ell}[F]$-module.  Then $\on{Ext}^1_{\overline{\mathbb{Z}_\ell}[F]}(\overline{\mathbb{Z}_\ell}, W')$ is $\ell^k$-torsion as well.
  \end{lemma}
  \begin{proof}
  This is immediate from the functoriality of $\on{Ext}^1$.
  \end{proof}
  \subsection{Periods of the fundamental group}\label{weight-2-subsection}  We assume for this section that $k$ is a field, and $X$ is a smooth, geometrically connected $k$-variety with simple normal crossings compactification $\overline{X}$, and $\sigma \in G_k$ is a quasi-scalar of degree $q$ with respect to the weight filtration on $\mathbb{Q}_\ell[[\pi_1^{\text{\'et}, \ell}(X_{\bar k}; x_i, x_j)]]$.  Our results require that:
  \begin{enumerate}[label=(A\arabic*)]
  \item \label{torsion-free-graded} $\mathscr{I}^{n-1}(x_i, x_j)/\mathscr{I}^{n}(x_i, x_j)$ is a torsion-free $\mathbb{Z}_\ell$-module for all $n$.
  \end{enumerate}
   Assumption \ref{torsion-free-graded} holds, for example, if $X$ is a curve and $\ell$ is different from the characteristic. It implies that the map $$\mathbb{Z}_\ell[[\pi_1^{\text{\'et}, \ell}(X_{\bar k}; x_i, x_j)]]\to \mathbb{Q}_\ell[[\pi_1^{\text{\'et}, \ell}(X_{\bar k}; x_i, x_j)]]$$ is injective.  If assumption \ref{torsion-free-graded} is satisfied, we define $$W^{-n}\mathbb{Z}_\ell[[\pi_1^{\text{\'et}, \ell}(X_{\bar k}; x_i, x_j)]]=\mathbb{Z}_\ell[[\pi_1^{\text{\'et}, \ell}(X_{\bar k}; x_i, x_j)]]\cap W^{-n}\mathbb{Q}_\ell[[\pi_1^{\text{\'et}, \ell}(X_{\bar k}; x_i, x_j)]].$$  By definition, then, $$W^{-n}\mathbb{Z}_\ell[[\pi_1^{\text{\'et}, \ell}(X_{\bar k}; x_i, x_j)]]/W^{-n-1}\mathbb{Z}_\ell[[\pi_1^{\text{\'et}, \ell}(X_{\bar k}; x_i, x_j)]]$$ is torsion-free for all $n$.  Often (for example, in the rest of this paper) questions about fundamental groups can be reduced to the case of curves, where Assumption \ref{torsion-free-graded} is satisfied, so we view it as harmless.

  We now study the integral $\ell$-adic periods associated to sequence \ref{weight-sequence}.  The idea of the argument is as follows; we will give it in detail later in this section. Let $$\overline{\mathbb{Z}_\ell}[[\pi_1]]_n=\overline{\mathbb{Z}_\ell}[[\pi_1^{\text{\'et}, \ell}(X_{\bar k}; \bar x_i, \bar x_j)]]/W^{-n}$$ and assume we have already computed the integral $\ell$-adic period $b_{n-1}^0$ associated to the sequence $$0\to W^{-1}/W^{-n}\to \overline{\mathbb{Z}_\ell}[[\pi_1]]_{n-1}\to \overline{\mathbb{Z}_\ell}\to 0.$$  That is, we have found a $\sigma$-equivariant map $$p_{i, j}^{n-1}: \overline{\mathbb{Z}_\ell}\to\overline{\mathbb{Z}_\ell}[[\pi_1]]_{n-1}$$ so that $$\epsilon_{n-1}\circ p_{i,j}^{n-1}=\ell^{\lceil b^0_{n-1} \rceil}\cdot \on{id},$$ where $\epsilon_{n-1}: \overline{\mathbb{Z}_\ell}[[\pi_1]]_{n-1}\to \overline{\mathbb{Z}_\ell}$ is the augmentation map.  Then we consider the pullback diagram of short exact sequences
  \begin{equation}\label{iterative-path-period} \xymatrix{
  0\ar[r] & W^{-n+1}/W^{-n}\ar[r] \ar@{=}[d] & V	\ar[r] \ar[d] & \overline{\mathbb{Z}_\ell}\ar[r] \ar[d]^{p_{i,j}^{n-1}} & 0\\
  0 \ar[r] & W^{-n+1}/W^{-n}\ar[r] 		& \overline{\mathbb{Z}_\ell}[[\pi_1]]_n \ar[r] &  \overline{\mathbb{Z}_\ell}[[\pi_1]]_{n-1}\ar[r] & 0.
  }\end{equation} 
  where    $$V=\overline{\mathbb{Z}_\ell}[[\pi_1]]_n\times_{\overline{\mathbb{Z}_\ell}[[\pi_1]]_{n-1}} \overline{\mathbb{Z}_\ell}.$$  
  We then bound the integral $\ell$-adic period $b_{n, n-1}$ associated to the top row of diagram \ref{iterative-path-period} by the annihilator of $$\on{Ext}^1_{\overline{\mathbb{Z}_\ell}[F]}(\overline{\mathbb{Z}_\ell}, W^{-n+1}/W^{-n}).$$  The integral $\ell$-adic period $b_n$ associated to sequence \ref{weight-sequence} will satisfy $$b_n\leq b_{n-1}+b_{n, n-1}.$$
  
We will need an auxiliary lemma bounding $v_\ell(q^k-1)$, for $q\in \mathbb{Z}_\ell$.

\begin{lemma}\label{q-powers-bound}
Let $\ell$ be a prime and $q$ an $\ell$-adic integer such that $\ell\mid q-1$ (if $\ell>2$) or $\ell^2\mid q-1$ if $\ell=2$.  Then $$|q^k-1|_\ell=|(q-1)k|_\ell=|q-1|_\ell|k|_\ell.$$  Equivalently,
$$v_\ell(q^k-1)=v_\ell(q-1)+v_\ell(k).$$
\end{lemma}
\begin{proof}
We have $$q^k-1=((q-1)+1)^k-1=\sum_{i=1}^k {k\choose i}(q-1)^i.$$  Now $$\left|{k\choose i}(q-1)^i\right|_\ell<|(q-1)k|_\ell$$ for $i>1$, so the result holds by the ultrametric inequality.
\end{proof}
\begin{corollary}\label{better-q-power-bound}
Let $\ell$ be a prime and $q$ an $\ell$-adic integer prime to $\ell$.  Let $r$ be the order of $q$ in $\mathbb{F}_\ell^*$ if $\ell\not=2$ and in $(\mathbb{Z}/4\mathbb{Z}^*)$ if $\ell=2$.  Then $$v_\ell(q^k-1)=\begin{cases} v_\ell(q^r-1)+v_\ell(k/r) \text{ if $r\mid k$}\\ 0 \text{ if $r\nmid k$ and $\ell\not=2$}\\ 1 \text{ if $r\nmid k$ and $\ell=2$}.\end{cases}$$ 
\end{corollary}
\begin{proof}
As in the statement, we consider the three cases $$r\mid k,$$ $$r\nmid k \text{ and }\ell\not=2,$$ and $$r\nmid k \text{ and }\ell=2.$$  The first case is immediate from Lemma \ref{q-powers-bound}, replacing $q$ with $q^r$ and $k$ with $k/r$.  The other two cases follow from the definition of $r$.
\end{proof}
\begin{corollary}\label{q-power-sum-bound}
Let $q, \ell, r$ be as in Corollary \ref{better-q-power-bound}.  Let $$C_{q, \ell, k}=\frac{k}{r}\left(v_\ell(q^r-1)+\frac{1}{\ell-1}\right)$$ if $\ell\not=2$ and $$C_{q, \ell, k}= \frac{k}{r} \left( v_\ell(q^r-1)+\frac{1}{\ell-1}+1\right)+\frac{1}{r}$$ if $\ell=2$.  Then $$\sum_{i=1}^k v_\ell(q^i-1)\leq C_{q, \ell, k}.$$
\end{corollary}
\begin{proof}
In the case $\ell\not=2$, we have by Lemma \ref{better-q-power-bound} that $v_\ell(q^i-1)=v_\ell(q^r-1)+v_\ell(i/r)$ if $r\mid k$ and $0$ otherwise.  Thus we have $$\sum_{i=0}^k v_\ell(q^i-1)=\sum_{i=1}^{\lfloor k/r \rfloor} v_\ell(q^{ri}-1)=\sum_{i=1}^{\lfloor k/r \rfloor} (v_\ell(q^r-1)+v_\ell(i)).$$  But this last satisfies
\begin{align*}
\sum_{i=1}^{\lfloor k/r \rfloor} (v_\ell(q^r-1)+v_\ell(i)) &=\lfloor k/r\rfloor v_\ell(q^r-1)+\lfloor \frac{k}{r\ell}\rfloor+\lfloor \frac{k}{r\ell^2}\rfloor+\lfloor \frac{k}{r\ell^3}\rfloor+\cdots\\
&\leq \frac{k}{r} v_\ell(q^r-1)+\frac{k}{r\ell(1-\frac{1}{\ell})}\\
&=\frac{k}{r}\left(v_\ell(q^r-1)+\frac{1}{\ell-1}\right)
\end{align*}
In the case $\ell=2$, and $r=1$, an identical argument shows that $$\sum_{i=0}^k v_\ell(q^i-1)\leq \frac{k}{r}\left(v_\ell(q^r-1)+\frac{1}{\ell-1}\right).$$  If $\ell=2$ and $r=2$, we have 
\begin{align*}
\sum_{i=1}^k v_\ell(q^i-1) &= \sum_{i \text{ odd}, 1\leq i\leq k} 1 +\sum_{i=1}^{\lfloor k/2\rfloor} (v_\ell(q^r-1)+v_\ell(i))\\
&\leq \frac{k+1}{r} +\frac{k}{r}v_\ell(q^r-1)+\frac{k}{\ell r}+\frac{k}{\ell^2 r}+\cdots\\
&= \frac{k+1}{r}+\frac{k}{r}v_\ell(q^r-1)+\frac{k}{ r(\ell-1)}\\
&= \frac{k}{r} \left( v_\ell(q^r-1)+\frac{1}{\ell-1}+1\right)+\frac{1}{r},
\end{align*}
which concludes the proof.
\end{proof}
  \begin{theorem}\label{weight-2-basic-lifts}
Let $k$ be a field, and $X/k$ a smooth variety with simple normal crossings compactification.  Let $x_1, x_2$ be $k$-rational points or rational tangential basepoints of $X$. Let $\sigma\in G_k$ be a quasi-scalar of degree $q^{-1}\in \mathbb{Z}_\ell^\times$ with respect to the weight filtration on $\mathbb{Q}_\ell[[\pi_1^{\text{\'et}, \ell}(X_{\bar k}; x_1, x_2)]]$.  Suppose that $X$ satisfies assumption \ref{torsion-free-graded}.     Then the integral $\ell$-adic period $b_n^i$ associated to the sequence of $\sigma$-modules $$0\to W^{-i-1}/W^{-n}\to W^{-i}/W^{-n}\to W^{-i}/W^{-i-1}\to 0$$  satisfies $$b_n^i\leq C_{q, \ell, n-i-1}.$$   Here $C_{q, \ell, n}$ is defined as in Corollary \ref{q-power-sum-bound}, and $W^{-i}$ is the $i$-th piece of the weight filtration on $\overline{\mathbb{Z}_\ell}[[\pi_1^{\text{\'et}, \ell}(X_{\bar k}; x_i, x_j)]]$, inherited from the weight filtration on $\mathbb{Z}_\ell[[\pi_1^{\text{\'et}, \ell}(X_{\bar k}; x_i, x_j)]]$ defined above.
\end{theorem}
\begin{proof}
By Corollary \ref{q-power-sum-bound}, it is enough to show that $$b_n^i\leq \sum_{j=1}^{n-i-1} v_\ell(q^j-1).$$  Let $\overline{\mathbb{Z}_\ell}(q)$ be the $\overline{\mathbb{Z}_\ell}[\sigma]$-module which is free of rank one as a $\overline{\mathbb{Z}_\ell}$-module, and on which $\sigma$ acts via the scalar $q$.  Assume we have already proven this for $b_{n-1}^i$, that is, that we have found a commutative diagram of $\sigma$-modules $$
\xymatrix{ & & &W^{-i}/W^{-i-1} \ar[d]^{\cdot a} \ar[ld]_{p_{n-1}}& \\
 0\ar[r] & W^{-i-1}/W^{-n+1}\ar[r] & W^{-i}/W^{-n+1}\ar[r] & W^{-i}/W^{-i-1}\ar[r] & 0
 }$$
 with $$v_\ell(a)\leq \sum_{j=1}^{n-i-2} v_\ell(q^j-1).$$
 
 Then by induction, it is enough to show the period $b^i_{n, n-1}$ associated to the top sequence in the diagram
 \begin{equation}\label{sequence-period-0}\xymatrix{
  0\ar[r] & W^{-n+1}/W^{-n}\ar[r] \ar@{=}[d] & V	\ar[r] \ar[d] & W^{-i}/W^{-i-1} \ar[r] \ar[d]^{p_{n-1}} & 0\\
  0 \ar[r] & W^{-n+1}/W^{-n} \ar[r] 		& W^{-i}/W^{-n} \ar[r] &  W^{-i}/W^{-n+1} \ar[r] & 0
  }\end{equation}
  satisfies $$b^i_{n, n-1}\leq v_\ell(q^{n-i-1}-1)$$ where $$V=W^{-i}/W^{-i-1}\times_{W^{-i}/W^{-n+1}} W^{-i}/W^{-n},$$ that is, the right square is cartesian.  Indeed, $$b^i_n\leq b^i_{n-1}+b^i_{n, n-1}.$$
  But by the assumption that $\sigma$ is a quasi-scalar of degree $q^{-1}$, $\sigma$ acts on $W^{-n+1}/W^{-n}$ via $q^{n-1}$, and on $W^{-i}/W^{-i-1}$ by $q^i$.
   Hence by Corollary \ref{ext-interpretation}, $b^0_{n, n-1}$ is bounded by the least power of $\ell$ annihilating $$\on{Ext}^1_{\overline{\mathbb{Z}_\ell}[\sigma]}(W^{-i}/W^{-i-1}, W^{-n+1}/W^{-n})=\bigoplus_N \on{Ext}^1_{\overline{\mathbb{Z}_\ell}[\sigma]}(\overline{\mathbb{Z}_\ell}(q^i), \overline{\mathbb{Z}_\ell}(q^{n-1})).$$ By Lemma \ref{ext-computation}, this latter Ext group is the cokernel of the map $(1-q^{n-i-1})$ and hence is annihilated by $\ell^{v_\ell(1-q^{n-i-1})}$ as desired.
 \end{proof}
Thus we have bounded all of the periods of  associated to the action of a quasi-scalar with respect to the weight filtration of $$\mathbb{Z}_\ell[[\Pi_1(X_{\bar k})]].$$ We will  reinterpret this result in Theorem \ref{diagonalizing-convergent-rings-mixed}.
\begin{remark}\label{conjecture-bound}
The estimates above required that $\sigma$ be a quasi-scalar.  For quasi-Frobenii $\sigma$, one may estimate that for almost all $\ell$, the integral $\ell$-adic periods above satisfy $$b^i_n=O(n\log n),$$ using Yu's $p$-adic Baker theorem on linear forms in logarithms \cite{yu-baker}.  We believe it would be very interesting if this could be improved to e.g.~a linear estimate, for example if $\sigma$ was a Frobenius element.  In particular, this would allow one to generalize the results of this paper to positive characteristic.
\end{remark}
\begin{remark}
One may imitate many of the methods here over $p$-adic fields (rather than finitely generated fields), taking $\ell=p$ and substituting some mild integral $p$-adic Hodge theory for the estimates proven in this section.  In particular, one may define $p$-adic analogues of integral $\ell$-adic periods and imitate many of the rest of the arguments in the paper to obtain similar restrictions on geometric monodromy to those obtained in the next sections.  However, we were unable to improve our results using these methods, so we omit them here.
\end{remark}

\section{Convergent group rings and Hopf groupoids}\label{convergent-section}
\subsection{Generalities}
Let $k$ be a field and $X$ a variety over $k$.  We now introduce certain Galois-stable sub-objects of $$\overline{\mathbb{Q}_\ell}[[\Pi_1(X_{\bar k})]]$$ on which $\on{Gal}(\overline k/k)$ acts; we will use the results of the previous sections to show that, in favorable circumstances, quasi-scalars in $\on{Gal}(\overline k/k)$ will act \emph{diagonalizably} on these objects.  

We first define a valuation $v_n$ on the groups $$\overline{\mathbb{Q}_\ell}[[\pi_1^{\text{\'et}, \ell}(X_{\bar k}; \bar x_i, \bar x_j)]]/\mathscr{I}^n(x_i, x_j).$$  Namely, if $g\in \overline{\mathbb{Q}_\ell}[[\pi_1^{\text{\'et}, \ell}(X_{\bar k}; \bar x_i, \bar x_j)]]/\mathscr{I}^n(x_i, x_j)$, we let $$v_n(g)=-\inf\{v(s)\mid s\in \overline{\mathbb{Q}_\ell} \text{ such that } s\cdot g \in \overline{\mathbb{Z}_\ell}[[\pi_1^{\text{\'et}, \ell}(X_{\bar k}; \bar x_i, \bar x_j)]]/\mathscr{I}^n(x_i, x_j)\}.$$  Here we write $$s\cdot g \in \overline{\mathbb{Z}_\ell}[[\pi_1^{\text{\'et}, \ell}(X_{\bar k}; \bar x_i, \bar x_j)]]/\mathscr{I}^n(x_i, x_j)$$ to indicate that $s\cdot g$ is in the image of the natural map $$  \overline{\mathbb{Z}_\ell}[[\pi_1^{\text{\'et}, \ell}(X_{\bar k}; \bar x_i, \bar x_j)]]/\mathscr{I}^n(x_i, x_j)\to  \overline{\mathbb{Q}_\ell}[[\pi_1^{\text{\'et}, \ell}(X_{\bar k}; \bar x_i, \bar x_j)]]/\mathscr{I}^n(x_i, x_j).$$  One may check that $v_n$ satisfies the usual properties of a non-archimedean valuation.  

We also let $$\pi_n: \overline{\mathbb{Q}_\ell}[[\pi_1^{\text{\'et}, \ell}(X_{\bar k}; \bar x_i, \bar x_j)]]\to \overline{\mathbb{Q}_\ell}[[\pi_1^{\text{\'et}, \ell}(X_{\bar k}; \bar x_i, \bar x_j)]]/\mathscr{I}^n(x_i, x_j)$$ be the natural quotient map.

Note  that for any $x$,
\begin{equation}\label{valuation-inequality}
v_n(\pi_n(x))\geq v_{n+1}(\pi_{n+1}(x)).
\end{equation}
\begin{defn}\label{convergent-defn}
Fix $r\in \mathbb{R}_{>0}$.  The \emph{convergent Hopf groupoid of radius $\ell^{-r}$} $$\overline{\mathbb{Q}_\ell}[[\Pi_1(X_{\bar k})]]^{\leq \ell^{-r}}\subset \overline{\mathbb{Q}_\ell}[[\Pi_1(X_{\bar k})]]$$ has the same objects $\{x_1, \cdots, x_n\}$ as $\mathbb{Q}_\ell[[\Pi_1(X_{\bar k})]]$, with morphisms $x_i$ to $x_j$, denoted  $$\overline{\mathbb{Q}_\ell}[[\pi_1^{\text{\'et}, \ell}(X_{\bar k}; \bar x_i, \bar x_j)]]^{\leq \ell^{-r}}$$ given by the collection of $g\in \overline{\mathbb{Q}_\ell}[[\pi_1^{\text{\'et}, \ell}(X_{\bar k}; \bar x_i, \bar x_j)]]$ such that $$v_n(\pi_n(g))+rn\to \infty$$ as $n$ tends to infinity.  If $x_i=x_j$, we call the Hopf algebra $$\overline{\mathbb{Q}_\ell}[[\pi_1^{\text{\'et}, \ell}(X_{\bar k}; \bar x_i)]]^{\leq \ell^{-r}}$$ the \emph{convergent group ring of radius $\ell^{-r}$}.  If $R\subset \overline{\mathbb{Q}_\ell}$ is a subring, we define $$R[[\Pi_1(X_{\bar k})]]^{\leq \ell^{-r}}=\overline{\mathbb{Q}_\ell}[[\Pi_1(X_{\bar k})]]^{\leq \ell^{-r}}\cap R[[\Pi_1(X_{\bar k})]].$$
\end{defn}
In other words, the convergent Hopf groupoid consists of those elements whose denominators ``do not grow too quickly" with respect to the $\mathscr{I}$-adic filtration.  One may check that $\overline{\mathbb{Q}_\ell}[[\Pi_1(X_{\bar k})]]^{\leq \ell^{-r}}$ does in fact satisfy the axioms of a Hopf groupoid, justifying its name.  We explicitly check that the composition law inherited from $\overline{\mathbb{Q}_\ell}[[\Pi_1(X_{\bar k})]]$ preserves $\overline{\mathbb{Q}_\ell}[[\Pi_1(X_{\bar k})]]^{\leq \ell^{-r}}$, and leave checking the rest of the axioms (Definition \ref{hopf-groupoid-definition}) to the interested reader.
\begin{prop}\label{multiplication-convergent}
Suppose $R\subset \overline{\mathbb{Q}_\ell}$. Let $$x\in R[[\pi_1^{\text{\'et}, \ell}(X_{\bar k}; \bar x_i, \bar x_j)]]^{\leq \ell^{-r}}, y\in R[[\pi_1^{\text{\'et}, \ell}(X_{\bar k}; \bar x_j, \bar x_k)]]^{\leq \ell^{-r}}$$ be two elements.  Then $$x\circ y\in R[[\pi_1^{\text{\'et}, \ell}(X_{\bar k}; \bar x_i, \bar x_k)]]^{\leq \ell^{-r}}.$$
\end{prop}
\begin{proof}
Clearly it suffices to check this for $R=\overline{\mathbb{Q}_\ell}$.

There exists $N$ such that $v_n(\pi_n(\ell^Nx)) +rn, v_n(\pi_n(\ell^Ny))+rn>0$ for all $n$.  Without loss of generality, we may replace $x, y$ with $\ell^Nx, \ell^Ny$, so we assume that $v_n(\pi_n(x))+rn, v_n(\pi_n(y))+rn>0$ for all $n$.  Let $f: \mathbb{N}\to \mathbb{R}$ be a positive, increasing function tending to infinity such that $$f(n)\leq v_n(\pi_n(x))+rn, v_n(\pi_n(y))+rn$$ for all $n$; such a function exists by hypothesis.

We have that $$v_n(\pi_n(x\circ y))\geq \min_{a+b=n}\left(v_a(\pi_a(x))+v_b(\pi_b(y))\right),$$ using that $\mathscr{I}^a\cdot\mathscr{I}^b\subset \mathscr{I}^{a+b}$, by definition.  Thus 
\begin{align*}
v_n(\pi_n(x\circ y))+rn &\geq \min_{a+b=n}\left(v_a(\pi_a(x))+v_b(\pi_b(y))+r(a+b))\right)\\
&=\min_{a+b=n}\left(v_a(\pi_a(x))+ra+ v_b(\pi_b(y))+rb))\right)\\
&\geq \min_{a+b=n} (f(a)+f(b))\\
&\geq \min_{a+b=n}(\max(f(a), f(b)))\\
&\geq f(\lfloor n/2\rfloor).
\end{align*}
This last clearly tends to infinity with $n$.
\end{proof}
In particular, we have that the sets $$\mathbb{Q}_\ell[[\pi_1^{\text{\'et}, \ell}(X_{\bar k}; \bar x)]]^{\leq \ell^{-r}}$$ are rings, with multiplication inherited from the usual group ring.
\begin{remark}
As the Galois action on $\overline{\mathbb{Q}_\ell}[[\Pi_1(X_{\bar k})]]$ preserves $\overline{\mathbb{Z}_\ell}[[\Pi_1(X_{\bar k})]]$, it also preserves $\overline{\mathbb{Q}_\ell}[[\Pi_1(X_{\bar k})]]^{\leq \ell^{-r}}$.  Thus the convergent Hopf groupoid/convergent group ring inherits a Galois action.
\end{remark}

Our main reason for defining convergent group rings and Hopf groupoids is that for $X$ satisfying Assumption \ref{torsion-free-graded}, continuous representations $$\rho: \pi_1^{\text{\'et}, \ell}(X_{\bar k}, \bar x)\to GL_n(\mathbb{Z}_\ell)$$ will induce natural ring homomorphisms  $${\mathbb{Q}_\ell}[[\pi_1^{\text{\'et}, \ell}(X_{\bar k}, \bar x)]]^{\leq \ell^{-r}}\to \mathfrak{gl}_n({\mathbb{Q}_\ell})$$ for for $r\gg 0$; $r$ depends on the image of $\rho$.  More generally, if $\mathscr{F}$ is a lisse $\ell$-adic sheaf on $X_{\bar k}$, we obtain a representation of $$\mathbb{Q}_\ell[[\Pi_1(X_{\bar k})]]^{\leq \ell^{-r}}$$ for $r\gg0$, in the sense of Definition \ref{hopf-groupoid-reps}.
\begin{defn}\label{r-0-defn}
Let $$\rho: \pi_1^{\text{\'et}, \ell}(X_{\bar k}, \bar x)\to GL_n(\mathbb{Z}_\ell)$$ be a continuous representation.  Let $${\mathbb{Z}_\ell}[[\rho]]: {\mathbb{Z}_\ell}[[\pi_1^{\text{\'et}, \ell}(X_{\bar k}, \bar x)]] \to \mathfrak{gl}_n({\mathbb{Z}_\ell})$$  be the map induced by $\rho$.   Then we define $r_0(\rho)$ be the supremum over all $r\in \mathbb{R}_{>0}$ such that $$({\mathbb{Z}_\ell}[[\rho]])(\mathscr{I}^s)\equiv 0 \bmod \ell^{\lfloor sr\rfloor}\text{ for $s\gg0$}.$$
Here $\mathscr{I}$ is the augmentation ideal of $\mathbb{Z}_\ell[[\pi_1^{\text{\'et}, \ell}(X_{\bar k}, \bar x)]].$
\end{defn}
\begin{prop}\label{convergence-prop}
Let $$\rho: \pi_1^{\text{\'et}, \ell}(X_{\bar k}, \bar x)\to GL_n(\mathbb{Z}_\ell)$$ be a continuous representation. Suppose that $X$ satisfies Assumption \ref{torsion-free-graded}.  Then for any $0< r<r_0(\rho)$,   ${\mathbb{Z}_\ell}[[\rho]]$ extends to a unique commutative diagram
$$\xymatrix{
{\mathbb{Z}_\ell}[[\pi_1^{\text{\'et}, \ell}(X_{\bar k}, \bar x)]] \ar[r]^-{{\mathbb{Z}_\ell}[[\rho]]} \ar[d] & \mathfrak{gl}_n({\mathbb{Z}_\ell})\ar[d]\\
{\mathbb{Q}_\ell}[[\pi_1^{\text{\'et}, \ell}(X_{\bar k}, \bar x)]]^{\leq \ell^{-r}} \ar@{.>}[r]^-{\rho_r} & \mathfrak{gl}_n({\mathbb{Q}_\ell})
}$$
where $\rho_r$ is a continuous map of ${\mathbb{Q}_\ell}$-algebras, and the vertical maps are the obvious inclusions.
\end{prop}
\begin{proof}
Let $x\in {\mathbb{Q}_\ell}[[\pi_1^{\text{\'et}, \ell}(X_{\bar k}, \bar x)]]^{\leq \ell^{-r}}$ be any element.  For each $n$, let $x_n$ be any lift of $\ell^{\lceil v_n(\pi_n(x))\rceil}\pi_n(x)$ to ${\mathbb{Z}_\ell}[[\pi_1^{\text{\'et}, \ell}(X_{\bar k}, \bar x)]]\subset {\mathbb{Q}_\ell}[[\pi_1^{\text{\'et}, \ell}(X_{\bar k}, \bar x)]]$.  By continuity of $\rho_r$, we must have that $$\rho_r(x)=\lim_{n\to \infty} \ell^{-\lceil v_n(\pi_n(x))\rceil} \cdot({\mathbb{Z}_\ell}[[\rho]])(x_n)$$ if the limit exists; this proves uniqueness of $\rho_r$.  (Observe that for this limit formula to make sense, we need Assumption \ref{torsion-free-graded}.)

We now show that the limit does in fact exist.   Let $m$ be a positive integer.  Then $$y_{m+1}:=x_{m+1}-\ell^{\lceil v_{m+1}(\pi_{m+1}(x)))\rceil-  \lceil v_m(\pi_m(x))\rceil}x_m\in \mathscr{I}^m,$$ because  $x_{m+1}$ and $\ell^{\lceil v_{m+1}(\pi_{m+1}(x)))\rceil-  \lceil v_m(\pi_m(x))\rceil}x_m$ are both lifts of $$\ell^{\lceil v_{m+1}(\pi_{m+1}(x)))\rceil}\pi_m(x)$$ to $${\mathbb{Z}_\ell}[[\pi_1^{\text{\'et}, \ell}(X_{\bar k}, \bar x)]].$$  (Here $\mathscr{I}^m$ refers to the $m$-th power of the augmentation ideal in ${\mathbb{Z}_\ell}[[\pi_1^{\text{\'et}, \ell}(X_{\bar k}, \bar x)]]$).  Hence  we  have $$x_{m+1}=\ell^{\lceil v_{m+1}(\pi_{m+1}(x)))\rceil-  \lceil v_m(\pi_m(x))\rceil}x_m+y_{m+1}$$ and thus $$x_n=y_n+\ell^{\lceil v_n(\pi_n(x))\rceil}(\sum_{j=1}^{n-1} \ell^{-\lceil v_j(\pi_j(x))\rceil}y_j)$$ where $y_j\in \mathscr{I}^j$.  Thus $$\ell^{-\lceil v_n(\pi_n(x))\rceil} \cdot({\mathbb{Z}_\ell}[[\rho]])(x_n)=\sum_{j=1}^{n} \ell^{-\lceil v_j(\pi_j(x))\rceil}(\mathbb{Z}_\ell[[\rho]])(y_j).$$

By definition of ${\mathbb{Q}_\ell}[[\pi_1^{\text{\'et}, \ell}(X_{\bar k}, \bar x)]]^{\leq \ell^{-r}}$, $v_j(\pi_j(x))+jr\to \infty$; moreover, by our choice of $r_0$ and the fact that $y_j \in \mathscr{I}^j$, we have $(\mathbb{Z}_\ell[[\rho]])(y_j)\equiv 0\bmod \ell^{\lfloor jr \rfloor}.$  Thus  $$\ell^{-\lceil v_j(\pi_j(x))\rceil}(\mathbb{Z}_\ell[[\rho]])(y_j)\to 0$$ and the sequence converges as desired.
\end{proof}
If $\rho$ is a representation of $\mathbb{Z}_\ell[[\Pi_1(X_{\bar k})]]$, an identical theorem with identical proofs holds for extending representations to $\mathbb{Q}_\ell[[\Pi_1(X_{\bar k})]]^{\leq p^{-r}}$; we omit the precise statement and proof as we will not require them.

Observe that for any continuous representation $$\rho:\pi_1^{\text{\'et}, \ell}(X_{\bar k}, \bar x)\to GL_n(\mathbb{Z}_\ell),$$ the real number $r_0$ defined in Definition \ref{r-0-defn} is always non-zero.  Indeed, the residual representation $$\bar \rho: \pi_1^{\text{\'et}, \ell}(X_{\bar k}, \bar x)\to GL_n(\mathbb{F}_\ell)$$ is necessarily unipotent, so the augmentation ideal of $\mathbb{Z}_\ell[[\pi_1^{\text{\'et}, \ell}(X_{\bar k}, \bar x)]]$ acts topologically nilpotently on $\mathbb{Z}_\ell^n$.
\begin{prop}\label{r0-bound}
Let $$\rho: \pi_1^{\text{\'et}, \ell}(X_{\bar k}, \bar x)\to GL_n(\mathbb{Z}_\ell)$$ be a continuous representation.  Then
\begin{enumerate}
\item  if $$\rho\equiv \on{triv}\bmod ~\ell^k,$$ we have $r_0(\rho)\geq k$.
\item Let $$G= \on{im}(\bar \rho: \pi_1^{\text{\'et}, \ell}(X_{\bar k}, \bar x)\to GL_n(\mathbb{F}_\ell)).$$ If $\mathscr{I}_G$ is the augmentation ideal of $\mathbb{F}_\ell[G]$, let $c$ be the least integer such that $\mathscr{I}_G^{c}=0$.  Then $r_0(\rho)\geq 1/c$.
\end{enumerate}
In particular, for any $\rho$, $r_0(\rho)>0$.
\end{prop}
\begin{proof}
We begin with a proof of (1).  We have $\mathbb{Z}_\ell[[\rho]](\mathscr{I})\equiv 0\bmod \ell^k$, hence $\mathbb{Z}_\ell[[\rho]](\mathscr{I}^n)\equiv 0\bmod \ell^{nk}$, giving the result immediately.

To prove (2), note that $\mathscr{I}_G^{c}=0$ implies that $\mathbb{Z}_\ell[[\rho]](\mathscr{I}^{c})\equiv 0\bmod \ell$.  Hence $$\mathbb{Z}_\ell[[\rho]](\mathscr{I}^{mc})\equiv 0\bmod \ell^m$$ giving the desired inequality.
\end{proof}
\subsection{Examples and relation to integral $\ell$-adic periods}
We now move on to studying specific examples of $\mathbb{Q}_\ell[[\Pi_1(X_{\bar k})]]^{\leq p^{-r}}$, and in particular to analyzing the Galois action on these convergent Hopf groupoids/group rings.

The following example justifies the word ``convergent" in the name ``convergent Hopf groupoid."
\begin{example}
Let $k$ be a field of characteristic different from $\ell$, and let $X=\mathbb{G}_{m, k}$.  Let $x$ be the rational basepoint $1\in \mathbb{G}_m(k)$.  Fixing an embedding $k\hookrightarrow \bar k$, we obtain a geometric point $\bar x$ of $X$.  Then $$\pi_1^{\text{\'et}, \ell}(X_{\bar k}, \bar x)=\mathbb{Z}_\ell(1).$$  Let $\Gamma=\mathbb{Z}_\ell(1)$ .  It is well-known that $$\mathbb{Z}_\ell[[\Gamma]]=\mathbb{Z}_\ell[[T]]$$ and that the $\mathscr{I}$-adically completed group ring satisfies $$\mathbb{Q}_\ell[[\Gamma]]\simeq \mathbb{Q}_\ell[[T]]$$ under the isomorphism $T\mapsto \gamma-1$, where $\gamma$ is a topological generator of $\Gamma$.  We have that $$W^{-n}=\mathscr{I}^n=(T^n).$$ Under this isomorphism, the subring of $\mathbb{Q}_\ell[[\Gamma]]^{\leq \ell^{-r}}\subset \mathbb{Q}_\ell[[\Gamma]]$ is identified precisely with the ring of convergent power series on the closed ball of radius $\ell^{-r}$ in $\mathbb{Q}_\ell[[T]]$, for $r>0$.  If $k$ is a field and $\sigma\in G_k$, we have that for $\gamma\in \Gamma$, $$\sigma(\gamma)=\gamma^{\chi(\sigma)},$$ where $\chi: G_k\to \mathbb{Z}_\ell^\times$ is the cyclotomic character.  (In particular, all $\sigma$ act via quasi-scalars.)

 Equivalently, $$\gamma(1+T)=(1+T)^{\chi(\sigma)}.$$ Hence  $\log(1+T)$ is sent to $\chi(\sigma)\log(1+T)$ by $\sigma$.  For any $r > 0$, $\log(1+T)\in \mathbb{Q}_\ell[[\Gamma]]^{\leq \ell^{-r}}$.  Indeed, the elements $$\log(1+T)^n, n\in \mathbb{Z}_{\geq 0}$$ are a basis of $\sigma$-eigenvectors in $\mathbb{Q}_\ell[[\Gamma]]^{\leq \ell^{-r}}$ with eigenvalues $\chi(\sigma)^n$, i.e.~they are linearly independent over $\mathbb{Q}_\ell$ and their span is dense.
\end{example}
In fact, we have a similar statement for any variety with abelian or nilpotent fundamental group (there exist varieties with non-abelian torsion-free nilpotent fundamental group---see e.g.~\cite{sommese-vandeven} or \cite{campana} for examples).
\begin{theorem}\label{diagonalizing-nilpotent-pi1}
Let $k$ be a field of characteristic different from $\ell$, and let $X/k$ be a smooth variety admitting a simple normal crossings compactification.  Let $x$ be a rational point or rational tangential basepoint of $X$, and $k\hookrightarrow \bar k$ an algebraic closure of $X$.  Suppose that $\sigma\in G_k$ acts on $\mathbb{Q}_\ell[[\pi_1^{\text{\'et}, \ell}(X_{\bar k}, \bar x)]]$ via a quasi-Frobenius with respect to the weight filtration.  Suppose further that $$\pi_1^{\text{\'et}, \ell}(X_{\bar k}, \bar x)$$ is finitely generated and nilpotent.  Then for any $r>0$, $$\overline{\mathbb{Q}_\ell}[[\pi_1^{\text{\'et}, \ell}(X_{\bar k}, \bar x)]]^{\leq \ell^{-r}}$$ admits a  set of  $\sigma$-eigenvectors whose span is dense.  Moreover, $$\mathscr{I}^n(x)\cap \overline{\mathbb{Q}_\ell}[[\pi_1^{\text{\'et}, \ell}(X_{\bar k}, \bar x)]]^{\leq \ell^{-r}}$$ admits a set of $\sigma$-eigenvectors of weight $\leq -n$ whose span is dense.
\end{theorem}
Recall that we defined the weight of a $\sigma$-eigenvector in Definition \ref{quasi-frobenii-defn}.

Before proving the theorem, we need a brief lemma.
\begin{lemma}\label{i-adic-weights}
Let $k$ be a field and $X$ a $k$-variety; let $x_i, x_j$ be rational points or rational tangential basepoints of $X$ such that for some $r>0$, $\overline{\mathbb{Q}_\ell}[[\pi_1^{\text{\'et}, \ell}(X_{\bar k}; \bar x_i, \bar x_j)]]^{\leq \ell^{-r}}$ admits a set of $\sigma$-eigenvectors whose span is dense, for some $\sigma\in G_k$ which is a quasi-Frobenius with respect to the weight filtration.  Let $w$ be the minimum weight appearing in the $G_k$-representation $H^1(X_{\bar k}, \mathbb{Q}_\ell)$.  Then $$\mathscr{I}^n\cap \overline{\mathbb{Q}_\ell}[[\pi_1^{\text{\'et}, \ell}(X_{\bar k}; \bar x_i, \bar x_j)]]^{\leq \ell^{-r}}$$ is (topologically) spanned by $\sigma$-eigenvectors of weight at most $-nw$.
\end{lemma}
\begin{proof}
Let $W_i\subset \overline{\mathbb{Q}_\ell}[[\pi_1^{\text{\'et}, \ell}(X_{\bar k}; \bar x_i, \bar x_j)]]^{\leq \ell^{-r}}$ be the span of the weight $i$ $\sigma$-eigenvectors.  Then if $i\not=0$, the augmentation map $W_i\to \overline{\mathbb{Q}_\ell}$ is zero, so $W_i\subset \mathscr{I}$; moreover, as the $W_i$ span $\overline{\mathbb{Q}_\ell}[[\pi_1^{\text{\'et}, \ell}(X_{\bar k}; \bar x_i, \bar x_j)]]^{\leq \ell^{-r}}$, the map $W_0\to \mathbb{Q}_\ell$ must be an isomorphism.   In general, comparison to the associated graded of the $\mathscr{I}$-adic filtration shows that any $\sigma$-eigenvector of weight $\leq -nw$ must be contained in $\mathscr{I}^n$.

Let us now prove that $\mathscr{I}^n\cap \overline{\mathbb{Q}_\ell}[[\pi_1^{\text{\'et}, \ell}(X_{\bar k}; \bar x_i, \bar x_j)]]^{\leq \ell^{-r}}$ is spanned by $\sigma$-eigenvectors, which by the previous paragraph necessarily have weight $\leq -wn$.  As the $\sigma$-eigenvectors have dense span in $\overline{\mathbb{Q}_\ell}[[\pi_1^{\text{\'et}, \ell}(X_{\bar k}; \bar x_i, \bar x_j)]]^{\leq \ell^{-r}}$, the $\sigma$-eigenvectors of weight at most $-1$ must span $\mathscr{I}$.  Let $v_1, \cdots, v_n$ be a set of lifts of the $\sigma$-eigenvectors of $\mathscr{I}/\mathscr{I}^2$ to $\mathscr{I}$.  Then the monomials in the $v_i$ of degree $\geq n$ have span which is dense in $\mathscr{I}^n$, e.g.~by considering the associated graded ring of the $\mathscr{I}$-adic filtration.
\end{proof}
\begin{proof}[Proof of Theorem \ref{diagonalizing-nilpotent-pi1}]
We will freely use facts about Mal'cev Lie algebras in this argument---see \cite[Appendix A3]{quillen-rational} for a reference.  We define $$\on{Lie} \pi_1^{\text{\'et}, \ell}(X_{\bar k})\subset \overline{\mathbb{Q}_\ell}[[\pi_1^{\text{\'et}, \ell}(X_{\bar k}, \bar x)]]$$ to be the set of primitive elements, namely those $x$ such that $$\Delta(x)=x\otimes1+1\otimes x.$$  Then the bracket $$[-, -]: x, y\mapsto xy-yx$$ turns $\on{Lie} \pi_1^{\text{\'et}, \ell}(X_{\bar k})$ into a $\mathbb{Q}_\ell$-Lie algebra, as the name would suggest.  As $ \pi_1^{\text{\'et}, \ell}(X_{\bar k})$ is finitely generated and nilpotent, $\on{Lie} \pi_1^{\text{\'et}, \ell}(X_{\bar k})$ is finite-dimensional over $\mathbb{Q}_\ell$ and generated over $\mathbb{Q}_\ell$ by $$\{\log x, x\in \on{im}(\pi_1^{\text{\'et}, \ell}(X_{\bar k})\to \overline{\mathbb{Q}_\ell}[[\pi_1^{\text{\'et}, \ell}(X_{\bar k}, \bar x)]])\}.$$  Hence $$\on{Lie} \pi_1^{\text{\'et}, \ell}(X_{\bar k})\subset \overline{\mathbb{Q}_\ell}[[\pi_1^{\text{\'et}, \ell}(X_{\bar k}, \bar x)]]^{\leq \ell^{-r}}$$ for any $r>0$.  For $N\gg0$, the natural map $$\on{Lie} \pi_1^{\text{\'et}, \ell}(X_{\bar k})\to \overline{\mathbb{Q}_\ell}[[\pi_1^{\text{\'et}, \ell}(X_{\bar k}, \bar x)]]/\mathscr{I}^N(x)$$ is injective, so it follows from Theorem \ref{semisimple-pi1} that $\sigma$ acts semisimply on $\on{Lie} \pi_1^{\text{\'et}, \ell}(X_{\bar k})$.  Thus $\on{Lie} \pi_1^{\text{\'et}, \ell}(X_{\bar k})$ admits a basis of eigenvectors.  

Let $\{v_1, \cdots, v_n\}$ be a set of linearly independent $\sigma$-eigenvectors in $\on{Lie} \pi_1^{\text{\'et}, \ell}(X_{\bar k})$ lifting a basis of eigenvectors of $\on{Lie} \pi_1^{\text{\'et}, \ell}(X_{\bar k})^{\text{ab}}$.  The map $$\on{Lie} \pi_1^{\text{\'et}, \ell}(X_{\bar k})\to \overline{\mathbb{Q}_\ell}[[\pi_1^{\text{\'et}, \ell}(X_{\bar k}, \bar x)]]$$ factors through $\mathscr{I}(x)$, and the induced map $$\on{Lie} \pi_1^{\text{\'et}, \ell}(X_{\bar k})^{\text{ab}} \to \mathscr{I}(x)/\mathscr{I}^2(x)$$ is an isomorphism.  Thus the diagram
$$\xymatrix{
\on{Span}(v_i)\simeq \on{Lie} \pi_1^{\text{\'et}, \ell}(X_{\bar k})^{\text{ab}}  \ar[r] \ar[d] & \mathscr{I}(x)\ar@{=}[d]\\
\on{Lie} \pi_1^{\text{\'et}, \ell}(X_{\bar k})\ar[r] \ar[d] &  \mathscr{I}(x) \ar[d]\\
\on{Lie} \pi_1^{\text{\'et}, \ell}(X_{\bar k})^{\text{ab}} \ar[r]^\sim & \mathscr{I}(x)/\mathscr{I}^2(x)
}$$
gives a $\sigma$-equivariant splitting of the map $$\mathscr{I}(x)\to \mathscr{I}(x)/\mathscr{I}^2(x).$$  Let $\{\bar v_1, \cdots, \bar v_n\}$ be the images of the $v_i$ in $\mathscr{I}(x)$.  Then the monomials in the $\bar v_i$ (including the empty monomial $1$) provide a set of $\sigma$-eigenvectors in $\overline{\mathbb{Q}_\ell}[[\pi_1^{\text{\'et}, \ell}(X_{\bar k}, \bar x)]]^{\leq \ell^{-r}}$ whose span is dense, as desired.  The span of a monomial in the $v_i$ of total degree $n$ has weight $-n$, giving the claim about the weights of $$\mathscr{I}^n(x)\cap \overline{\mathbb{Q}_\ell}[[\pi_1^{\text{\'et}, \ell}(X_{\bar k}, \bar x)]]^{\leq \ell^{-r}}.$$  (Alternately, we may observe that the weights of $H^1(X_{\bar k}, \mathbb{Q}_\ell)$ are in $\{1,2\}$, so the claim about weights follows from Lemma \ref{i-adic-weights}.)
\end{proof}

We now recast the results of Section \ref{integral-l-adic-periods-section} in terms of convergent group rings and convergent Hopf groupoids.  
\begin{theorem}\label{diagonalizing-convergent-rings-pure}
Let $k$ be a field and $X/k$ a smooth variety satisfying Assumption \ref{torsion-free-graded} of Section \ref{integral-l-adic-periods-section} with $H^1(X_{\bar k}, \mathbb{Q}_\ell)$ pure of weight $\iota$.  Let $x_i, x_j$ be rational points or rational tangential basepoints of $X$; choosing an embedding $k\hookrightarrow \bar k$ gives geometric basepoints $\bar x_i, \bar x_j$.  Suppose that $\sigma\in G_k$ acts via a quasi-scalar of degree $q^{-1}$ on $\mathbb{Q}_\ell[[\pi_1^{\text{\'et}, \ell}(X_{\bar k}; \bar x_i, \bar x_j)]]$ with respect to the weight filtration.  Let $s$ be the order of $q$ in $\mathbb{F}_\ell^\times$ if $\ell\neq 2$ and in $(\mathbb{Z}/4\mathbb{Z})^\times$ if $\ell=2$.  Let $\epsilon=1$ if $\ell=2$ and $\epsilon=0$ otherwise.  Then for any $$r> \frac{1}{s}\left(v_\ell(q^s-1)+\frac{1}{\ell-1}+\epsilon\right),$$ the topological vector space $$\overline{\mathbb{Q}_\ell}[[\pi_1^{\text{\'et}, \ell}(X_{\bar k}; \bar x_i, \bar x_j)]]^{\leq \ell^{-r}}$$ is spanned by $\sigma$-eigenvectors.  Moreover, $$\mathscr{I}^n(x_i, x_j)\cap \overline{\mathbb{Q}_\ell}[[\pi_1^{\text{\'et}, \ell}(X_{\bar k}; \bar x_i, \bar x_j)]]^{\leq \ell^{-r}}$$ admits a set of $\sigma$-eigenvectors of weight $\leq -n$ whose span is dense.
\end{theorem}
\begin{proof}
Under the assumption that $H^1(X_{\bar k}, \mathbb{Q}_\ell)$ is pure, the weight filtration on $\mathbb{Q}_\ell[[\pi_1^{\text{\'et}, \ell}(X_{\bar k}; \bar x_i, \bar x_j)]]$ is the same as the $\mathscr{I}$-adic filtration.  We use this fact freely below.

We first consider the case $x_i=x_j$; denote this basepoint or tangential basepoint simply by $x$.  The bounds on the integral $\ell$-adic periods $b^1_n$ in Theorem \ref{weight-2-basic-lifts} precisely show that the natural map $$\mathscr{I}(x)\cap \overline{\mathbb{Q}_\ell}[[\pi_1^{\text{\'et}, \ell}(X_{\bar k}, \bar x)]]^{\leq \ell^{-r}}\to \mathscr{I}(x)/\mathscr{I}^2(x)$$ splits in a canonical, $\sigma$-equivariant manner.  Indeed, the bounds of that theorem show exactly that if $\alpha \in \mathscr{I}(x)/\mathscr{I}^2(x)$ is such that $v_2(\alpha)=0$, the unique $\alpha_n\in \mathscr{I}(x)/\mathscr{I}^n(x)$ such that $\alpha_n$ lifts $\alpha$ as a $\sigma$-eigenvector satisfies: $$v_n(\alpha_n)> - \frac{n}{s}\left(v_\ell(q^s-1)+\frac{1}{\ell-1}+\epsilon\right).$$   
The $(\alpha_n)$ give exactly the desired $\sigma$-equivariant splitting.  Let $\{v_1, \cdots, v_N\}$ be a basis of the image of this splitting.  Then the monomials in the $v_i$ (including the empty monomial $1$) are a spanning set of $\sigma$-eigenvectors (they are all contained in $\overline{\mathbb{Q}_\ell}[[\pi_1^{\text{\'et}, \ell}(X_{\bar k}, \bar x)]]^{\leq \ell^{-r}}$ by Proposition \ref{multiplication-convergent}); the monomials of degree $\geq n$ have weight $\leq -n\iota$ and span $\mathscr{I}^n\cap \overline{\mathbb{Q}_\ell}[[\pi_1^{\text{\'et}, \ell}(X_{\bar k}, \bar x)]]^{\leq \ell^{-r}}$.

Now, we consider the case $x_i\not=x_j$.  The bounds on the integral $\ell$-adic periods $b^0_n$ in Theorem \ref{weight-2-basic-lifts} precisely show that the natural augmentation map $$\overline{\mathbb{Q}_\ell}[[\pi_1^{\text{\'et}, \ell}(X_{\bar k}; \bar x_i, \bar x_j)]]^{\leq \ell^{-r}}\to \overline{\mathbb{Q}_\ell}$$ splits $\sigma$-equivariantly, that is, the canonical path $p_{i,j}$ constructed in Proposition-Construction \ref{canonical-paths-prop-constr} lies in $\overline{\mathbb{Q}_\ell}[[\pi_1^{\text{\'et}, \ell}(X_{\bar k}; \bar x_i, \bar x_j)]]^{\leq \ell^{-r}}$.  Now if $\{w_m\}_{m\in I}$ is a spanning set of $\sigma$-eigenvectors in $\overline{\mathbb{Q}_\ell}[[\pi_1^{\text{\'et}, \ell}(X_{\bar k}, \bar x_i)]]^{\leq \ell^{-r}}$, the elements $\{w_m\circ p_{i,j}\}_{m\in I}$ are a spanning set of $
\sigma$-eigenvectors in $\overline{\mathbb{Q}_\ell}[[\pi_1^{\text{\'et}, \ell}(X_{\bar k}; \bar x_i, \bar x_j)]]^{\leq \ell^{-r}}$.
\end{proof}
We now consider the case where $H^1(X_{\bar k}, \mathbb{Z}_\ell)$ is mixed.
\begin{theorem}\label{diagonalizing-convergent-rings-mixed}
Let $k$ be a field and $X/k$ a smooth variety satisfying Assumption \ref{torsion-free-graded} of Section \ref{integral-l-adic-periods-section}.  Let $x_i, x_j$ be rational points or rational tangential basepoints of $X$; choosing an embedding $k\hookrightarrow \bar k$ gives geometric basepoints $\bar x_i, \bar x_j$.  Suppose that $\sigma\in G_k$ acts via a quasi-scalar of degree $q^{-1}$ on $\mathbb{Q}_\ell[[\pi_1^{\text{\'et}, \ell}(X_{\bar k}; \bar x_i, \bar x_j)]]$ with respect to the weight filtration.  Let $s$ be the order of $q$ in $\mathbb{F}_\ell^\times$ if $\ell\not=2$, and in $(\mathbb{Z}/4\mathbb{Z})^\times$ if $\ell=2$.  Let $\epsilon=1$ if $\ell=2$ and $\epsilon=0$ otherwise.  Then for any $$r>\frac{2}{s}\left(v_\ell(q^{s}-1)+\frac{1}{\ell-1}+\epsilon\right),$$ we have that  $$\overline{\mathbb{Q}_\ell}[[\pi_1^{\text{\'et}, \ell}(X_{\bar k}; \bar x_i, \bar x_j)]]^{\leq \ell^{-r}}$$ is spanned by $\sigma$-eigenvectors.  Moreover, $$\mathscr{I}^n(x_i, x_j)\cap \overline{\mathbb{Q}_\ell}[[\pi_1^{\text{\'et}, \ell}(X_{\bar k}; \bar x_i, \bar x_j)]]^{\leq \ell^{-r}}$$ admits a set of $\sigma$-eigenvectors of weight $\leq -n$ whose span is dense.
\end{theorem}
\begin{proof}
The proof is identical to that of Theorem \ref{diagonalizing-convergent-rings-pure}, replacing the $\mathscr{I}$-adic filtration with the weight filtration.  By Proposition \ref{weight-prop}, we have that $W^{-2n-1}\subset \mathscr{I}^n$; we let $$p_n: \mathbb{Q}_\ell[[\pi_1^{\text{\'et}, \ell}(X_{\bar k}; \bar x_i, \bar x_j)]]/W^{-2n-1}\to \mathbb{Q}_\ell[[\pi_1^{\text{\'et}, \ell}(X_{\bar k}; \bar x_i, \bar x_j)]]/\mathscr{I}^n$$ be the natural quotient map.  Let $\alpha\in W^{-i}/W^{-i-1}$ be a $\sigma$-eigenvector.  By Theorem \ref{weight-2-basic-lifts}, for $n\gg 0$ the unique $\sigma$-eigenvector $\alpha_n$ in $\mathbb{Q}_\ell[[\pi_1^{\text{\'et}, \ell}(X_{\bar k}; \bar x_i, \bar x_j)]]/W^{-2n-1}$ lifting $v$ satisfies $$v_n(p_n(\alpha_n))>-\frac{2n}{s}\left(v_\ell(q^{s}-1)+\frac{1}{\ell-1}+\epsilon\right)+C$$ for some constant $C$.  Thus the sequence $p_n(\alpha_n)$ defines the desired eigenvector in $\overline{\mathbb{Q}_\ell}[[\pi_1^{\text{\'et}, \ell}(X_{\bar k}; \bar x_i, \bar x_j)]]^{\leq \ell^{-r}}$.
\end{proof}
\begin{remark}\label{bound-observation-2}
Observe that the bound on $r$ so that the action of our quasi-scalar $\sigma$ on the convergent group ring of radius $\ell^{-r}$ depends only on the degree of $\sigma$.
\end{remark}
\begin{remark}
The bounds in Theorems \ref{diagonalizing-convergent-rings-pure} and \ref{diagonalizing-convergent-rings-mixed} are clearly not optimal.  For example, let $X, k$ be as in \ref{diagonalizing-convergent-rings-pure}, and let $k'$ be a finite extension of $k$.  Suppose $\sigma\in G_k$ acts on $\overline{\mathbb{Q}_\ell}[[\pi_1^{\text{\'et}, \ell}(X_{\bar k}; \bar x_i, \bar x_j)]]^{\leq \ell^{-r}}$ via a quasi-scalar of degree $q^{-1}$ with respect to the weight filtration.  Then for any $r$ as in the theorem, $$\overline{\mathbb{Q}_\ell}[[\pi_1^{\text{\'et}, \ell}(X_{\bar k}; \bar x_i, \bar x_j)]]^{\leq \ell^{-r}}$$ admits a set of $\sigma$-eigenvectors whose span is dense.  But if we instead applied the theorem to $X_{k'}, k'$ (replacing $\sigma$ with $\sigma^n\in G_{k'}\subset k'$), we would obtain significantly worse bounds for $r$.  We formalize one solution to this problem with the following corollary.
\end{remark}
\begin{corollary}\label{trivial-corollary}
Let $k$ be a finite field and $X/k$ a variety.  Let $x_i, x_j$ be rational points or rational tangential basepoints of $X$.  Let $\sigma\in G_k$ be any element.  Suppose that for some $r>0$, $$\overline{\mathbb{Q}_\ell}[[\pi_1^{\text{\'et}, \ell}(X_{\bar k}; \bar x_i, \bar x_j)]]^{\leq \ell^{-r}}$$ admits a set of $\sigma$-eigenvectors whose span is dense.  Then for any finite extension $k'/k$, with $\sigma^n\in G_{k'}\subset G_k$, the same is true for $X_{k'}$, i.e. $$\overline{\mathbb{Q}_\ell}[[\pi_1^{\text{\'et}, \ell}(X_{\bar k}; \bar x_i, \bar x_j)]]^{\leq \ell^{-r}}$$ admits a set of $\sigma^n$-eigenvectors whose span is dense.
\end{corollary}
\begin{proof}
If $\{v_i\}\subset \overline{\mathbb{Q}_\ell}[[\pi_1^{\text{\'et}, \ell}(X_{\bar k}; \bar x_i, \bar x_j)]]^{\leq \ell^{-r}}$ is a set of eigenvectors for the $\sigma$  whose span is dense, then it is also a set of eigenvectors for $\sigma^n$, concluding the proof.
\end{proof}

\section{Applications to monodromy}\label{applications-section}
We now use the results of the preceding sections to study monodromy representations arising from geometry.  Let $k$ be a  field and $X$ a $k$-variety.  Let $\bar x$ be a geometric point of $X$.  Suppose that $$\rho: \pi_1^{\text{\'et}}(X_{\bar k}, \bar x)\to GL_n(\mathbb{Z}_\ell)$$ is a representation arising from geometry---i.e.~suppose $\rho$ arises as the monodromy representation on $R^i\pi_*\mathbb{Z}_\ell$ for some proper morphism $\pi: Y\to X_{\bar k}$ of varieties over $\on{Spec}(\bar k)$.  Then the first observation of this section is that there exists a finite extension $k\subset k'$, and a commutative diagram
$$\xymatrix{
\pi_1^{\text{\'et}}(X_{\bar k}, \bar x) \ar[r]^\rho \ar[d]& GL_n(\mathbb{Z}_\ell)\\
\pi_1^{\text{\'et}}(X_{k'}, \bar x) \ar[ru]_{\rho'}.
}$$
That is, the representation $\rho$ extends to a representation of $\pi_1^{\text{\'et}}(X_{k'}, \bar x)$.  By the results of the previous sections, this will place serious restrictions on $\rho$.

In fact, many of our results do not require that the representation $\rho$ come from geometry.  We define another notion --- that of arithmeticity --- which we will use in our proofs.

\begin{defn}
Let $k$ be a field and $X$ a $k$-variety with geometric point $\bar x$. Then we say that a continuous representation $$\rho: \pi_1^{\text{\'et}, \ell}(X_{\bar k}, \bar x)\to GL_n(\mathbb{Z}_\ell)$$ is \emph{arithmetic} if it arises as a subquotient of a representation $$\tilde{\rho}: \pi_1^{\text{\'et}, \ell}(X_{\bar k}, \bar x)\to GL_n(\mathbb{Z}_\ell),$$ such that $\tilde{\rho}$ extends to a continuous representation of $\pi_1^{\text{\'et}, \ell}(X_{k'}, \bar x)$ for some finite extension $k'$ of $k$.
\end{defn}
We first observe that for arithmetic $\rho$, we may choose $\tilde{\rho}$ as above having many of the same properties as $\rho$.  For example, if $\rho$ is trivial mod $\ell^r$, we may assume that $\tilde{\rho}|_{\pi_1^{\text{\'et}, \ell}(X_{\bar k}, \bar x)}$ is trivial mod $\ell^r$.
\begin{lemma}\label{subquotient-lemma}
Suppose that $$\rho: \pi_1^{\text{\'et}, \ell}(X_{\bar k}, \bar x)\to GL_n(\mathbb{Z}_\ell)$$ is an arithmetic representation, which is trivial mod $\ell^r$, and that $\rho\otimes \mathbb{Q}_\ell$ is irreducible.  Then a finite extension $k\subset k'$ and a representation $\beta: \pi_1^{\text{\'et}, \ell}(X_{k'}, \bar x)\to GL_n(\mathbb{Z})$ such that 
\begin{itemize}
\item $\rho$ is a subquotient of $\beta|_{\pi_1^{\text{\'et}, \ell}(X_{\bar k}, \bar x)}$, and
\item $\beta|_{\pi_1^{\text{\'et}, \ell}(X_{\bar k}, \bar x)}$ is trivial mod $\ell^r$.
\item $r_0(\beta|_{\pi_1^{\text{\'et}, \ell}(X_{\bar k}, \bar x)})\geq r_0(\rho)$, where $r_0(\rho)$ is defined as in Definition \ref{r-0-defn}.
\end{itemize}
\end{lemma}
\begin{proof}
By assumption, there exists a finite extension $k\subset k'$ and a representation $$\gamma: \pi_1^{\text{\'et}, \ell}(X_{k'}, \bar x)\to GL_n(\mathbb{Z}_\ell)$$ such that $\rho$ arises as a subquotient of $\gamma|_{\pi_1^{\text{\'et}, \ell}(X_{k'}, \bar x)}$.  Let $0=V^0\subset V^1 \subset \cdots V^r=\gamma$ be the socle filtration of $\gamma|_{\pi_1^{\text{\'et}, \ell}(X_{\bar k}, \bar x)}\otimes\mathbb{Q}_\ell$ (viewed as a representation of $\pi_1^{\text{\'et}, \ell}(X_{\bar k}, \bar x)$), i.e. $V^i$ is the largest subrepresentation of $\gamma|_{\pi_1^{\text{\'et},\ell}(X_{\bar k}, \bar x)}$ such that $V^i/V^{i-1}$ is semi-simple as a $\pi_1^{\text{\'et}, \ell}(X_{\bar k}, \bar x))$-representation.  As the socle filtration is canonical, each $V^i/V^{i-1}$ also extends to a representation of $\pi_1^{\text{\'et}, \ell}(X_{k'}, \bar k)$.  As $\rho\otimes \mathbb{Q}_\ell$ is irreducible, it arises as a direct summand of $V^i/V^{i-1}$ for some $i$ (viewed as a $\pi_1^{\text{\'et}, \ell}(X_{\bar k}, \bar x)$-representation).  Thus we may assume $\gamma|_{\pi_1^{\text{\'et}, \ell}(X_{\bar k}, \bar x)}\otimes \mathbb{Q}_\ell$ is semisimple.

Let $W\subset V^i/V^{i-1}$ be the minimal sub-$\pi_1^{\text{\'et}, \ell}(X_{k'}, \bar x)$-representation containing $\rho\otimes \mathbb{Q}_\ell$.  As a $\pi_1^{\text{\'et}, \ell}(X_{\bar k}, \bar x)$-representation, $W$ splits as a finite direct sum of the form $$W=\bigoplus (\rho^{\sigma}\otimes \mathbb{Q}_\ell)^{\alpha_\sigma},$$ where $\rho^{\sigma}$ denotes the representation obtained by pre-composing $\rho$ with an automorphism $\sigma$ of $\pi_1^{\text{\'et}, \ell}(X_{\bar k}, \bar x)$.  Let $\beta$ be the $\mathbb{Z}_\ell$-submodule of $W$ spanned by $\tau(\rho)$, for $\tau\in \pi_1^{\text{\'et}, \ell}(X_{k'}, \bar x).$  $\beta$ is a lattice by the compactness of $G_k$, and $\rho$ is clearly a subquotient of $\beta$.  Moreover $\beta$ is spanned by the modules $\tau(\rho)$, each of which is trivial mod $\ell^r$ (as representations of $\pi_1^{\text{\'et}, \ell}(X_{\bar k}, \bar x)$), and hence is trivial mod $\ell^r$ as desired.  The statement about $r_0$ follows identically (from the fact that $\beta$ is spanned by the $\tau(\rho)$).
\end{proof}
The main arguments of this section will  exploit the fact that in many cases, we have shown that if $k$ is a finitely generated field, there exists an element $\sigma\in G_k$ such that for $r\gg0$, $$\mathscr{I}^n(x_i, x_j)\cap \overline{\mathbb{Q}_\ell}[[\pi_1^{\text{\'et}, \ell}(X_{\bar k}, \bar x)]]^{\leq \ell^{-r}}$$ is spanned by $\sigma$-eigenvectors of weight $\leq -n$, and where $\sigma\in G_k$ is acts on $\mathbb{Q}_\ell[[\pi_1^{\text{\'et}, \ell}(X_{\bar k}, \bar x)]]$ via a quasi-scalar with respect to the weight filtration.  Recall that if $\rho: \pi_1^{\text{\'et}, \ell}(X_{\bar k}, \bar x)\to GL_n(\mathbb{Z}_\ell)$ is a representation, then for $r<r_0(\rho)$ (defined as in Definition \ref{r-0-defn}), $\rho$ induces a map $$\rho_r: \overline{\mathbb{Q}_\ell}[[\pi_1^{\text{\'et}, \ell}(X_{\bar k}, \bar x)]]^{\leq \ell^{-r}}\to \mathfrak{gl}_n(\mathbb{Q}_\ell).$$  But $\rho_r$ must kill any $\sigma$-eigenspace in $\overline{\mathbb{Q}_\ell}[[\pi_1^{\text{\'et}, \ell}(X_{\bar k}, \bar x)]]^{\leq \ell^{-r}}$ of weight not appearing in $\mathfrak{gl}_n(\mathbb{Q}_\ell)$.  Thus if $r$ is large enough that $\sigma$ acts on $\overline{\mathbb{Q}_\ell}[[\pi_1^{\text{\'et}, \ell}(X_{\bar k}, \bar x)]]^{\leq \ell^{-r}}$ diagonalizably, $\rho_r(W^{-N})=0$ for $N\gg 0$.
\begin{theorem}\label{general-monodromy}
Let $k$ be a field and $X/k$ a variety satisfying Assumption \ref{torsion-free-graded}.  Let $x\in X$ be a rational point or rational tangential basepoint, and suppose that $$\overline{\mathbb{Q}_\ell}[[\pi_1^{\text{\'et}, \ell}(X_{\bar k}, \bar x)]]^{\leq \ell^{-r}}$$ is spanned by $\sigma$-eigenvectors for some $r>0$, where $\sigma\in G_k$ is a quasi-Frobenius with respect to the weight filtration.  Then any arithmetic representation $$\rho: \pi_1^{\text{\'et}, \ell}(X_{\bar k}, \bar x)\to GL_n(\mathbb{Z}_\ell)$$ with $r_0(\rho)>r$ is unipotent, where $r_0(\rho)$ is defined as in Definition \ref{r-0-defn}.  If $\rho\otimes \mathbb{Q}_\ell$ is irreducible, then $\rho$ is trivial.
\end{theorem}
\begin{proof} 
By Lemma \ref{subquotient-lemma}, we may assume $\rho$ arises as a subquotient of $\beta|_{\pi_1^{\text{\'et}, \ell}(X_{\bar k}, \bar x)}$, for a representation $$\beta: \pi_1(X_{k'}, \bar x)\to GL_n(\mathbb{Z}_\ell)$$ satisfying $r_0(\beta|_{\pi_1^{\text{\'et}, \ell}(X_{\bar k}, \bar x)})\geq r_0(\rho)$, and where $k'/k$ is finite.  It suffices to show that $\beta|_{\pi_1^{\text{\'et}, \ell}(X_{\bar k}, \bar x)}\otimes\mathbb{Q}_\ell$ is unipotent.  Indeed, any subquotient of a unipotent representation is unipotent, and any irreducible unipotent representation is trivial.

By Proposition \ref{convergence-prop}, $\beta|_{\pi_1^{\text{\'et}, \ell}(X_{\bar k}, \bar x)}$ extends to a ring map $$\beta_r: \overline{\mathbb{Q}_\ell}[[\pi_1^{\text{\'et}, \ell}(X_{\bar k}, \bar x)]]^{\leq \ell^{-r}}\to \mathfrak{gl}_n(\overline{\mathbb{Q}_\ell}).$$  As $\beta$ extends to a representation of ${\pi_1^{\text{\'et}, \ell}(X_{k'}, \bar x)}$, there exists an action of $\sigma^m$ on $\mathfrak{gl}_n(\overline{\mathbb{Q}_\ell})$ for some $m$ so that $\beta_r$ is equivariant with respect to this action and the natural action of $\sigma^m$ on $\overline{\mathbb{Q}_\ell}[[\pi_1^{\text{\'et}, \ell}(X_{\bar k}, \bar x)]]^{\leq \ell^{-r}}$.

Let $q^{-1}$ be the degree of $\sigma$. Now $\mathfrak{gl}_n(\overline{\mathbb{Q}_\ell})$ splits into finitely many (generalized) eigenspaces under the action of $\sigma$, as $\mathfrak{gl}_n(\overline{\mathbb{Q}_\ell})$ is finite dimensional; in particular, for $N\gg0$, $q^N$ does not appear as a generalized eigenvalue of the $\sigma$-action on $\mathfrak{gl}_n(\overline{\mathbb{Q}_\ell})$.  But by Lemma \ref{i-adic-weights}, $\mathscr{I}^n$ is spanned by $\sigma$-eigenvectors of weight $\leq -n$. Hence $\beta_r$ must annihilate $\mathscr{I}^N$; in particular, $\beta$ must be unipotent (as for every $g\in \pi_1^{\text{\'et}, \ell}(X_{\bar k}, \bar x)$, we have that $g-1$ acts nilpotently).
\end{proof}
\subsection{Nilpotent Fundamental Groups}\label{nilpotent-pi1}
We begin by analyzing the possible geometric representations of nilpotent fundamental groups of algebraic varieties.  
\begin{theorem}\label{nilpotent-group-pro-l}
Let $k$ be a finitely generated field of characteristic different from $\ell$, and let $X/k$ be a smooth variety admitting a simple normal crossings compactification.  Let $x$ be a rational point or rational tangential basepoint of $X$, and $k\hookrightarrow \bar k$ an algebraic closure of $X$.  Suppose that $$\pi_1^{\text{\'et}, \ell}(X_{\bar k}, \bar x)$$ is torsion-free nilpotent.  Then any arithmetic representation $$\rho: \pi_1^{\text{\'et}, \ell}(X_{\bar k}, \bar x)\to GL_n(\mathbb{Z}_\ell)$$ is unipotent.  
\end{theorem}
\begin{proof}
First, note that $X$ as in the theorem satisfies Assumption \ref{torsion-free-graded}, so we may freely apply Theorem \ref{general-monodromy}.

Let $r_0(\rho)$ be as in Definition \ref{r-0-defn}; by Proposition \ref{r0-bound}, $r_0(\rho)>0$.  Choose any $r$ with $0<r<r_0(\rho)$.  Then by Theorem \ref{diagonalizing-nilpotent-pi1}, $\overline{\mathbb{Q}_\ell}[[\pi_1^{\text{\'et}, \ell}(X_{\bar k}, \bar x)]]^{\leq \ell^{-r}}$ admits a dense set of quasi-Frobenius eigenvectors, for any quasi-Frobenius.  But there are elements in $G_k$ which act via a quasi-Frobenius on $\overline{\mathbb{Q}_\ell}[[\pi_1^{\text{\'et}, \ell}(X_{\bar k}, \bar x)]]$ with respect to the weight filtration.  Indeed, there exists a sub-algebra $R\subset k$ with $k=\on{Frac}(R)$ so that $X$ has good reduction over $R$ (i.e. there exists a smooth model of $X$ over $R$ so that the simple normal crossings compactification of $X$ extends to a relative simple normal crossings compactification).  Then any Frobenius at a closed point of $\on{Spec}(R)$ acts via a quasi-Frobenius on $\overline{\mathbb{Q}_\ell}[[\pi_1^{\text{\'et}, \ell}(X_{\bar k}, \bar x)]]$.    We may now conclude the result by Theorem \ref{general-monodromy}.
\end{proof}
We now deduce some geometric consequences of this theorem.
\begin{theorem}\label{no-abelian-schemes}
Let $k$ be a field, and let $X/k$ be a smooth variety admitting a simple normal crossings compactification, and let $\eta(X)$ be the generic point of $X$.  Suppose $$\pi_1^{\on{tame}}(X_{\bar k}, \bar x)$$ is torsion-free nilpotent.  Let $A$ be an Abelian scheme over $X$, and suppose that the monodromy representation $$\rho: \pi_1(X_{\bar k}, \bar x)\to GL(T_\ell(A))$$ factors through $\pi_1^{\on{tame}}(X_{\bar k}, \bar x)$ for some $\ell$ different from the characteristic of $k$.  Then $A_{\eta(X)}$ is  isogenous to an isotrivial Abelian variety.
\end{theorem}
\begin{remark}
The condition that the monodromy representation be tame is automatic in characteristic zero or in characteristic $p$ if $p>2\dim(A/X)+1$.  It is a necessary condition---for example, there are non-isotrivial elliptic curves over $\mathbb{G}_m=\on{Spec}(k[t, t^{-1}])$, namely $$y^2=x^3+x^2-t$$ in characteristic $2$ and $$y^2+xy=x^3+t$$ in characteristic $3$.  These examples were suggested to the author by Noam Elkies.  Higher dimensional Abelian schemes over $\mathbb{G}_m$ (in small characteristic relative to their dimension) may be constructed as Jacobians of Drinfel'd modular curves.
\end{remark}
\begin{proof}[Proof of Theorem \ref{no-abelian-schemes}]
By passing to a finite (tame) cover $p: X'\to X$, we may assume that $\rho$ is residually trivial, and in particular factors through $\pi_1^{\text{\'et}, \ell}(X'_{\bar k}, \bar{x'})$.  As $p$ was tame, we have that $\pi_1^{\text{tame}}(X'_{\bar k}, \bar{x'})\subset \pi_1^{\on{tame}}(X_{\bar k}, \bar x)$ and hence $\pi_1^{\text{tame}}(X'_{\bar k}, \bar{x'})$ is torsion-free nilpotent.  Thus we may assume that $\rho$ factors through $\pi_1^{\text{\'et}, \ell}(X'_{\bar k}, \bar{x'})$.  Thus (renaming $X'$ as $X$) it suffices to show that if $\rho$ factors through $\pi_1^{\text{\'et}, \ell}(X_{\bar k}, \bar x)$, then $A_{\eta(X)}$ is inseparably isogenous to a \emph{constant} Abelian scheme (i.e.~one pulled back from $k$).  By the Lang-N\'eron theorem \cite{lang-neron}, it suffices to show that in this case the monodromy representation on $T_\ell(A)$ is unipotent.  In this case, the rational torsion subgroup of $A_{\eta(X)}$ is not finitely generated, so $A_{\eta(X)}$ has non-trivial trace; replacing $A_{\eta(X)}$ with the quotient by its trace, we are done by induction on the dimension of $A$

Now, there exists a finitely generated field $k'\subset k$ over which $X'$ and $A$ are defined, and so we are done by Theorem \ref{nilpotent-group-pro-l}, 
\end{proof}
\begin{remark}
We view Theorems \ref{nilpotent-group-pro-l} and \ref{no-abelian-schemes} as being analogous to the hyperbolicity of period domains.  This hyperbolicity rules out e.g.~maps from $\mathbb{C}^*$ or Abelian varieties to period domains---similarly, Theorem \ref{nilpotent-group-pro-l} shows that there are \emph{representation-theoretic} obstructions to the existence of such maps in the case that such maps arise from algebraic geometry; Theorem \ref{no-abelian-schemes} is a precise version of this for period domains arising as moduli of Abelian schemes with level structure.
\end{remark}
\begin{remark}
One may prove Theorem \ref{nilpotent-group-pro-l} by other means---for example, by using Deligne's theorem on semisimplicity of geometric monodromy (though we were unable to find Theorem \ref{nilpotent-group-pro-l} in the literature).  However, as we will see soon, the method used here generalizes to varieties with more interesting fundamental group.
\end{remark}
The proof also shows
\begin{theorem}
Let $k$ be a field, and let $X/k$ be a smooth variety admitting a simple normal crossings compactification.  Suppose $$\pi_1^{\text{tame}}(X_{\bar k}, \bar x)$$ is torsion-free nilpotent.  Let $$\rho: \pi_1^{\text{tame}}(X_{\bar k}, \bar x)\to GL_n(\mathbb{Z}_\ell)$$ be any arithmetic representation, where $\ell$ is different from the characteristic of $k$.  Then $\rho$ is quasi-unipotent (i.e.~its restriction to some finite index subgroup of $\pi_1^{\text{tame}}$ is unipotent).
\end{theorem}
\subsection{Restrictions on monodromy in general}
We now turn to results for arbitrary normal varieties in characteristic zero.  Our main result is:
\begin{theorem}\label{main-arithmetic-result}
Let $k$ be a finitely generated field of characteristic zero, and $X/k$ a normal variety.  Let $\ell$ be a prime.  Then there exists $$N=N(X,\ell)$$ such that any arithmetic representation $$\rho: \pi_1(X_{\bar k})\to GL_n(\mathbb{Z}_\ell)$$ which is trivial mod $\ell^N$ is unipotent.
\end{theorem}
\begin{remark}
We have suppressed the choice of basepoint in this theorem because the notion of arithmeticity is independent of basepoint --- we will choose a basepoint in the course of the proof.
\end{remark}
\begin{proof}
By \cite{deligne2}, there exists a curve $C/k$ and a map $C\to X$ such that $$\pi_1^{\text{\'et}}(C_{\bar k}, \bar x)\to \pi_1^{\text{\'et}}(X_{\bar k}, \bar x)$$ is surjective for any geometric point $\bar x$ of $C$.  Thus we may immediately replace $X$ with $C$; observe that $\pi_1(C_{\bar k}, \bar x)$ satisfies Assumption \ref{torsion-free-graded}.  After replacing $k$ with a finite extension $k'$, we may assume $\bar x$ comes from a rational point of $C$.  Note that the choice of $C$ and $k'$ does not depend on $\rho$.  Thus we replace $k$ with $k'$ and $X$ with $C$.

By Corollary \ref{pi1-quasi-scalar}, there exists $\sigma \in G_k$ such that $\sigma$ acts on $\mathbb{Q}_\ell[[\pi_1^{\text{\'et}, \ell}(C_{\bar k}, \bar x)]]$ via a quasi-scalar with respect to the weight filtration.  Then by Theorem \ref{diagonalizing-convergent-rings-mixed}, there exists $r>0$ such that $\sigma$ acts diagonalizably on $\mathbb{Q}_\ell[[\pi_1^{\text{\'et}, \ell}(C_{\bar k}, \bar x)]]^{\leq \ell^{-r}}$.  Set $N$ to be the least integer greater than $r$; note again that $N$ does not depend on $\rho$.  Now if $r_0(\rho)\geq N$, (which it is if $\rho$ is trivial mod $\ell^N$, using Proposition \ref{r0-bound}), we may conclude the theorem by Theorem \ref{general-monodromy}.
\end{proof}
\begin{remark}
Observe that $N$ only depends on $X$ and $\ell$, and not on any of the parameters of $\rho$ (for example, its dimension).  To our knowledge this was not expected.  
\end{remark}
We may now prove Theorem \ref{main-theorem-intro}, from the introduction.

Recall that we say a representation $\rho$ of $\pi_1(X)$ \emph{arises from geometry} if there exists a proper map $\pi: Y\to X$ so that $\rho$ arises as the monodromy representation of $\pi_1$ on a locally constant (resp.~lisse) subquotient of $R^i\pi_*\underline{\mathbb{Z}}$ (resp.~$R^i\pi_*\underline{\mathbb{Z}_\ell}$) for some $i$, where the pushforward is taken with respect to the analytic (resp.~pro-\'etale) topology.
\begin{corollary}\label{geometric-monodromy-corollary}
Let $X/\mathbb{C}$ be a normal variety, and $\ell$ a prime.  Then there exists $N=N(X,\ell)$ so that any representation $$\pi_1(X(\mathbb{C})^{\text{an}})\to GL_n(\mathbb{Z})$$ which arises from geometry, and is trivial modulo $\ell^N$, is unipotent.
\end{corollary}
\begin{proof}
While the corollary is stated in the analytic topology, we may by standard comparison results immediately replace our representation $\rho: \pi_1(X)\to GL_n(\mathbb{Z})$ with the associated continuous $\ell$-adic representation of $\pi_1^{\text{\'et}}(X)$.

Let $k\subset \mathbb{C}$ be a finitely generated subfield over which $X$ is defined; suppose that $\rho$ arises as a subquotient of the monodromy representation on $R^i\pi_*\underline{\mathbb{Z}_\ell}$ for some proper morphism $\pi: Y\to X$.  Then there exists a finitely generated $k$-algebra $R\subset \mathbb{C}$, a proper $R$-scheme $\mathcal{Y}$, and a map of $R$-schemes $$\tilde \pi: \mathcal{Y}\to X_R$$ such that $\pi$ is obtained by base change along the inclusion $R\subset \mathbb{C}$.  $R^i\tilde{\pi}_*\underline{\mathbb{Z}_\ell}$ is constructible; hence by shrinking $R$ and quotienting out torsion we may replace it by a lisse subquotient.  Now specializing to any closed point of $\on{Spec}(R)$, we see that $\rho|_{\pi_1(X_{\bar k})}$ is arithmetic, so we are done by Theorem \ref{main-arithmetic-result}.

Observe that the $N$ from Theorem \ref{main-arithmetic-result} only depends on $X$ and $k$, which do not depend on $\rho$ in any way, so we have the desired uniformity in $\rho$.
\end{proof}

In some cases, one can make the $N$ appearing in Theorem \ref{main-arithmetic-result} explicit.  Indeed, $N$ only depends on the index of the representation $$G_k\to GL(H^1(C_{\bar k}, \mathbb{Z}_\ell))$$ in its Zariski closure (see Remarks \ref{bound-observation-1} and \ref{bound-observation-2}).  Thus we now turn to cases where we understand this representation.

Our main example will be the case $$X=\mathbb{P}^1_k\setminus \{x_1, \cdots, x_n\},$$ where $x_1,\cdots, x_n\in \mathbb{P}^1(k)$.  But our results will apply equally well to any smooth variety satisfying Assumption \ref{torsion-free-graded} with a simply connected simple normal crossings compactification, or indeed to any smooth variety $X$ with simple normal crossings $\overline{X}$ such that $H^1(\overline{X}_{\bar k}, \mathbb{Q}_\ell)=0$, and which satisfies Assumption \ref{torsion-free-graded}.  Unlike the previous result \ref{main-arithmetic-result}, this theorem works in arbitrary characteristic.

\begin{theorem}\label{p1-main-theorem}
Let $k$ be a finitely generated field with prime subfield $k_0$, and let $$X=\mathbb{P}^1_{\bar k}\setminus\{x_1, \cdots, x_n\},$$ for $x_1, \cdots, x_n\in \mathbb{P}^1(k)$.     Let $\ell$ be prime different from the characteristic of $k$ and $q\in \mathbb{Z}_\ell^\times$ any element of the  image of the cyclotomic character $$\chi: \on{Gal}(\overline{K}/K)\to \mathbb{Z}_\ell^\times;$$ let $s$ be the order of $q$ in $\mathbb{F}_\ell^\times$ if $\ell\not=2$ and in $(\mathbb{Z}/4\mathbb{Z})^\times$ if $\ell=2$.  Let $\epsilon=1$ if $\ell=2$ and $0$ otherwise.  For any basepoint $x$ of $X$, let $$\rho: \pi_1^{\text{\'et}, \ell}(X, x)\to GL_m(\mathbb{Z}_\ell)$$ be an arithmetic representation with $$r_0(\rho)>\frac{1}{s}\left(v_\ell(q^s-1)+\frac{1}{\ell-1}+\epsilon\right).$$  Then  $\rho$ is unipotent. 
\end{theorem}

\begin{remark}\label{dumb-bounds-remark}
For any $X$, the bound above becomes $$r_0(\rho)>\frac{1}{\ell-1}+\frac{1}{(\ell-1)^2}$$ for all $\ell\gg 0$ (indeed, for any $\ell>2$ such that the cyclotomic character $$\chi: \on{Gal}(\overline{K}/K)\to \mathbb{Z}_\ell^\times$$ is surjective).  In general, it is uniform in the degree of $K/\mathbb{Q}$ if $K$ is a number field.  Thus in particular if $\chi$ is surjective at $\ell$ for some odd $\ell$, then any arithmetic representation of $\pi_1(X_{\bar k})$ as above which is trivial mod $\ell$, is unipotent.
\end{remark}
\begin{proof}[Proof of Theorem \ref{p1-main-theorem}]  
We first choose a rational (or rational tangential) basepoint $x$ of $X$ over $k$.

In this case, $$H^1(X, \mathbb{Z}_\ell)=\mathbb{Z}_\ell(1)^{n-1},$$ so a generic element $\sigma\in G_k$ acts via a quasi-scalar on $\mathbb{Q}_\ell[[\pi_1^{\text{\'et},\ell}(X)]]$ with respect to the weight filtration (which equals the $\mathscr{I}$-adic filtration).  In particular, there is a $\sigma\in G_k$ with $\chi(\sigma)=q$; hence $\sigma$ acts via a quasi-scalar of degree $q^{-1}$.  Thus by Theorem \ref{diagonalizing-convergent-rings-pure}, we have that for $$r=\frac{1}{s}\left(v_\ell(q^s-1)+\frac{1}{\ell-1}+\epsilon\right),$$ $\sigma$ acts diagonalizably on $$\mathbb{Q}_\ell[[\pi_1(X, \bar x)]]^{\leq \ell^{-r}}.$$  Now the proof is complete by Theorem \ref{general-monodromy}.
\end{proof}
We now discuss some corollaries of Theorem \ref{p1-main-theorem}.  

First, we may deduce Theorem \ref{main-theorem-P1} from the introduction.
\begin{proof}[Proof of Theorem \ref{main-theorem-P1}]
We let $X=\mathbb{P}^1_{k}\setminus\{x_1, \cdots, x_n\}$.  Then $X$ admits a natural model over the field $$K=k_0\left(\frac{x_a-x_b}{x_c-x_d}\right)_{1\leq a<b<c<d\leq n}.$$ Now the result follows from Theorem \ref{p1-main-theorem} in a manner identical to the proof of Corollary \ref{geometric-monodromy-corollary}.
\end{proof}

\begin{corollary}\label{easy-bound-corollary}
Let $X=\mathbb{P}^1_{\mathbb{C}}\setminus\{x_1, \cdots, x_n\}$, where the cross-ratios of the $x_i$ lie in $\mathbb{Q}$.  Let $x$ be a $\mathbb{C}$-point of $X$. Let $\rho: \pi_1(X^{\on{an}}, x)\to GL_m(\mathbb{Z})$ be a representation of geometric origin, and suppose that $\rho$ is trivial mod $\ell$ for some prime $\ell>2$, or that $\rho$ is trivial mod $8$.  Then $\rho\otimes \mathbb{Q}$ is unipotent.  
\end{corollary}
\begin{proof}
Tensoring $\rho$ with $\mathbb{Z}_\ell$, we obtain a morphism $$\rho\otimes \mathbb{Z}_\ell: \pi_1(X^{\on{an}}, x)\to \widehat{\pi_1(X^{\on{an}}, x)}\simeq \pi_1^{\text{\'et}}(X, x)\to GL_n(\mathbb{Z}_\ell)$$ where $\widehat{\pi_1(X^{\on{an}}, x)}$ is the profinite completion of ${\pi_1(X^{\on{an}}, x)}$.  If $\rho$ is trivial mod $\ell$, the map factors through $$\ker(GL_m(\mathbb{Z}_\ell)\to GL_m(\mathbb{F}_\ell)),$$ which is a pro-$\ell$ group.  Hence we obtain a map $$\rho': \pi_1^{\text{\'et}, \ell}(X, x)\to GL_m(\mathbb{Z}_\ell).$$

By Proposition \ref{r0-bound}, if $\rho'$  is trivial mod $\ell$ then $r_0(\rho')\geq 1,$ and if $\ell=2$ and $\rho'$ is trivial mod $8$ then $r_0(\rho')\geq 3$.  Now the result follows from Theorem \ref{main-theorem-P1} and Remark \ref{dumb-bounds-remark}. Indeed, for $\ell>2$, $$1>\frac{1}{\ell-1}+\frac{1}{(\ell-1)^2},$$ and for $\ell=2$, $$3>\frac{1}{2-1}+\frac{1}{(2-1)^2}=2.$$
\end{proof}
\begin{remark}
One may also use identical methods to prove an analogous statement for any variety $X$ such that $$H^1(X_{\bar k}, \mathbb{Z}_\ell)=\mathbb{Z}_\ell(-1)^{\oplus N}$$ and satisfying Assumption \ref{torsion-free-graded}.  For example, one may take $X$ to be the complement of certain hyperplane arrangements in $\mathbb{P}^2$.
\end{remark}
We also have the following amusing corollary, which may also be proven using the quasi-unipotent local monodromy theorem.
\begin{corollary}
Let $D\subset \mathbb{P}^1_{\mathbb{C}}$ be finite, and let $$X=\mathbb{P}^1_{\mathbb{C}}\setminus D.$$  Let $$\rho: \pi_1(X)\to GL_n(\mathbb{Z})$$ be a representation arising from geometry which is not unipotent.  If $\rho$ is trivial mod $\ell$ for any prime $\ell>2$, then $|D|>3$.
\end{corollary}
\begin{proof}
If $D\leq 3$, there are no non-trivial cross-ratios between elements of $D$.  Now the result follows from Corollary \ref{easy-bound-corollary}.
\end{proof}
We give one final corollary over arbitrary fields.  Observe that if there is a map $\mathbb{P}^1\setminus\{x_1, \cdots, x_n\}\to X$ which induces a surjection on geometric pro-$\ell$ fundamental groups, we may apply Theorem \ref{main-theorem-P1} to find restrictions on monodromy representations of the fundamental group of $X$.  Such a map exists if $X$ is an open subvariety of a separably rationally connected variety, by the main result of \cite{kollar}.  Thus we have
\begin{corollary}
Let $k$ be a field an $X$ an open subset of a separably rationally connected $k$-variety.  Let $\ell$ be a prime different from the characteristic of $k$.  Then there exists $N=N(X,\ell)$ such that if $$\rho:\pi_1^{\text{\'et},\ell}(X_{\bar k})\to GL_m(\mathbb{Z}_\ell)$$ is a representation which arises from geometry, and is trivial mod $\ell^N$, then $\rho$ is unipotent.
\end{corollary}
\subsection{Sharpness of results and lower bounds on periods}\label{sharp-section}
We expect that most of the bounds on integral $\ell$-adic periods (or equivalently, bounds on $r$ such that $$\overline{\mathbb{Q}_\ell}[[\pi_1^{\text{\'et},\ell}(X_{\bar k}, \bar x)]]^{\leq \ell^{-r}}$$ is spanned by $\sigma$-eigenvectors for some quasi-scalar $\sigma$) are far from sharp.  In positive characteristic, a positive answer to the question posed in Remark \ref{conjecture-bound} would imply that for any $X$ over a finite field satisfying Assumption \ref{torsion-free-graded}, one has that there exists $r>0$ such that $$\overline{\mathbb{Q}_\ell}[[\pi_1^{\text{\'et},\ell}(X_{\bar k}, \bar x)]]^{\leq \ell^{-r}}$$ is spanned by Frobenius eigenvectors.  If this conjecture holds for a given $X$, the methods of this section would then allow us to place restrictions on geometric monodromy representations of $\pi_1(X)$ which come from geometry, in positive characteristic.

Even in characteristic zero, we expect our bounds are not in general sharp --- see for example Corollary \ref{trivial-corollary} and the remark preceding it.

On the other hand, our results are sharp in some cases---indeed, the contrapositives of results in the last section allow us to find \emph{lower bounds} on integral $\ell$-adic periods by finding non-unipotent geometric monodromy representations which are trivial mod $\ell$.

\begin{theorem}\label{periods-lower-bound}
Let $\ell$ be a prime and $k$  a finite field of characteristic different from $\ell$.  Let $Y(\ell)$ be the modular curve over $k$ parametrizing elliptic curves with full level $\ell$ structure.  Then $$\overline{\mathbb{Q}_\ell}[[\pi_1^{\text{\'et}, \ell}(Y(\ell)_{\bar k}, \bar x)]]^{\leq \ell^{-r}}$$ is not spanned by Frobenius eigenvectors for any $r<1$.
\end{theorem}
\begin{remark}
Equivalently, the integral $\ell$-adic periods $b^i_n$ associated to the weight filtration on the $\ell$-adic fundamental group of $Y(\ell)$ (as in Section \ref{integral-l-adic-periods-section}) grow faster than $n$.
\end{remark}
\begin{proof}[Proof of Theorem \ref{periods-lower-bound}]
The Tate module of the universal family over $Y(\ell)$ provides an arithmetic monodromy representation $\rho$ which is not unipotent, but is trivial mod $\ell$.  But $r_0(\rho)\geq 1$ by Lemma \ref{r0-bound}, so if $$\overline{\mathbb{Q}_\ell}[[\pi_1^{\text{\'et}, \ell}(X(\ell)_{\bar k}, \bar x)]]^{\leq \ell^{-r}}$$ was spanned by Frobenius eigenvectors for some $r<1$, this would contradict Theorem \ref{general-monodromy}.
\end{proof}
In particular, this shows that Theorem \ref{p1-main-theorem} is reasonably sharp for $\ell\leq 5$.  Indeed, for $\ell\leq 5$, $Y(\ell)$ has genus zero, but nonetheless admits interesting geometric monodromy representations which are trivial mod $\ell$.  This does not contradict Corollary \ref{easy-bound-corollary} because the cusps of $Y(\ell)$ have cross-ratios which generate $\mathbb{Q}(\zeta_\ell)$.

One may generate more examples from the list of Shimura curves of genus zero in \cite{voight}.  

The finiteness of the list of genus zero modular curves suggests the following question, which is a strong version of the question posed in Remark \ref{conjecture-bound}, and is also suggested by the geometric torsion conjecture:
\begin{question}
Let $X$ be a smooth curve over a finite field $k$, and let $\ell$ be a prime different from the characteristic of $k$.  Does there exist an $r=r(\on{gonality}(X), \ell)>0$, depending only on the gonality of $X$ and on $\ell$, such that $$\overline{\mathbb{Q}_\ell}[[\pi_1^{\text{\'et}, \ell}(X_{\bar k}, \bar x)]]^{\leq \ell^{-r}}$$ is spanned by Frobenius eigenvectors?  If so, does $r$ tend to zero rapidly with $\ell$?
\end{question}
A positive answer to this question would imply a positive answer to a ``pro-$\ell$" version of the geometric torsion conjecture, by the methods of this section.

\section{Appendix: Hopf groupoids}\label{hopf-groupoids}
This appendix is meant to collect the various definitions and facts required about Hopf groupoids, which are certain linear algebraic objects associated to groupoids, just as group algebras are associated to groups.  See \cite[Part I(c), Chapter 9]{fresse2015homotopy} for another presentation of this topic.
\begin{defn}\label{hopf-groupoid-definition}
Let $R$ be a commutative ring.  An $R$-Hopf groupoid $\mathcal{G}=(\on{Ob}(\mathcal{G}), \mathcal{G}(a, b), \nabla, \Delta, S, \epsilon, \eta),$ is the data of
\begin{itemize}
\item  a set $\on{Ob}(\mathcal{G})$,
\item for each $a, b\in \on{Ob}(\mathcal{G})$, an $R$-module $\mathcal{G}(a,b)$,
\item associative composition maps $\nabla: \mathcal{G}(a,b)\otimes_R \mathcal{G}(b, c)\to \mathcal{G}(a,c)$,
\item for each $a\in \on{Ob}(\mathcal{G})$, unit maps $\eta: R\to \mathcal{G}(a,a)$ compatible with $\nabla$ (i.e. turning $\mathcal{G}(a,a)$ into an associative $R$-algebra), 
\item for each $a, b\in \on{Ob}(\mathcal{G}(a,b)),$ counit maps $\epsilon: \mathcal{G}(a,b)\to R$, such that if $a=b$, $\epsilon\circ \eta=\on{id}_R$,
\item antipode maps $S: \mathcal{G}(a,b)\to \mathcal{G}(b,a)$ such that $S\circ S=\on{id}$,
\item coassociative comultiplication maps $\Delta: \mathcal{G}(a,b)\to \mathcal{G}(a,b)\otimes \mathcal{G}(a, b)$ compatible with $\Delta$ (i.e. turning each $\mathcal{G}(a,b)$ into an $R$-coalgebra), and compatible with the unit in the sense that each $\mathcal{G}(a,a)$ is a bialgebra.
\end{itemize}
such that the following diagrams commute:
 $$\xymatrix{
 \mathcal{G}(a, b)\otimes \mathcal{G}(a, b) \ar[rr]^{S\otimes \text{id}}& & \mathcal{G}(b, a)\otimes \mathcal{G}(a, b) \ar[d]^\nabla\\
\mathcal{G}(a, b) \ar[u]^\Delta \ar[r]^\epsilon & R \ar[r]^\eta & \mathcal{G}(b, b)
}$$
 $$\xymatrix{
 \mathcal{G}(a, b)\otimes \mathcal{G}(a, b) \ar[rr]^{\text{id}\otimes S}& & \mathcal{G}(a, b)\otimes \mathcal{G}(b, a) \ar[d]^\nabla\\
\mathcal{G}(a, b) \ar[u]^\Delta \ar[r]^\epsilon & R \ar[r]^\eta & \mathcal{G}(a, a)
}$$
$$\xymatrix{
\mathcal{G}(a,b)\otimes \mathcal{G}(b,c) \ar[r]^-\nabla \ar[d]_{\epsilon\otimes \epsilon} & \mathcal{G}(a, c) \ar[ld]^\epsilon\\
R &
}$$
$$\xymatrix{
\mathcal{G}(a,b)\otimes \mathcal{G}(b,c) \ar[d]_{\Delta\otimes \Delta} \ar[r]^\nabla & \mathcal{G}(a, c) \ar[r]^{\Delta} & \mathcal{G}(a,c)\otimes \mathcal{G}(a,c)\\
\mathcal{G}(a,b)\otimes \mathcal{G}(a,b)\otimes \mathcal{G}(b,c)\otimes \mathcal{G}(b,c) \ar[rr]^-{\on{id}\otimes \tau\otimes \on{id}} & &\mathcal{G}(a,b)\otimes \mathcal{G}(b,c)\otimes \mathcal{G}(a,b)\otimes \mathcal{G}(b,c) \ar[u]^{\nabla\otimes\nabla}
}$$
where $\tau: \mathcal{G}(a,b)\otimes \mathcal{G}(b,c)\to \mathcal{G}(b,c)\otimes\mathcal{G}(a,b)$ is the map $x\otimes y\mapsto y\otimes x$.
\end{defn}
If $H$ is a groupoid with set of objects $\on{Ob}(H)$ and arrows $H(a,b)$ for $a, b\in \on{Ob}(H)$, the $R$-Hopf groupoid $\mathcal{G}(H)$ associated to $H$ is given by $$\text{Ob}(\mathcal{G}(H))=\on{Ob}(H)$$ and $$\mathcal{G}(H)(a,b)=R[H(a,b)].$$  The composition maps are defined by linearly extending the composition law of $H$; the unit maps are given by sending $1\in R$ to $\on{id}$ in $\mathcal{G}(H)(a,a)$; and the counit maps are given by sending $h\in H(a,b)$ to $1\in R$.  The antipode is constructed by extending the map $h\mapsto h^{-1}$ linearly, and the comultiplication arises by extending $h\mapsto h\otimes h$ linearly.  Note that such Hopf groupoids are cocommutative, in the sense that each $(\mathcal{G}(H)(a,b), \Delta, \eta)$ is a cocommutative coalgebra.
\begin{defn}\label{grouplike}
Let $\mathcal{G}$ be an $R$-Hopf groupoid.  We say an element $g\in \mathcal{G}(a,b)$ is \emph{grouplike} if $\Delta(g)=g\otimes g$.
\end{defn}
\begin{prop}
Let $g\in \mathcal{G}(a,b)$ be a grouplike element.  Then $g$ is invertible.
\end{prop}
\begin{proof}
$S(g)$ is inverse to $g$.
\end{proof}
\begin{defn}
Suppose $\mathcal{G}$ is an $R$-Hopf groupoid.  An ideal in $\mathcal{G}$ is an $R$-submodule $\mathscr{I}(a,b)\subset \mathcal{G}(a,b)$ for each $a,b$ such that 
\begin{itemize}
\item $\mathscr{I}(a,b)$ is a coideal of the coalgebra $\mathcal{G}(a,b)$ for each $a,b$,
\item $S(\mathscr{I}(a,b))=\mathscr{I}(b,a)$,
\item for each $a,b,c$, $$\nabla(\mathscr{I}(a,b)\otimes \mathcal{G}(b,c))\subset \mathscr{I}(a,c)$$ and $$\nabla(\mathcal{G}(a,b)\otimes \mathscr{I}(b,c))\subset \mathscr{I}(a,c).$$
\end{itemize}
\end{defn}
We will mostly consider the ``augmentation ideal" $\mathscr{I}=\on{ker}(\epsilon)$.
\begin{lemma}\label{augmentation-module}
Let $\mathcal{G}$ be a Hopf groupoid with augmentation ideal $\mathscr{I}=\{\mathscr{I}(a,b)\}$.  Suppose that each $\mathcal{G}(a, b)$ contains a grouplike element. Then for each $a, b$, $$\mathscr{I}(a,b)=\mathscr{I}(a,a)\mathcal{G}(a, b)=\mathcal{G}(a,b)\mathscr{I}(b,b).$$
\end{lemma}
\begin{proof}
It is clear that $$\mathscr{I}(a,a)\mathcal{G}(a, b)\subset \mathscr{I}(a,b) \text{ and } \mathcal{G}(a,b)\mathscr{I}(b,b)\subset \mathscr{I}(a,b).$$
For the converse inclusions, let $g\in \mathcal{G}(a,b)$ be a grouplike element; we denote its inverse $S(g)$ by $g^{-1}$.  Then for any $x\in \mathscr{I}(a, b)$, $x\circ g^{-1}\in \mathscr{I}(a,a)$, and $x=(x\circ g^{-1})\circ g\in \mathscr{I}(a,a)\mathcal{G}(a, b)$.  Analogously, $x=g\circ (g^{-1}\circ x)\in \mathcal{G}(a, b)\mathscr{I}(b,b).$
\end{proof}
\begin{corollary}\label{I-adic-filtration}
Let $\mathcal{G}=\{\mathcal{G}(a, b)\}$ be an $R$-Hopf groupoid such that each non-empty $\mathcal{G}(a,b)$ contains a grouplike element.  Let $\mathscr{I}$ be the augmentation ideal, and let $\mathscr{I}^n$ be the ideal generated by $x\circ y$ where $x\in \mathscr{I}^{n-1}(a, c)$ and $y\in \mathscr{I}^{n}(c, b)$.  Then the filtration $\{\mathscr{I}^n(a,a)\}$ of $\mathcal{G}(a, a)$ is the $\mathscr{I}$-adic filtration of $\mathcal{G}(a,a)$ viewed as a ring; the filtration $\{\mathscr{I}^n(a,b)\}$ of $\mathcal{G}(a,b)$ is the $\mathscr{I}(a,a)$-adic filtration (resp.~$\mathscr{I}(b,b)$-adic filtration) obtained by viewing $\mathcal{G}(a,b)$ as a left $\mathcal{G}(a,a)$-module (resp.~right $\mathcal{G}(b,b)$-module). 
\end{corollary}
\begin{proof}
The assertion about $\mathcal{G}(a,a)$ follows from the assertion for $\mathcal{G}(a,b)$; we check the assertion comparing the filtration $\{\mathscr{I}^n(a,b)\}$ to the $\mathscr{I}(a,a)$-adic filtration, as the other statement follows identically.  

Lemma \ref{augmentation-module} shows that $\mathscr{I}(a,b)=\mathscr{I}(a,a)\mathcal{G}(a,b)$.  In general, we may assume by induction that $\mathscr{I}^{n-1}(a,c)=\mathscr{I}(a,a)^{n-1}\mathcal{G}(a, c)$ for all $c$.  Then we have for any $x\in \mathscr{I}^{n-1}(a,c), y\in \mathscr{I}(c, b)$ that $$x\circ y \in \mathscr{I}(a,a)^{n-1}\mathcal{G}(a,c)\mathscr{I}(c,c)\mathcal{G}(c,b).$$  Applying Lemma \ref{augmentation-module} again, we see that $x\circ y\in \mathscr{I}(a,a)^n\mathcal{G}(a,c)\mathcal{G}(c,b)$ as desired.
\end{proof}
Thus when we refer to the $\mathscr{I}$-adic filtration on $\mathcal{G}(a, b)$, there is no ambiguity as to what is meant.
\begin{lemma}\label{augmentation-abelianization-groupoid}
Let $\mathcal{G}$ be a groupoid and $R$ a commutative ring.  Let $\mathscr{I}$ be the augmentation ideal of the $R$-Hopf groupoid $\mathcal{G}(H)$.  Then for each $a, b$ with $\mathcal{G}(a,b)$ non-empty, there is are natural isomorphisms $$\mathscr{I}(a,b)/\mathscr{I}^2(a,b)\simeq \mathcal{G}(a,a)^{\text{ab}}\otimes R\simeq \mathcal{G}(b,b)^{\text{ab}}\otimes R.$$
\end{lemma}
\begin{proof}
The case $a=b$ is follows from a proof identical to that in Lemma \ref{augmentation-abelianization-group-ring}.  For $a\not=b$, let $g\in \mathcal{G}(a,b)$ be any element.  Then the maps $$\mathscr{I}(a,b)/\mathscr{I}^2(a,b)\to \mathscr{I}(a,a)/\mathscr{I}^2(a,a)\simeq \mathcal{G}(a,a)^{\text{ab}}\otimes R$$ $$\mathscr{I}(a,b)/\mathscr{I}^2(a,b)\to \mathscr{I}(b,b)/\mathscr{I}^2(b,b)\simeq \mathcal{G}(b,b)^{\text{ab}}\otimes R$$ induced by composition with $g$ (resp.~$g^{-1}$) are isomorphisms with inverse given by composition with $g^{-1}$ (resp.~$g$).  

If $h\in \mathcal{G}(a,b)$ is any other element, then $g-h\in \mathscr{I}(a,b)$, so the map induced by composition with $g-h$ is identically zero.  Thus the maps above do not depend on the choice of $g$.
\end{proof}
We also require the following definition:
\begin{defn}\label{hopf-groupoid-reps}
Let $\mathcal{G}$ be an $R$-Hopf groupoid.  A representation of $\mathcal{G}$ is a choice of $R$-module $M_a$ for each object $a$ of $\mathcal{G}$, and an $R$-linear map $$\mathcal{G}(a,b)\to \on{Hom}_{R\text{-mod}}(M_a, M_b)$$ for each $a, b$ compatible with composition.  
\end{defn}
As in the case of Hopf algebras, $R$-Hopf groupoids have the property that if $\{M_a\}, \{N_a\}$ are representations of $\mathcal{G}$, then $\{M_a\otimes_R N_a\}$ naturally has the structure of a representation.  Indeed, if $\mathcal{G}$ had only been a category enriched over $R$-modules, $\{M_a\otimes_R N_a\}$ would naturally have the structure of a $\mathcal{G}\otimes \mathcal{G}$-representation (where here $(\mathcal{G}\otimes \mathcal{G})(a, b)=\mathcal{G}(a,b)\otimes \mathcal{G}(a,b)$).  But the coproduct gives us a morphism $\mathcal{G}\to \mathcal{G}\otimes\mathcal{G}$.

\bibliographystyle{alpha}
\bibliography{monodromy-bibtex}

\end{document}